%% file: main.tex
\documentclass[11pt]{amsart}
 \pdfoutput=1
\usepackage[utf8]{inputenc}
\usepackage[textheight=615pt, textwidth=360pt]{geometry}
\usepackage[dvipdfm]{graphicx} 
\usepackage{bmpsize}
\usepackage{amsmath,amssymb, mathrsfs}
\usepackage{bbold}
\usepackage{centernot}
\usepackage{cite}
\usepackage[dvipsnames]{xcolor}

 \usepackage{mathrsfs}
 \usepackage{enumerate}
 
 \usepackage{cancel}

\newcommand{\R}{\mathbb{R}}
\newcommand{\N}{\mathbb{N}}
\newcommand{\Z}{\mathbb{Z}}
\newcommand{\C}{\mathbb{C}}
\newcommand{\de}{\partial}
\renewcommand{\-}{\smallsetminus}
\newcommand{\D}{\mathbb{D}}

\renewcommand{\a}{\alpha}

\newcommand{\f}{\varphi}
\newcommand{\e}{\varepsilon}
\newcommand{\w}{{\omega}}

\newcommand{\gap}{r}

\newcommand{\h}{\theta}
\newcommand{\spi}[1]{\mathcal{T}^{\otimes #1}}
\newcommand{\dd}{\eth}
\newcommand{\ddb}{\overline{\eth}}

\newcommand{\us}{{\underline{s}}}
\newcommand{\Es}{{\mathcal{E}^{\us}}}
\newcommand{\Esn}{{\mathcal{E}^{{\us}_n}}}
\newcommand{\usig}{\underline{\sigma}}
\newcommand{\usign}{\underline{\sigma_n}}
\newcommand{\usignp}{\underline{\sigma_n'}}
\newcommand{\uXn}{\underline{X_n}}
\newcommand{\ukn}{\underline{k_n}}
\newcommand{\ubeta}{\underline{\beta}}
\newcommand{\uxin}{\underline{\xi_n}}
\newcommand{\uxinf}{\underline{\xi_\infty}}

\newcommand{\transv}{\mathrel{\text{\tpitchfork}}}
\makeatletter
\newcommand{\tpitchfork}{%
  \raise-0.1ex\vbox{
    \baselineskip\z@skip
    \lineskip-.52ex
    \lineskiplimit\maxdimen
    \m@th
    \ialign{##\crcr\hidewidth\smash{$-$}\hidewidth\crcr$\pitchfork$\crcr}
  }%
}
\makeatother

\newcommand{\vol}{\mathrm{vol}}

\renewcommand{\P}{\mathbb{P}}
\newcommand{\E}{\mathbb{E}}
\newcommand{\nrw}{\Rightarrow}
\newcommand{\spt}{\text{supp}}
\newcommand{\mC}{\mathcal{C}}

\newcommand{\crit}{\mathrm{Crit}}

\usepackage{hyperref}
\hypersetup{
	colorlinks=true,       
	linkcolor=blue,         
	citecolor=blue}

 \usepackage[usenames,dvipsnames]{pstricks}
 \usepackage{pst-grad} 
 \usepackage{pst-plot} 

\newtheorem{thm}{Theorem}
\newtheorem{lemma}[thm]{Lemma}
\newtheorem{cor}[thm]{Corollary}
\newtheorem{prop}[thm]{Proposition}

\theoremstyle{definition}
\newtheorem{ass}[thm]{Assumption}
\newtheorem{defi}[thm]{Definition}
\newtheorem{remark}[thm]{Remark}

\newtheorem{example}[thm]{Example}
\newcommand{\be}{\begin{equation}}
\newcommand{\ee}{\end{equation}}
\newcommand{\bega}{\begin{equation}\begin{aligned}}
\newcommand{\eega}{\end{aligned}\end{equation}}

\usepackage{mathtools} 
\mathtoolsset{showonlyrefs,showmanualtags} 

\numberwithin{equation}{section}

\title{Geometry and topology of spin random fields\\
\medskip
July 14, 2022}
\author{A. Lerario, D. Marinucci, M. Rossi, M. Stecconi}
\thanks{DM acknowledges the MIUR Excellence Department Project awarded to the Department of Mathematics, University of Rome “Tor Vergata”, CUP E83C18000100006. The research of MR has been supported by the ANR-17-CE40-0008 Project UNIRANDOM. 
MS is supported by the grant TROPICOUNT of Région Pays de la Loire, and the ANR project ENUMGEOM NR-18-CE40-0009-02.}

\begin{document}

\begin{abstract}
Spin (spherical) random fields are very important in many physical applications, in particular they play a key role in Cosmology, especially in connection with the analysis of the Cosmic Microwave Background radiation. These objects can be viewed as random sections of the $s$-{th} complex tensor power of the tangent bundle of the 2-sphere. 
In this paper, we discuss how to characterize their expected geometry and topology. In particular, we investigate the asymptotic behaviour, under scaling assumptions, of general classes of geometric and topological functionals including Lipschitz-Killing Curvatures and Betti numbers for (properly defined) excursion sets; we cover both the cases of fixed and diverging spin parameters $s$. In the special case of monochromatic fields (i.e., spin random eigenfunctions) our results are particularly explicit; we show how their asymptotic behaviour is non-universal and we can obtain in particular complex versions of Berry's random waves and of Bargmann-Fock's models as subcases of a new generalized model, depending on the rate of divergence of the spin parameter $s$.



\end{abstract}
\maketitle
\tableofcontents
\input{a_Intro_and_MainResults}
\input{b_SpinBundles_and_RandomFields}
\input{c_Jet_and_LipKill}
\input{d_MainResultsAndProofs}

\input{e_Monocromatic}
\appendix
\input{f_Appendices}

\bibliographystyle{alpha}

\bibliography{spin}

\medskip

\medskip

\noindent Antonio Lerario\\
Mathematics Area\\
SISSA Trieste\\
lerario@sissa.it

\medskip

\noindent Domenico Marinucci\\
Dipartimento di Matematica\\
Universita' di Roma Tor Vergata\\
marinucc@mat.uniroma2.it

\medskip

\noindent Maurizia Rossi\\
Dipartimento di Matematica\\
Universita' di Milano Bicocca\\
maurizia.rossi@unimib.it

\medskip

\noindent Michele Stecconi (corresponding author)\\
Laboratoire de Mathématiques Jean Leray\\
University of Nantes\\
michele.stecconi@univ-nantes.fr

\end{document}

%% file: a_Intro_and_MainResults.tex
\section{Introduction and Motivations}\label{sec_intro}
The notion of spin $s$ property ($s\in \mathbb Z$) for functions on the sphere was first introduced in the physics literature by  Newman and Penrose \cite{NP66},  as follows: 

\noindent \emph{a quantity $\tau$ defined on the unit two-dimensional sphere $\mathbb S^2$ has spin weight $s$  if, whenever a tangent vector $v$
 at any point $p$ on the sphere transforms under coordinate change  by
$v'=e^{i \psi} v$, then the quantity at this point $p$ transforms
by  $\tau'=e^{is\psi} \tau$}.

In the mathematical literature Newman and Penrose's theory was developed by \cite{geller2008spin} (see also \cite[Chapter 12]{libro}), who linked the notion of spin $s$ quantity to that of section of the so-called spin $s$ line bundle on the sphere; later, many other papers such as \cite{malya11, LeoSak12, BR13, malyabook} dealt with these geometric objects, strongly motivated by both theoretical interests and cosmological applications \cite[Section 1.2]{libro}.

Apart from their pure mathematical interest, spin spherical functions have drawn extremely strong attention in the last two decades in the Cosmological literature, in particular, in the context of so-called  Cosmic Microwave Background (CMB) polarization data (see e.g. \cite{Durrer}, Chapter 5). Such data are modelled as a section of a vector bundle on the sphere, and indeed they are commonly viewed in a probabilistic sense as a single realization of a random section of a Gaussian spin bundle. The analysis of polarization data extends and generalize the investigation of CMB temperature data, which are viewed as a realization of a Gaussian scalar-valued random field on the sphere; the study of CMB is the major tool to probe Big Bang models and to determine the main cosmological constants, and as such it has been the object of an enormous interest in the last 20 years, leading to two major satellite missions, NASA's WMAP and ESA's Planck, see \cite{Planck20} and the references therein. A similar amount of interest is currently drawn by CMB polarization, which will be the object of the future satellite Mission Lite-Bird and of several ground based observational experiments (\cite{LiteBird}). For instance, it is expected that polarization data may probe the existence of primordial gravitational waves, thus providing the definite proof for the so-called inflationary scenario in Big Bang dynamics, as discussed for instance in \cite{BICEP}. Spin function emerge also in other very important Cosmological observations, most notably in so-called weak gravitational lensing data, the object of the ESA's satellite Mission Euclid (\cite{Euclid}).

\section{An Overview of the Main Results}\label{sec:overview}
\subsection{Our Setting}
The purpose of this paper is to establish a general technique to characterize geometric functionals of spin fiber bundles. These functionals cover, among others, Lipschitz-Killing curvatures for excursion sets, zeroes, critical points and Betti numbers. 
The formal statement of our results will require a considerable amount of discussion and definitions which will be given in the following Sections. We believe it is nonetheless useful to provide first a general overlook of the framework we are interested in and our main results.

In particular, we shall be concerned with random sections of spin fiber bundles of possibly varying order $s_n$, considering also an asymptotic framework where both the spin order $s_n$ and the variances of the random coefficients are allowed to vary with $n$. More precisely, we shall be concerned with spin isotropic Gaussian sections, that is Gaussian random sections $\sigma_n$ of $\spi{s_n}$; we refer to \cite{Ancona2021}, \cite{FengZel}, \cite{GayetWelschinger2016}, \cite{GayetWelschingerCMP2021}, \cite{Letendre2016} for some recent results on zero sets of Gaussian random sections, under different settings than ours in this paper. It is known (see the discussion in Sections \ref{sec_bundles} and \ref{sec:SpinRandomFields}) that spin random sections can be given as a spectral representation of the form 
\be 
\sigma_n:=\sum_{\ell=0}^\infty \sum_{m=-\ell}^{\ell}a^\ell_{m,s_n}(n)Y^\ell_{m,s_n};
\ee
here, $Y^\ell_{m,s_n}$ denotes the family of \emph{spin spherical harmonics}, which we view as deterministic sections of the spin line bundle $\spi{s_n}$, while the random array $a^\ell_{m,s_n}(n)$ represents the spin spherical harmonic coefficients. Spin spherical harmonics were introduced in \cite{NP66} and their connection with the elements of the Wigner's matrices representations for the group $SO(3)$ is discussed in \cite[Chapter 12]{libro}, \cite{BR13}, \cite{malyabook}. 

Our first important remark is that $\sigma_n$ is characterized by the \emph{circular covariance} function $k_n\colon \R\to \R$, where
\be 
\E\{X_n(\mathbb{1})\overline{X_n(R(\f,\h,\psi))}\}=k_n(\h)e^{is(\f+\psi)}
\ee
where $R(\f,\h,\psi) \in SO(3)$ is a rotation characterized by the three Euler angles $(\f,\h,\psi)$. In the special case where $s=0$, i.e., for scalar random fields, it is immediately seen that the covariance function depends only on the parameter $\h$, to be interpreted as an angular distance - but this does no longer hold for general $s \neq 0$, see also \cite{stecconi2021isotropic}.

The behaviour of scalar-valued Gaussian isotropic random fields is well-known to be fully characterized by their angular power spectra, i.e., the variance of the random spherical harmonic coefficients $a^\ell_{m,0}$. This is still the case for a fixed, arbitrary value of the spin parameter; however two random fields with unequal spin parameters $s \neq s'$ will have different geometric and topological properties even if they are endowed the same angular power spectra. 
In particular, it should be noted that the derivatives of the spin field are not stochastically independent for $s\neq0$; indeed we can represent $\sigma$ in local coordinates as a Gaussian field $\xi\colon \D\to \C$ on the disk. Then, the covariance matrix of the random vector $(\xi,\frac{\de \xi}{\de x},\frac{\de \xi}{\de y})$ is the following (see Corollary \ref{cor:cov1stjet}):
\be 
\E\left\{\begin{pmatrix}
\xi \\ \frac{\de \xi}{\de x} \\ \frac{\de \xi}{\de y}
\end{pmatrix}\begin{pmatrix}
\overline{\xi} & \overline{\frac{\de \xi}{\de x}} & \overline{\frac{\de \xi}{\de y}}
\end{pmatrix}\right\}=\begin{pmatrix}
1 & 0 &0 
\\ 0 & -k''(0) & -i\frac{s}{2}k(0) 
\\ 0 & i\frac{s}{2}k(0) & -k''(0)
\end{pmatrix}.
\ee
\begin{remark}
Note that for $s\neq 0$, the ``real and imaginary'' components of $\xi$ are not independent as fields for any choice of local coordinates.
\end{remark}
A simple consequence of this phenomenon is given by our first result below, where the expected value of the number of zeros for a Gaussian spin bundle is established and shown to depend explicitly on $s$: 
\begin{thm}[Expected number of zeroes]\label{thm:Enumber} Let $\sigma\colon S^2\to\spi{s}$ be a smooth Gaussian isotropic spin $s$ random field. Let $k(\h)$ be its \emph{circular covariance} function. Then
\be 
\E\{\# \{\sigma=0\}\}=2\frac{|k''(0)|}{k(0)}+\frac{s^2}2 \frac{k(0)}{|k''(0)|}.
\ee
\end{thm}
The proof of the above theorem will be given in Appendix \ref{sec:zeroproof}.
\begin{example}
The previous equation takes an especially simple form in the case of random eigenfunctions of the spin Laplacian (to be discussed below), i.e., when $\sigma=\sum_m a^\ell_{m,s}Y^\ell_{m,s}$ is monochromatic, with spin equal to $s$; in this case  $k(\h)=d^\ell_{-s,-s}(\h)$ is a Wigner d-function, see e.g. \cite{libro}, Section 3.3.  Here we have $k(0)=1$ and $k''(0)=-\frac12 (\ell(\ell+1)-s^2)$, so that
\be 
\E\{\# \{\sigma=0\}\}=\ell(\ell+1)-s^2+\frac{s^2}{\ell(\ell+1)-s^2}.
\ee
Note that the corresponding eigenvalue with respect to the spin Laplacian $\dd\ddb$ (to be discussed below) is $=-(\ell-s)(\ell+s+1)$.
\end{example}
\subsection{Scaling Limits and Asymptotics}
For the results to follow, we shall assume that a scaling condition of the following form holds; for some scaling sequence $\rho_n\to 0$ and such that
\be\label{eq:beta} 
    \lim_{n\to \infty}s_n\rho_n^2=\beta\in\R,
    \ee 
we have that
    \be \label{eq:introscalass}
    k_{X_n}(x\rho_n)\to k_\infty(x) \text{ and } \lim_{t\to +\infty} k_\infty(t)\to 0.
    \ee
Scaling conditions for sequences of standard (scalar-valued) random fields are known to hold in many circumstances, including random eigenfunctions and needlet fields, see \cite{Wigman_2010}, \cite{CammarotaM2015}, \cite{SarnakWigman2019}, \cite{NourdinPeccatiRossi2019} and \cite{CanzaniHanin2020}; as noted before, we also admit the possibility that the spin parameter $s_n$ depends on $n$.

Under these conditions, we will show that the geometry and topology of $\sigma_n|_{B_{\rho_n}}$ converges to those of a stationary Gaussian Random Field $\xi_\infty\colon \D\to \C$ having covariance function $k_\infty$; with some additional work, we will be able to say something about the global geometry and topology. 
\begin{remark}
$\xi_\infty: \C\to \C$ is circularly symmetric ($\E\{\xi_\infty(z)\xi_\infty(w)\}=0$ because the same is assumed for $X_n$), stationary and with real covariance function: $\E\{\xi(z)\overline{\xi(w)}\}=k_\infty(|z-w|)\in \R$, by construction. It follows that its real and imaginary parts are independent and identically distributed. This allows to apply directly the formulas from \cite{AdlerTaylor} for the Lipschitz-Killing curvatures of the escursion sets of $|\xi_\infty|$.
\end{remark}
For the next statement we need to anticipate the notion of type-W singularities of $\sigma_n$, denoted $Z^W(\sigma_n)$; a rigorous and detailed definition will be given later in Section \ref{subsec:intrising}, see subsection \ref{subsec:examples} below for some examples. For the moment, it suffices to say that $Z^W(\sigma_n)\subset S^2$ is a subset of $S^2$ identified in terms of conditions on the modulus of the random section and its higher order derivatives, encoded in $W$; in practice, the class of random subsets $Z^W(\sigma)$ is general enough to cover for instance excursion sets, level curves, zeroes, critical points and basically all other examples which are usually investigated in stochastic geometry.

Let us define also $Z_\infty^W:=Z^W(\xi_\infty)\subset \D$ to be the random subset of the disk defined by the limit field $\xi_\infty$ mentioned above. We are now able to state the two main results of this paper up to some qualifications to be discussed later, the first regarding asymptotic laws, while the second expected values. Let $B_{\rho_n}$ be a spherical ball of radius $\rho_n$. Heuristically, our objective will be to show a form of convergence in law for the random sets on the shrinking ball to analogous limit random sets for a suitably defined limiting process:
\be 
Z^W(\sigma_n)\cap B_{\rho_n} \to Z^W_\infty.
\ee
These subsets can be characterized as random smooth Whitney stratified subsets (see \cite{AdlerTaylor}), although for the time being we do not discuss this issue in full details.
To give some very simple example our results cover the excursion sets for the norm of random sections; here, the strata are given by the boundary and the interior. With our tools we will also be able to cover much more complicated frameworks, such as the intersections of excursion sets; other examples include the set of critical points for the norm or the set of critical points of the norm of one random section restricted to the excursion set of another, the set of flexes of a level set of the norm, the set of points where the rank of the covariant derivative of the section is one, and many others.

We are now ready to give a more precise statement of our next result.

\begin{thm}\label{thm:mainlaw}
\begin{enumerate}
Under suitable regularity conditions
 \item Almost surely, $Z^W(\sigma_n)\subset S^2$ is regular (i.e. it is a Whitney stratified subset of $S^2$) for $n$ big enough. The same holds for $Z^W_\infty\subset \D$.
\item There exists a discrete limiting probability law $p^W_\infty(Z)$
on the set diffeomorphisms classes of Whitney stratified subsets of $\D$ such that:
\be 
\exists \lim_{n\to +\infty}\P\{Z^W(\sigma_n)\cap B_{\rho_n}\text{ is diffeomorphic to }Z\}=p_\infty^W(Z).
\ee
\item Whenever $Z$ can be realized as a regular type-W singularity of some smooth function $f\in \spt(\xi_\infty)$, we have that $p^W_\infty(Z)>0$.
\item There is convergence in law: $\mathcal{L}_i(Z_n^W\cap B_{\rho_n})\nrw \mathcal{L}_i(Z_\infty^W)$ and $b_i(Z_n^W\cap B_{\rho_n})\nrw b_i(Z_\infty^W)$.
\end{enumerate}
\end{thm}
A more rigorous statement will be given with Theorem \ref{thm:tecnomainlaw}.
\begin{remark} The regularity conditions that we need are going to be discussed below, see Section \ref{sec:setting}; in short, these conditions ensure the regularity (transversality) of the equations that define $Z^W(\sigma_n)$. They can be viewed as a generalization to the spin bundle case of the Morse functions requirements that are needed for the application of Kac-Rice arguments for standard scalar valued fields.
\end{remark}
\begin{remark}
The convergence in law of $Z^W(\sigma_n)\cap B_{b,\rho_n}$ implies the convergence in law of any functional $\mathcal{F}(Z^W(\sigma_n)\cap B_{b,\rho_n})$ depending continuously on $\sigma_n$ with respect to the $\mC^\infty$ topology. The same properties holds for those functionals that have discontinuities contained in the subset of irregular subsets, since these subsets have probability zero of occurring. In particular this holds for the Lipschitz-Killing curvatures $\mathcal{F}=\mathcal{L}_i$ (see \cite{AdlerTaylor} and below for more discussions and exact definitions) and Betti numbers $\mathcal{F}=b_i$.
\end{remark}

\begin{remark}
Theorem \ref{thm:mainlaw} implies for instance that for all sequences of spin random fields having covariance functions which satisfy the same scaling limit, then there exists a \emph{universal} discrete law for the limiting topology of their excursions sets. For instance, as we shall see below this covers the asymptotic topology for the excursion sets of spin eigenfunctions (spin spherical harmonics) for arbitrary, but fixed, values of $s$.
We refer among others to \cite{SarnakWigman2019} for a recent important universality result on the limiting topology of excursion sets for random eigenfunctions on generic two-dimensional surfaces.
\end{remark}
\begin{remark}
The result can actually be stated in a stronger form replacing diffeomorphic with diffeotopic.
\end{remark}
Our next result refers to the global study of the expected values of Lipschitz-Killing curvatures and Betti numbers. 
\begin{thm}\label{thm:mainE}
For all $i=0,1,2$ we have:
\begin{enumerate}
\item 
 \be \E\mathcal{L}_i(Z^W(\sigma_n))=
    \rho_n^i\frac{\vol(S^2)}{\vol(B_{\rho_n})}\left(\E\mathcal{L}_i(Z_\infty^{\mathrm{(int)}})+O(1)\right);
    \ee
    \item 
    There are constants $c_i^W\ge 0,C_i^W>0$ such that
    \begin{equation}
    \frac{\vol(S^2)}{\vol(B_{\rho_\ell})}c_i^W\le \E b_i(Z_n^W)\le 
    \frac{\vol(S^2)}{\vol(B_{\rho_n})}C_i^W;
    \end{equation}
    \item If there exists a smooth function $f\in \spt(\xi_\infty)\subset \mC^\infty (\D,\C)$ such that $Z^W(f)$ is regular and it has a connected component $C\subset \mathrm{int}(\D)$, with $b_i(C)>0$, then $c^W_i>0$.
\end{enumerate}
\end{thm}
As before a more rigorous statement will be given with Theorem \ref{thm:tecnomainE}.
\begin{remark}
It should be noted here that $\mathcal{L}_0(Z_\infty^{\mathrm{(int)}})\neq \chi(Z^W_\infty)$, because the former does not take into account the intersection with the boundary, see Section \ref{sssec:EPproof}. Recall also the standard fact that
\be 
\vol(B_{\rho})=2\pi\left(1-\cos(\rho)\right)=\pi\rho^2+O(\rho^4).
\ee
\end{remark}
\begin{remark}
With the same technique we can include many examples of interest, for instance the expected value of the number of critical values and/or extremes in the regions where the modulus of the spin random section exceeds a certain (fixed) threshold value $u$. As we mentioned earlier these statistics, as well as the Lipschitz-Killing curvatures mentioned before, have many important applications arising in the framework of Cosmic Microwave Background data analysis, see for instance \cite{Cheng2020}.
\end{remark}
\subsection{Non-universal Asymptotic Geometry of Spin Eigenfunctions}\label{sec:nonuniversality}
The setting considered in the previous theorem can be applied to a number of different circumstances where the asymptotic behavior of spin random fields is of interest. For physical applications, natural examples are spin eigenfunctions and their averages, also known as spin needlet fields.  

In this paper for brevity and definiteness we will consider only the former case, i.e. spin eigenfunctions. In the scalar (spin zero) case, the geometry and topology of random eigenfunctions has been the object of very strong interest in the last decade, see among others  \cite{NazarovSodin2009}, \cite{Zel_S3} for the number of nodal domain, \cite{Wigman_2010} for the variance of nodal lines, \cite{MPRW}, \cite{NourdinPeccatiRossi2019}, \cite{MRossiWigman2020}) for their limiting distributions, \cite{SarnakWigman2019} for universality results on topology and Betti numbers, \cite{CammarotaM2018} for Lipschitz-Killing Curvatures, \cite{CanzaniHanin2020} for universality results on two-dimensional manifolds. 

We will consider below three different settings. In particular, we shall consider the limiting behavior of the spin eigenfunctions
\be\label{eq:monocromatic}
\sigma_{\ell}=\sum_{m}a^\ell_{m,s(\ell)}Y^\ell_{m,s(\ell)}
\ee
corresponding to eigenvalues $\lambda_{\ell,s}:=-(\ell-s)(\ell+s+1)$,  where $n=\ell\to+\infty$ and
\be 
|s_\ell|=\ell-\gap_\ell,
\ee
with $\gap_\ell\le \ell$, under three different regimes:
\begin{enumerate}[a.]
    \item (The Berry regime) $\liminf_{\ell\to \infty}\gap_\ell=+\infty$; this covers the cases where $s$ is fixed (and $\ell\to\infty$) or $s$ grows with $\ell$ even linearly, but $\ell-s_\ell$ diverges. In this case the shrinking rate is $\rho_\ell(s_\ell)$, where
    \be\label{eq:berryrate} 
\rho_\ell(s)=\frac{1}{\sqrt{(\gap+1)(2\ell-\gap)}}\sim\frac{1}{\sqrt{(\ell-s) (\ell+s+1)}}=\frac{1}{\sqrt{\lambda_{\ell,s}}}.
\ee
    and the associated limit field on the disc is the Berry random field, indeed the limit of the circular covariance function is $k_\infty(x)=J_0(x)$ and $\beta=0$.    
    \item (The middle regime) $\gap_\ell=\gap$ for some fixed $\gap\in \N\-\{0\}$. The real part of the covariance $k_\infty$ is an explicit analytic function, computed in Equation \ref{eq:Polynomiallimit}.:
    \be 
    M_\gap(x):=\sum_{j=0}^\gap\frac{\gap!}{(\gap+1)^j(\gap-j)!}\frac{(-1)^j}{j!j!}\left(\frac{x}{2}\right)^{2j}e^{-\frac{x^2}{4(\gap+1)}}
    \ee
    In this case the shrinking rate is analogous to formula \eqref{eq:berryrate} and we can write:
    \be\label{eq:polyrate} 
    \rho_\ell(s_\ell)=\frac{1}{\sqrt{(\gap+1)(2\ell-\gap)}}
    \ee
    \item (The Bargmann-Fock regime) when $|s_\ell|=\ell$, i.e. $\gap_\ell=0$; this is the only case where the section is holomorphic; the rate of convergence is 
     \be\label{eq:fockrate} 
    \rho_\ell=\frac{1}{\sqrt{2\ell}}.
    \ee
    Here the associated limit field is the complex Bargmann-Fock field, with $k_\infty(x)=e^{-\frac{x^2}{4}}$ and $\beta=\pm 1$; note that the rate of convergence is indeed the same as \eqref{eq:polyrate} in the special case where $r=0$.
    
\end{enumerate}
\begin{remark}\label{rem:polybessel}
Note that all rates between $\rho_\ell(s_{\ell})=O(\ell^{-1})$ and $\rho_\ell(s_{\ell})=O(\ell^{-1/2})$ can be attained for suitable choices of $s_{\ell}$. We can moreover observe that as $\gap\to +\infty$
\be 
\lim_{\gap\to+\infty} M_\gap(x)=J_0(x).
\ee
Hence the middle regime can be heuristically viewed as a form of smooth interpolation between complex Berry's Random Waves (obtained for $\gap\to +\infty$) and the complex Bargmann-Fock model (obtained for $\gap=0$).
\end{remark}
\begin{remark}
The monochromatic waves in the case $s=\ell$ are holomorphic sections of the line bundle $\spi{s}$, indeed they are in correspondence (by identifying $S^2$ with the Riemann sphere $\C\P^1$ and $\spi{s}$ with $O(2s)$) with polynomials of degree $2s$ in one complex variable, see \cite{stecconi2021isotropic}. Moreover, we note that in this case the sequence of monochromatic spin Gaussian fields with $s=\ell$ corresponds to the sequence of complex Kostlan polynomials of degree $2s$, see also \cite{AnconaLetendre,BeliaevMuirhead2021,BeliaevMuirhead2020,Armentano}.
\end{remark}
    \begin{remark}
    The limit that we obtain in the so called Bargmann-Fock case (i.e. the regime $s=\ell$) can be written explicitely as 
    \be 
    \xi_\infty(z)=\sum_{n\in\N}\frac{1}{\sqrt{n!}}\gamma_n\left(\frac{z}{\sqrt{2}}\right)^n e^{-\frac{|z|^2}{4}}.
    \ee
    It should be noted that neglecting the factor $e^{-\frac{|z|^2}{4}}$ this model would correspond to the well known Gaussian entire process, considered for instance in \cite{SodinTsirelson}. This model is not stationary, indeed the variance grows with $|z|$ as $e^{\frac{|z|^2}{4}}$; heuristically, this can be explained by noting that the stereographic projection, which is holomorphic, over the tangent plane stretches the length of tangent vectors more and more as they get further and further away from the origin of the coordinates. For our model, this would correspond to the variance of the scaling limit getting larger and larger as $z$ grows. Indeed, the factor $e^{-\frac{|z|^2}{4}}$ is a consequence of the fact that we use a trivialization of the bundle and of the sphere that comes from the exponential map instead than from the stereographic projection: the metric on the fiber differs by a factor that exactly compensates.
    Despite the fact that the variance is constant the limit is not stationary, in fact it has covariance
    \be 
K_{\xi_\infty}(z_1,z_2)=\E\left\{\xi_\infty(z_1)\overline{\xi_\infty(z_2)}\right\}=e^{\frac{|z_1-z_2|}{4}}\exp\left(\frac{i}{2}\Im (z_1\overline{z_2})\right).
\ee
    \end{remark}
On the basis of these results it is possible to give more explicit formulae for the expectations of Lipschitz-Killing curvatures, for instance see Theorem \ref{thm:berry} below.
In particular, in the Berry regime, we can provide the following result on the Lipschitz-Killing curvatures for the excursion sets of spin random sections.

\begin{thm}\label{thm:berry}
Assume that $\E\|\sigma_{\ell}(p)\|^2=1$ and that $\sigma_\ell$ is as above in the Berry regime. 
Then for any $u>0$, we have the following asymptotic identities, with $\rho_\ell=\frac{1}{\sqrt{(\ell-s_\ell) (\ell+s_\ell+1)}}$:
\begin{enumerate}[i.]
    \item $
    \E\vol_2(\{|\sigma_\ell|\ge u\})=4\pi e^{-\frac{u^2}{2}}+o(1)$.
    \item 
    $
    \E\vol_1(\{|\sigma_\ell|= u\})=\frac{1}{\rho_\ell} \cdot (2\pi^{\frac32}ue^{-\frac{u^2}{2}}+o(1))$.
    \item $
    \E\chi(\{|\sigma_\ell|\ge u\})=
    \frac{1}{\rho_\ell^2}\cdot ((u^2-1)e^{-\frac{u^2}{2}}+o(1))$.
    \item $\E\#\{\sigma_\ell=0\}=\frac{1}{\rho_\ell^2}(1+o(1))$.
    \item There are positive constants $c^W_i,C^W_i>0$ such that for $\ell$ big enough, we have \be 
    c^W_i\frac{4}{\rho_\ell^2}\le\E b_i\{\sigma_\ell=0\}\le C^W_i\frac{4}{\rho_\ell^2}.
    \ee
\end{enumerate}
\end{thm}
\begin{remark}
Notice that $\E\chi(\{|\sigma_0|\ge u\})$ is not continuous at $u= 0$, in that
\be\label{eq:echicont1} 
\lim_{u\to 0^+}\chi(\{|\sigma_0|\ge u\})=2-\#\{\sigma_0=0\},
\ee
while $\chi(S^2)=2$. This should not surprise, in that for small values of $u$, the excursion set $\{|\sigma_0|\ge u\}$ is just the complement of a small neighborhood of the zero set, thus \eqref{eq:echicont1} holds almost surely. On the other hand, it is clear that the first two identities $i.$ and $ii.$ are still true for $u=0$.
\end{remark}
\begin{remark}
The same result 
can be easily shown to hold also in the case when $\hat{\sigma}_\ell$ is non-Gaussian, but uniformly distributed on the sphere of radius $2\sqrt{\pi}$ in $L^2(S^2|\spi{s(\ell)})$. Indeed
\bega
\E\{\mathcal{L}(\sigma_\ell\ge u)\}=\int_{0}^\infty
\E\left\{\mathcal{L}\left(\hat{\sigma}_\ell\ge \frac{u2\sqrt{\pi}}{t}\right)\right\}p_{\|\sigma_\ell\|_{L^2}}(t)dt
\eega
and since $\|\sigma_\ell\|_{L^2}\to 2\sqrt{\pi}$ a.s., it follows that the density $p_{\|\sigma_\ell\|_{L^2}}(t)\to \delta_{2\sqrt{\pi}}$ ($\delta_t$ being the Dirac mass in $t$), in the distributional sense, thus the asymptotics.
\end{remark}
\begin{remark}
Results analogous to theorem \ref{thm:berry} can be established for the two other regimes b. and c. as a consequence of Theorem \ref{thm:mainE}, of course replacing the scaling factors $\rho_\ell$ appropriately. However, it should be noted that the computation of the multiplicative constants $\E\mathcal{L}_i(Z_\infty^{\mathrm{(int)}})$ is in these cases more challenging: for instance, in case c. (the complex Bargmann-Fock) the real and imaginary parts of the limit field are not independent and hence the Gaussian kinematic formula \cite{AdlerTaylor} does not hold. In any case, we stress that what we are omitting here is just a (tedious) computation concerning only the limit field.
\end{remark}
\begin{remark}(On the law of large numbers)
Using similar techniques as in the proof of Theorem \ref{thm:mainE}, together with the $\mC^\infty$ convergence of the covariance kernel of any pair of rescaled fields it is actually possible to prove a law of large numbers result for the Lipschitz-Killing curvatures, i.e.
\be 
\frac{\mathcal{L}_i(Z^W(\sigma_\ell)}{\E\{\mathcal{L}_i(Z^W(\sigma_\ell))\}}\xrightarrow[\ell\to\infty]{L^2}1
\ee
We plan to address these issues and related ones about central limit theorems in a forthcoming paper.
\end{remark}

\begin{remark}[Non-universal asymptotic topology of spin eigenfunctions]
As for the Lipschitz-Killing curvatures the scaling factors appearing in the asymptotic behavior of the Betti numbers are different in each of the three considered regimes.

In the Berry case we prove below that $c^W_0>0$; the proof requires a modified maximum principle for the solutions of the Helmholtz equation discussed in Appendix \ref{app:berry}; 
more details are given below. For $b_1$ one can modify the scaling sequence $\rho_n$ by a constant factor $c$ strictly bigger than the first minimum point of $J_0$ and then run an analogous argument together with Alexander's duality to prove that $c^W_1>0$.

The behavior of Betti numbers in cases b. and c. is discussed in Remark \ref{rem:excpol}. The upper bound for the expected values of the Betti numbers $b_0$ and $b_1$ takes the same form (the constant $C^W_i$ is finite) as for the Berry case; moreover, the lower bound holds with a strictly positive constant $c_0^W$ for $b_0$ in both cases.
On the other hand, it is not possible in those environments to prove that the constant $c^W_1$ appearing in the lower bound for $b_1$ is strictly positive. In fact condition $(3)$ in Theorem \ref{thm:mainE} is not satisfied by the complex Bargmann-Fock field, due to the maximum principle. However, this does not imply that the lower bound fails, although we conjecture that it does.

It may be further noted that the expected number of connected components for the excursion sets is $O(\ell^2)$ when $s$ is fixed or bounded away from $\ell$, it is $O(\ell(\ell-|s_\ell|)|)$, if $s_\ell<\ell$ can grow as quickly as $\ell$,  and finally it is $O(\ell)$ in the holomorphic case  $s=\ell$; the same asymptotics hold for the first Betti number $b_1$.

\end{remark}
 \subsection{Asymptotic behaviour of Lipschitz-Killing curvatures and Betti numbers of singularities}\label{subsec:examples}
We note first that the excursion sets of the norms of monochromatic spin Gaussian fields give us the possibility to illustrate some very concrete examples of singularity sets. For instance
\begin{enumerate}
    \item If $r=0$, then $J^0(S^2|\spi{s})=E(\spi{s})$ is the total space of the line bundle. Let $B_u(\spi{s})$ be the total space of the $u$ ball bundle and let $W_u$ be its complement. Then
    \be 
    Z^{W_u}_\ell=\{|\sigma_\ell|\ge u\}
    \ee
    is the excursion set of the norm of the field.
    In this case the only meaningful Betti number is $b_0$, the number of connected components. The number of connected components of the boundary is the $b_0=b_1$ of 
    \be 
    Z^{S_u(\spi{s})}_\ell=\{|\sigma_\ell|=u\}.
    \ee
    \item Given any function $f\colon S^2\to \R$ we can define a singularity $W\subset J^1(S^m|\spi{s}\oplus \R)$ such that
    \be 
    \mathrm{Crit}\left(f|_{\{|\sigma_\ell|=u\}}\right)=\left(j^1(\sigma_\ell,f)\right)^{-1}(W).
    \ee
    \item 
    The random set
    \be 
    \{p\in S^2\colon |\{|\sigma_\ell|=u\} \text{ has a flex in p}\}
    \ee
    is of this form.
\end{enumerate}   
   \subsection{Plan of the paper}
   Sections \ref{sec_bundles} and \ref{sec:SpinRandomFields} introduce our framework in terms of the formal construction of spin line bundles and the definition of isotropic spin random fields on the sphere; these Sections build upon some previous references, including in particular \cite{geller2008spin}, \cite{BR13}, \cite{malyabook}. Section \ref{sec:jetW} and \ref{sec:LKcurv} introduce the geometrical tools that we are going to explore, in particular jet bundles, type-W singularities (see \cite{mttrps}) and their description as Whitney stratified subsets of the sphere, Lipschitz-Killing curvatures in their integral form and their alternative expression in terms of critical points of stratified Morse functions. Section \ref{sec:scaass} and \ref{sec:MainResults} give our asymptotic framework and main results, Theorems \ref{thm:mainlaw} and \ref{thm:mainE}, whose proofs are collected also in Section \ref{sec:proofmainE}.
   Finally, Section \ref{sec:monocrom} specializes our results to the monochromatic case, whereas some technical lemmas are collected in the Appendix.
   
 

%% file: b_SpinBundles_and_RandomFields.tex
\section{Spin line bundles}\label{sec_bundles}

In this Section, coherently with Newman and Penrose's theory (see Section \ref{sec_intro}), we introduce the notion of spin line bundles on the sphere giving both the intrinsic definition and the description in terms of an atlas, with great attention to so-called spin sections. 

\subsection{Intrinsic definition}\label{sec_intrinsic}

The $3$-dimensional special group of rotations $SO(3)$ acts transitively on the two-dimensional unit sphere $\mathbb S^2\subset \mathbb R^3$ with an action that we denote by $g\mapsto g p$, $g\in SO(3), p\in \mathbb S^2$. Let us fix once for all a point $p_0\in \mathbb S^2$, and define $K$ to be the isotropy group of $p_0$, i.e. the subgroup of elements $g\in SO(3)$ such that $g p_0=p_0$, then $K\cong U(1)$ the circle group, and $\mathbb S^2 \cong SO(3)/K$. 
Let us denote by $\chi_s$, $s\in \mathbb Z$, the family of characters of $K$ ($\chi_0$ is the trivial representation); for every $s\in \mathbb Z$, the group $K$ acts on $SO(3)\times \mathbb C$ as follows: for $(g,z)\in SO(3)\times \mathbb C$,
$$
   k\mapsto k(g, z) := (gk, \chi_s(k^{-1})z).
$$
We denote by $SO(3)\times_s \mathbb C$ the space of orbits $\lbrace \theta(g,z), (g,z)\in SO(3)\times \mathbb C\rbrace$, where $\theta(g,z) = \lbrace k(g,z), k\in K\rbrace$, and consider the (projection) map
\begin{eqnarray*}
    \pi_s : SO(3)\times_s \mathbb C &\to& \mathbb S^2\\
    \theta(g,z) &\mapsto& gK.
\end{eqnarray*}
Let us set $\xi_s := (SO(3)\times_s \mathbb C, \pi_s, \mathbb S^2)$, then $\xi_s$ is a complex line bundle, indeed the fiber over $p$ is $\pi_s^{-1}(\lbrace p\rbrace)\cong \mathbb C$ for every $p\in \mathbb S^2$.  We call $SO(3)\times_s \mathbb C$ (resp. $\mathbb S^2$) the total (resp. base) space of $\xi_s$. Plainly, for $s=0$ we obtain the trivial bundle, in particular $SO(3)\times_0 \mathbb C \cong \mathbb S^2\times \mathbb C$. 
\begin{defi}\label{def_spinbundle}
The \emph{spin $s$ line bundle on the sphere} is the triplet $\xi_s = (SO(3)\times_s \mathbb C, \pi_s, \mathbb S^2)$. 
\end{defi}
Let us now recall the notion of section of $\xi_s$: it is a map $\sigma:\mathbb S^2 \to SO(3)\times_s \mathbb C$ 
that associates to each $p\in \mathbb S^2$ one element of its fiber $\pi_s^{-1}(\lbrace p \rbrace)$, i.e. some $\theta(g,z)\in SO(3)\times_s \mathbb C$ such that $gK = p$. We call such a $\sigma$ a \emph{spin $s$ section}. Plainly, spin $0$ sections are identified with complex valued functions on the sphere. 
\begin{remark}\label{rem_pullback}
There is a one to one correspondence between spin $s$ sections $\sigma$ and complex valued functions $f$ on $SO(3)$ such that for every $g\in SO(3)$ and every $k\in K$
\begin{equation}\label{type_s}
    f(gk) = \chi_s(k^{-1}) f(g).
\end{equation}
(We call $f:SO(3)\to \mathbb C$ satisfying \eqref{type_s} a \emph{function of right spin $-s$}.)  Indeed, given $f$ satisfying \eqref{type_s}, the corresponding section $\sigma=\sigma^f$ is defined as follows: for $\mathbb S^2\ni p=g_pK$,
$$
\sigma(p) := \theta(g_p, f(g_p)).
$$
(Note that this definition does not depend on the coset representative.) 
On the other hand, consider a section $\sigma$ of $\xi_s$, then $\sigma(p) = \theta(g_p,z_p)$ where $p=g_pK$. Define the corresponding function $f=f^\sigma$ of right spin $-s$ (called the \emph{pullback function of $\sigma$} in \cite{BR13}) as follows: 
$$f(g_p) := z_p, \qquad f(g_pk) := \chi_s(k^{-1})z_p \ \text{ for } k\in K.$$ 
Plainly, functions of type $0$ are constant on left cosets of $K$ in $SO(3)$, hence they are identified with complex valued functions on $\mathbb S^2$.  
\end{remark}
We will always work with sections that are at least \emph{continuous}: on the total space and the base space of $\xi_s$ we consider the respective Borel $\sigma$-fields, in particular this ensures the bundle projection $\pi_s$ to be continuous itself. Hence a spin $s$ section $\sigma$ is continuous if and only if its pullback function $f^\sigma$ is a continuous function of right spin $-s$.

\subsubsection{Tensor representation}

In this paper, we will extensively use the following tensor representation, an alternative approach to the theory of spin line bundles than the one leading to Definition \ref{def_spinbundle}.

\begin{remark}\label{rem_tensor}
There are two choices for the isomorphism $\chi_1\colon K\to U(1)$ depending on the orientation of $K$. To make such choice is equivalent to choose an orientation of the tangent space at $p_0\in \mathbb S^2$ and thus an orientation of $\mathbb S^2$: this is due to the fact that $K$ can be embedded as a small circle around $p_0$ by drawing the orbit $Kp$ of a point $p\in\mathbb S^2$ close to $p_0$.
Then, the tangent bundle on the sphere, denoted by $\mathcal{T}:=\left(T \mathbb S^2,\pi,\mathbb S^2\right)$ and equipped with the rotation of angle $\frac\pi 2$ coherent with the given orientation, is isomorphic to $\xi_1$ as a complex line bundle. Observing that $\chi_s=(\chi_1)^s$, it follows that for all $s\in\N$, we have
\begin{equation}\label{def_tensor}
\xi_s\cong  \mathcal{T}^{\otimes s}, \qquad \xi_{-s}\cong  \left(\mathcal{T}^*\right)^{\otimes s},
\end{equation}
where $\mathcal{T}^*$ is the so-called cotangent bundle on the sphere, equipped with the dual almost complex structure  (recall that $\xi_0$ is the trivial bundle) and $\otimes$ denotes the complex tensor product.
In other words, $\xi_s$ is \emph{the} complex line bundle with Euler characteristic $2s$. This holds for whatever choice of orientation of $\mathbb S^2$. 
\end{remark}
For $s\in \mathbb N$, bearing in mind (\ref{def_tensor}), 
let $p\in \mathbb S^2$ and $v\in T_p \mathbb S^2\-\{0\}$, where $T_p\mathbb S^2$ denotes the tangent space at point $p$, then for the fiber over $p$ we have 
\bega\label{fiber_tensor}
\pi_s^{-1}(\lbrace p \rbrace) &\cong\left\{\sum_{i}v_1^i\otimes\dots\otimes v_s^i\colon v_j^i\in T_p \mathbb S^2\right\}=\{z\cdot v^{\otimes s}\colon z\in\C\},
\eega
where as usual \be\label{eq:voov} v^{\otimes s} := \underbrace{v\otimes\dots \otimes v}_{s \text{ times}};
\ee 
see also Figure \ref{fig1} for an alternative representation.
When $v$ changes, say $v'=wv$, then the vector $v^{\otimes s}$ changes accordingly to
\be 
(v')^{\otimes s}=w^s v^{\otimes s}.
\ee
\begin{remark}\label{rem:cozione}
Note that the coordinates of $\tau = z v^{\otimes s}=z'(v')^{\otimes s}$, identified with an element of the fiber over $p$ via (\ref{fiber_tensor}), have spin weight $-s$, i.e.
\be 
z'=w^{-s}z,
\ee
indeed ``the coordinates of vectors are covectors, hence they belong to the dual''.
\end{remark}
An analogous representation holds for $s<0$, it suffices to replace $T_p(\mathbb S^2)$ with the cotangent space $T^*_p(\mathbb S^2)$ at point $p$, hence for $v\in T^*_p(\mathbb S^2)\setminus \lbrace 0 \rbrace$
\begin{equation*}
    \pi^{-1}_s(p) \cong \lbrace z \cdot v^{\otimes -s}: z\in \mathbb C\rbrace. 
\end{equation*}

It is worth stressing that (\ref{def_tensor}), in light of (\ref{fiber_tensor}) and the discussion thereafter, gives the most natural definition of spin $s$ line bundle on the sphere according to Newman and Penrose's theory, see Section  \ref{sec_intro}. \emph{From now on, the spin $s$ line bundle $\xi_s$ (Definition \ref{def_spinbundle}) will be tacitly identified with $\mathcal T^{\otimes s}$ for $s>0$ (resp. with $\left ( \mathcal T^* \right )^{\otimes -s}$ for $s<0$), and with the trivial bundle for $s=0$ -- as explained in Remark \ref{rem_tensor}.}

\begin{figure}
\centering
\includegraphics[scale=0.4]{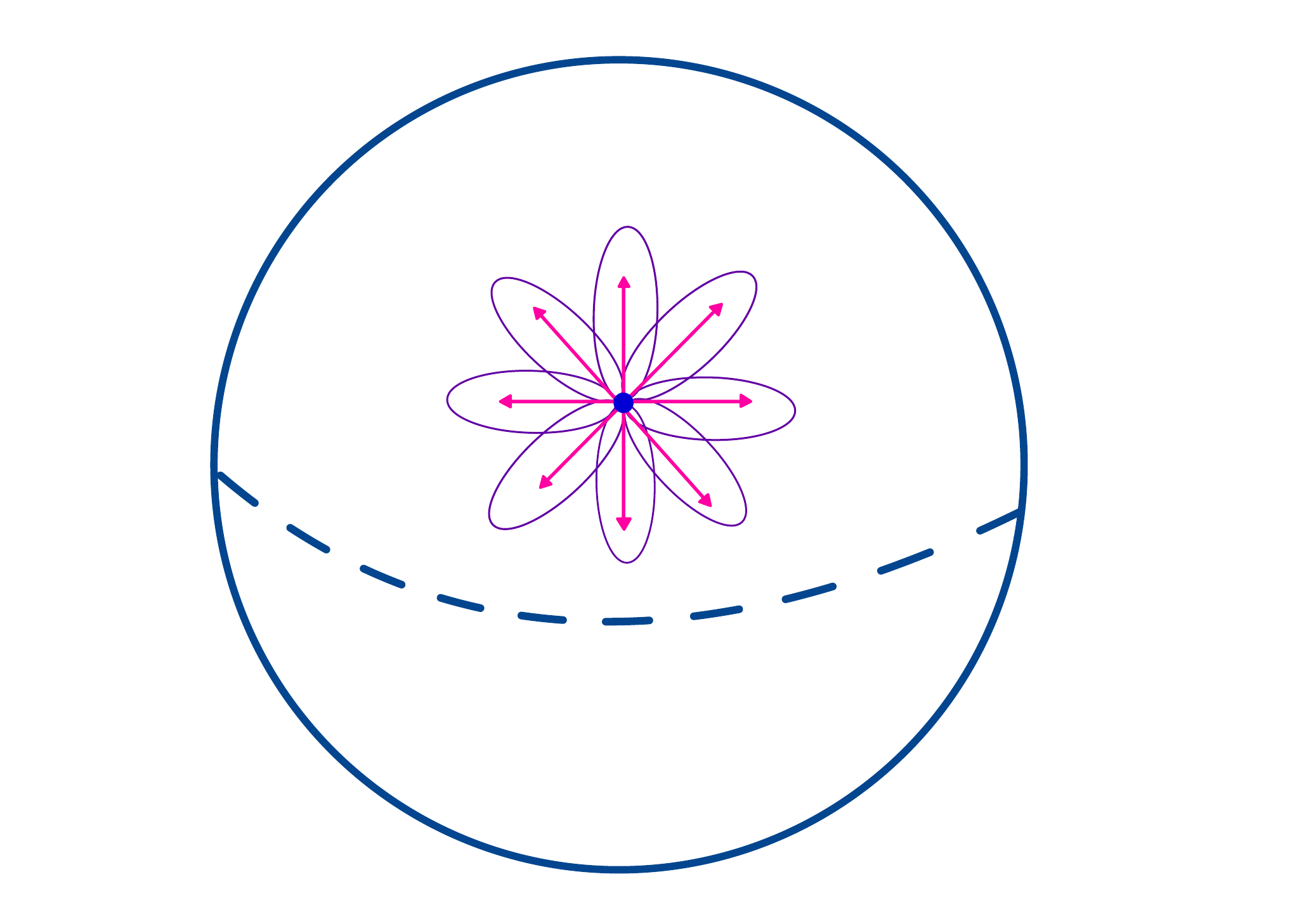}
\caption{\emph{Non c'è rosa senza spin.} 
We can represent an element belonging to $\pi_s^{-1}(\lbrace p\rbrace)$ as $\tau=(z_1v)\otimes\dots \otimes(z_sv)$, where $z_1,\dots,z_s$ are the solutions of $z^s=1$. This can be pictured as flower centered at $p$ with $s$ petals, drawed on the surface of the sphere.}\label{fig1}
\end{figure}

Let $s>0$ and $\sigma$ be a section of $\xi_s$. Then from (\ref{fiber_tensor}), obviously,  for any point $p\in \mathbb S^2$ and $v\in T_p \mathbb S^2=p^\perp$, we have 
\be 
\sigma(p)=z v^{\otimes s},
\ee
for some $z=z_\sigma(p,v)\in\C$. A convenient way to understand this $z_\sigma(p,v)$ is to observe that for any such $p,v$
 there exists a unique positive ``rotation'' $g\in SO(3)$ such that $ge_3=p$ and $ge_2=v$. Here $e_1,e_2,e_3$ are the axes of the coordinate system (we consider the standard, right-handed basis for $\mathbb R^3$). 
It follows that the section $\sigma \colon \mathbb S^2 \to SO(3)\times_s \mathbb C$ is uniquely determined by a function $f=f^\sigma\colon SO(3)\to \C$ such that
\be \label{eq:idspirule}
\sigma (g e_3)= f (g)\left(g e_2\right)^{\otimes s},
\ee
where here $ge_2$ must be intended as an element of $T_{ge_3}\mathbb S^2$, cf. Remark \ref{rem_pullback}. Analogous considerations hold for $s<0$ replacing the tangent space with the cotangent space. 
\subsection{Hermitian metric}
The complex line bundles $\spi{s}$ are endowed with a natural hermitian metric, defined as follows, via the induced norm, see also \cite{stecconi2021isotropic}.
\begin{defi}
Let $\|\cdot \|\colon \spi{s}\to [0,+\infty)$, such that if $v\in T_pS^2$ has length $\|v\|=1$, then $\|v^{\otimes s}\|=\frac1s$, see equation \eqref{eq:voov}.
\end{defi}
\begin{remark}
This is the only choice for which all the maps below are Riemannian coverings, see \cite{stecconi2021isotropic}.
\be 
SO(3)\xrightarrow{\cong} S(\spi{1})\to \frac{1}{s}S(\spi{s})\to \frac{1}{ks}S(\spi{ks})
\ee
\be 
\begin{pmatrix}
u & v & p
\end{pmatrix}\mapsto (v,p)\mapsto (v^{\otimes s},p)\mapsto (v^{\otimes sk},p).
\ee
Here, $S(\spi{s})=\{\|\cdot\|=1\}$ denotes the unit sphere bundle of $\spi{s}$ with respect to the chosen metric.
\end{remark}
\subsection{Trivialization via Euler's angles} \label{sec:EulerAngles}

Euler's angles are three angles the we denote by $\f,\h, \psi$
describing the orientation of a rigid body with respect to a fixed coordinate system. We use the same convention as in \cite[Section 3.2]{libro};
let $g\in SO(3)$ be any rotation, \cite[Proposition 3.1]{libro} ensures that $g$ can be realized as the sequential composition of three elementary rotations, i.e., rotations around the axes $e_1,e_2,e_3$ of the coordinate system, as follows. 
\begin{prop}[Proposition 3.1 in \cite{libro}]
Each rotation $g\in SO(3)$ can be realized sequentially as
\begin{equation}\label{euler_angles}
    g = R(\f,\h,\psi)=R_3(\f)R_2(\h)R_3(\psi), \qquad \f\in [0,2\pi),\,\h\in [0,\pi],\, \psi\in [0,2\pi),
\end{equation}
where for $\alpha \in \mathbb R$
\bega
R_3(\alpha):=\begin{pmatrix}
\cos\alpha & -\sin\alpha & 0\\
\sin\alpha &\cos\alpha & 0\\
0&0&1
\end{pmatrix}, \qquad 
R_2(\alpha):=\begin{pmatrix}
\cos\alpha &0& \sin\alpha \\
0&1&0\\
-\sin\alpha &0&\cos\alpha
\end{pmatrix}.
\eega
Representation (\ref{euler_angles}) is unique whenever $\h \ne 0,\pi$. If $\h = 0$, then only the sum $\f + \psi$ is uniquely defined. If $\h = \pi$, then only the difference $\f -\psi$ is uniquely defined. 
\end{prop}
The matrix $g(\f,\h,\psi)\in SO(3)$ can be interpreted as an element of the positive orthonormal frame bundle of $\mathbb S^2$ as follows. Let $T_p\mathbb S^2=p^\perp$ be endowed with the standard complex structure: multiplication by $i$ is the anticlockwise rotation by angle $\frac12\pi$. Let $p\in \mathbb S^2$ have (standard) polar coordinates $(\h,\f)$ and let us define the orthonormal basis of $T_p \mathbb S^2$, given by the downward meridian and anticlockwise parallel directions:
\be 
\hat{\h}(p)=\frac{\de}{\de \h}(p) = \begin{pmatrix}
\cos \varphi \cos \theta\\
\sin \varphi \cos \theta\\
-\sin \theta 
\end{pmatrix}, \qquad \hat{\f}(p)=\frac{1}{\sin\h}\frac{\de}{\de \f}(p) = \begin{pmatrix}
-\sin \varphi \\
\cos \varphi \\
0
\end{pmatrix}.
\ee
Then
\bega\label{eq:so3euler}
g(\f,\h,\psi)
&=(e^{i\psi}\hat{\h}(p),e^{i\psi}\hat{\f}(p),p)=(\cos\psi \hat{\h}+\sin\psi \hat{\f},-\sin\psi \hat{\h}+\cos\psi \hat{\f},p)\\
&=(\hat{\h}(p),\hat{\f}(p),p)\begin{pmatrix}
\cos\psi & -\sin\psi & 0\\
\sin\psi &\cos\psi & 0\\
0&0&1
\end{pmatrix}.
\eega

\begin{remark}\label{rem_orientation}
Recall Remark \ref{rem_tensor}. In this paper we consider the sphere $\mathbb S^2$ to be oriented in the usual way, with respect to the outer normal direction, so that we define, for every $\psi\in \R$,
\be 
\chi_s\begin{pmatrix}
\cos\psi & -\sin\psi & 0 \\
\sin\psi & \cos\psi & 0
\\
0 & 0& 1
\end{pmatrix}:=e^{is\psi},
\ee
where $\chi_s$ still denotes the $s$-th linear character of $K$, the isotropy group of $p_0$. 
Notice that when the sphere $\mathbb S^2$ is identified with the Riemann sphere $\mathbb{CP}^1$ by means of the stereographic projection from the north pole, the orientation induced by the complex structure is the opposite.
\end{remark} 

\begin{remark}
In \cite[p. 287]{libro} the transition functions, see e.g. \cite[Definition 2.3]{Hus94}, for the bundle $\xi_s$ are 
\be \label{eq:tranz}
f_{g_2}(x)=e^{is\psi_{g_2 g_1}}f_{g_1}(x),
\ee
where $\psi_{g_1 g_2}$ is the  angle between $\frac{\de}{\de \f_{g_1}}$ and $\frac{\de}{\de \f_{g_2}}$. In other words, 
\be 
\frac{\de}{\de \f_{g_1}}=e^{i\psi_{g_2 g_1}}\frac{\de}{\de \f_{g_2}}.
\ee
Therefore the rule \eqref{eq:tranz} is equivalent to the transition rule for $\mathcal T^{\otimes s}$:
\be 
f_{g_2}(x)\left( \frac{\de}{\de \f_{g_2}}\right)^{\otimes s}=f_{g_1}(x)\left( \frac{\de}{\de \f_{g_1}}\right)^{\otimes s}.
\ee
\end{remark}
It is easy to see that a function $f \colon SO(3)\to \C$ is associated with a section $\sigma$ of $\xi_s$ if and only if
\be\label{eq:Fspinrule}
f(gR_3(\psi))=f(g)e^{-is\psi},
\ee
for any $\psi\in \R$. Indeed
\bega
f^\sigma(g)\left(ge_2\right)^{\otimes s}&=\sigma(ge_3)=
\sigma(gR_3(\psi)e_3)=f^\sigma(gR_3(\psi))\left(gR_3(\psi)e_2\right)^{\otimes s}
\\
&=f^\sigma(gR_3(\psi))\left(-\sin\psi ge_1+\cos\psi ge_2\right)^{\otimes s}
\\
&=f^\sigma(gR_3(\psi))\left(e^{i\psi}ge_2\right)^{\otimes s}
\\
&=f^\sigma(gR_3(\psi))e^{is\psi}\left( ge_2\right)^{\otimes s}.
\eega
\begin{thm}\label{thm:pullback}
Sections of $\xi_s$ are in bijections with functions $f\colon SO(3)\to \C$ that satisfy the rule \eqref{eq:Fspinrule}, via the identity \eqref{eq:idspirule}. We say that $f=f^\sigma$ is the \emph{pullback} of $\sigma$ (see \cite{BR13}) and that $f$ has \emph{right spin $-s$}. 
\end{thm}
\begin{remark}
This change of sign in the spin weight is explained by the fact that $f^\sigma(g)$ is actually a function that expresses the \emph{coordinates} (see Remark \ref{rem:cozione}) of $\sigma$ in the trivialization of the bundle $\spi{s}$ determined by $g$.
\end{remark}
Consider now the Euler angles $\h,\f,\psi\in (0,\pi)\times (0,2\pi)\times (0,2\pi)$  on $SO(3)\-\{\pm R_3(t)\colon t\in\R\}$ as coordinates on the frame bundle of $T \mathbb S^2$ (which is indeed isomorphic to $SO(3)$). In particular, we see that for any fixed $\psi$, the angles $\h,\f$ give trivializations of $\xi_s$ over the set $\mathbb S^2\-\{\pm e_3\}$ as follows:
\be\label{eq:triviatrivializacion}
 (0,\pi)\times \R_{/2\pi\Z} \times \C\cong SO(3)\times_s \mathbb C|_{\mathbb S^2\-\{e_3,-e_3\}}=\bigsqcup_{p}\{p\}\times \pi_s^{-1}(\lbrace p\rbrace)
\ee
\be 
\begin{aligned}
\h,\f,z \mapsto &
\left(\left(R(\f,\h,\psi)e_3\right), z\cdot \left(R(\f,\h,\psi)e_2 \right)\otimes \dots \otimes \left(R(\f,\h,\psi)e_2\right)\right)
\\
&=
\left(\begin{pmatrix}
\sin\h\cos\f \\
\sin\h\sin\f\\
\cos\h
\end{pmatrix}, z \cdot (e^{i\psi}\hat{\f}) \otimes \dots \otimes (e^{i\psi}\hat{\f})\right)
\\
&=\left(p,z\cdot \left(e^{i\psi}\hat{\f}\right)^{\otimes s}\right)
\end{aligned}
\ee
where $g=R(\f,\h,\psi)=(e^{i\psi}\hat{\h}(p),e^{i\psi}\hat{\f}(p),p)$ is interpreted as in \eqref{eq:so3euler}. Thus, the transformation rule for $z$ (i.e. for local sections of $\xi_s$), when we pass from $\psi=\psi_1$ to $\psi=\psi_2$ is
\be 
z_1e^{is\psi_1}=z_2e^{is\psi_2}.
\ee
The local representation of a section $\sigma\colon \mathbb S^2\to SO(3)\times_s \mathbb C$ with respect to the trivialization given by $\psi$ is  then a function $f_\psi(\h,\f)$ defined by the expression
\bega
\sigma(p)&=f_\psi(\h,\f)\left(e^{i\psi}\hat{\f}\right)^{\otimes s} 
\\
\text{i.e.}\quad f_\psi(\h,\f)&=F_\sigma\left(R(\f,\h,\psi)\right)=F_\sigma\left(R(\f,\h,0)\right)e^{-is\psi}=f_0(\h,\f)e^{-is\psi}.
\eega
\begin{remark}
The same reasoning applies for $\psi$ a function $\psi=\psi(\h,\f)$.
\end{remark}
\begin{prop}\label{prop:smoothatpole}
A section $\sigma$ of $\xi_s$ can be defined (almost everywhere) by its local expression, i.e. by specifying the function
\be 
f_0\colon (0,\pi)\times (0,2\pi)\to \C
\ee
In this case, the section is continuous if and only if: $f_0$ is continuous and
\be 
\sigma( e_3)=\left(\lim_{\h\to 0^+}f_0(\h,\f)e^{is\f}\right)(e_2)^{\otimes s}, \quad \sigma(- e_3)=\left(\lim_{\h\to \pi^-}f_0(\h,\f)e^{is\f}\right)(e_2)^{\otimes s},
\ee
uniformly with respect to $\f$.
\end{prop}\label{prop:localspindef}
\begin{proof}
See  \cite[Theorem 3.1]{geller2008spin}.
\end{proof}

\subsection{Spectral representation of spin sections}

By the Peter-Weyl theorem (see \cite[p. 288]{libro}), any function $f \in L^2(SO(3))$ with right spin $-s$, 
i.e. corresponding to a section of $\xi_s$ in the sense of Theorem \ref{thm:pullback}, can be represented by a series (convergent in $L^2(SO(3))$) of the form
\be 
f(g)=\sum_{\ell\ge s}\sum _{m}b^\ell_{ms}D_{ms}^\ell(g),
\ee
where $D^\ell_{ms}(g(\f,\h,\psi))=e^{-im\f}d^\ell_{ms}(\h)e^{-is\psi}$ is the $(m,s)$ entrance of the $\ell^{th}$ Wigner $D$ matrix, see \cite{libro}, and 
\begin{equation}\label{fourier_coeff}
b^\ell_{ms} := \int_{SO(3)} f(g) \overline{D}^\ell_{ms}(g)\,dg
\end{equation}
are the Fourier coefficients of $f$. 
Therefore the section $\sigma$ associated to $f$ is determined (on $\mathbb S^2\-\{\pm e_3\}$) by the series
\be 
f_0(\h,\f)=\sum_{\ell\ge s}\sum _{m}b^\ell_{ms}e^{-im\f}d_{ms}^\ell(\h).
\ee

\begin{defi}\label{def:spharmo}
The $m$-th \emph{spin $s$ spherical harmonic} of degree $\ell$  $\sigma_{\ell;ms}\colon \mathbb S^2\to SO(3)\times_s \mathbb C$ (see \cite[p. 289]{libro}) is the section with 
\be 
f_0(\h,\f)=\sigma_{\ell;ms}(\h,\f)=\sqrt{\frac{2\ell+1}{4\pi}}\overline{D}_{m,-s}^\ell(R(\f,\h,0))=\sqrt{\frac{2\ell+1}{4\pi}}e^{im\f}d^\ell_{m,-s}(\h).
\ee  
\end{defi}
Its pullback function $f_{\ell;ms}\colon SO(3)\to\C$, with right spin $-s$, is
\be 
f_{\ell;ms}(g)=\sqrt{\frac{2\ell+1}{4\pi}}\overline{D}_{m,-s}^\ell(g).
\ee
Note that 
$
\lbrace \sigma_{\ell;ms}, m=-\ell,\dots, \ell, \ell \ge |s|\rbrace
$, the set of spin $s$ spherical harmonics of degree $\ell\ge |s|$, 
is an orthonormal basis for the space of square integrable spin $s$ sections. 

\section{Spin random fields}\label{sec:SpinRandomFields}

In this Section we define and study basic properties of so-called spin random fields, which are random sections of the spin line bundles on the sphere introduced in Section \ref{sec_bundles}, focusing on their spectral representation. Let us fix once for all a probability space $(\Omega, \mathcal F, \mathbb P)$. 

\subsection{Random sections}

\begin{defi}
A \emph{spin $s$ random field} $U$ is a random section of the spin $s$ line bundle $\xi_s$, i.e. a measurable map 
\begin{equation}\label{def_spinfield}
    U : \Omega \times \mathbb S^2 \to SO(3)\times_s \mathbb C
\end{equation}
such that, for every $\omega\in \Omega$, $U(\omega, \cdot)$ is a section of $\xi_s$, i.e.  $\pi_s(U(\omega, \cdot))=\text{id}_{\mathbb S^2}(\cdot)$ for every $\omega\in \Omega$, where $\text{id}_{\mathbb S^2}$ denotes the identity function on the sphere. 
\end{defi}
To be more precise, in (\ref{def_spinfield}) we consider the $\sigma$-field $\mathcal F \otimes \mathcal B(\mathbb S^2)$ on $\Omega\times \mathbb S^2$. 
Plainly, in light of Remark \ref{rem_pullback}, there is a one to one correspondence between spin $s$ random 
fields $U$ and complex-valued random fields $X$ on $SO(3)$ of type $s$, that is, measurable maps $X:\Omega \times SO(3)\to \mathbb C$ whose sample paths are functions of right spin $-s$, i.e. 
 for every $\omega\in \Omega$, every $g\in SO(3)$ and every $k\in K$,
\begin{equation}
    X(\omega, gk) = \chi_s(k^{-1}) X(\omega,g).
\end{equation}
For the sake of brevity we omit the dependence on $\omega$ from now on. We call $X$ the \emph{pullback random field} of $U$, as in \cite{BR13}, where this ``pullback approach'' for spin random fields was first developed.

The Fourier coefficients $b^\ell_{ms}$ of $X$ as defined in (\ref{fourier_coeff}) are random variables and, if $X\in L^2(SO(3))$ (a.s.), then
Peter-Weyl Theorem applies pathwise (up to a negligible set of trajectories), so that (a.s.) in $L^2(SO(3))$ we have the following spectral representation
\begin{equation}
    X(g) = \sum_{\ell\ge s}\sum _{m}b^\ell_{ms}D_{ms}^\ell(g).
\end{equation}

\subsection{Isotropy and Gaussianity}


Assume that $U$ is a.s. square integrable, then the inner product 
\begin{equation}
    U(h) := \int_{\mathbb S^2} \langle U(p), h(p) \rangle_{\pi_s^{-1}(p)}\,dp 
\end{equation}
is well defined for every square integrable spin $s$ section $h$. We say that $U$ is Gaussian if the vector $(U(h_1), \dots, U(h_n))$ is Gaussian for any finite number of square integrable spin $s$ sections $h_1, \dots, h_n$. Hence $U$ is Gaussian if and only if $X$ is Gaussian, seen as a random variable taking values in $L^2(SO(3))$. 

Of course, 
if the spin $s$ random field $U: \mathbb S^2\to SO(3)\times_s \mathbb C$ is a.s. continuous (as we shall always assume), then it is Gaussian if and only if its pullback random field $X\colon SO(3) \to \C$ is Gaussian, namely if and only if the random vector $(X(g_1),\dots,X(g_n))\in\C^n$ is complex Gaussian for any finite number of points $g_1,\dots,g_n\in SO(3)$. 

We will restrict to the case of circularly symmetric complex Gaussian random vectors: $\gamma\sim N_\C(0,K)$ that is:
\be 
\E[\gamma]=0; \quad \E[\gamma\overline{\gamma}^T]=K, \quad \text{and} \quad \E[\gamma\gamma^T]=0.
\ee
\begin{remark}
Given a random field $f\colon A\to B$, let us denote by $[f]$ its class up to equivalence of fields, namely $[f]$ is the probability measure induced on the space $B^A$ of functions from $A$ to $B$, endowed with the product $\sigma-$algebra. Notice that the correspondence $[\sigma]\mapsto [X]$ is a bijection, since there is a $\C-$linear isomorphism of vector spaces:
\be \label{eq:bije}
\{F\colon SO(3)\to \C| \text{ $F$ has right spin$=-s$}\}=\{\sigma\colon S^2\to \spi{s}| \text{ $\sigma$ is a section}\},
\ee 
and this induces a bijection on the space of probability measures on those spaces. Moreover, by linearity, this bijection sends Gaussian measures to Gaussian measures. Even more, one can easily see that the bijection \eqref{eq:bije}, when restricted to $\mathcal{C}^r$ functions/sections, is a homeomorphism with respect to the $\mathcal{C}^r$ topologies, for all $r\in\N\cup\{+\infty\}$. In other words, it is completely equivalent to define a ($\mathcal{C}^r$ and/or Gaussian) random section of $\spi{s}$ or a ($\mathcal{C}^r$ and/or Gaussian) random function $X\colon SO(3)\to\C$ with right spin$=-s$.
\end{remark}
\begin{defi}
We say that $\sigma$ is \emph{isotropic} if and only if $X$ is isotropic on the left:
\be 
X(g\cdot)\sim X(\cdot), \quad \text{for any $g\in SO(3)$.}
\ee
\end{defi}
In other words, the random section $\sigma$ is isotropic if $g_*(\sigma(\cdot))\sim \sigma(g\cdot)$, for every $g\in SO(3)$, where 
\be
g_*\colon \spi{s}_p\to \spi{s}_{gp}, \quad g_*(zv^{\otimes s})=z(gv)^{\otimes s}
\ee
In terms of the covariance function of $X$, we have the following characterization.
\begin{remark}
Recall that $X(gR_3(\psi))=X(g)e^{-is\psi}$ for all $\psi$.
\end{remark}
\begin{prop}\label{prop:Gamma}
$\sigma$ is isotropic if and only if there exists a function $\Gamma\colon SO(3)\to \C$ such that
\be 
\textrm{Cov}\left(X(g),X(h)\right)=\E\{X(g)\overline{X(h)}\}=\Gamma(g^{-1}h).
\ee
Moreover, $\Gamma$ has right spin$=s$ and left spin$=s$:
\be 
\Gamma\left(R_3(\f)hR_3(\psi)\right)=e^{is\f}\Gamma(h)e^{is\psi},
\ee
and is ``hermitian'',
\be 
\Gamma(g^{-1})=\overline{\Gamma(g)}.
\ee
\end{prop}
\begin{proof}
Define $\Gamma(h)=\E\{X(\mathbb{1})\overline{X(h)}\}$. The rest is straighfrward.
\end{proof}
It follows that the whole random structure of an isotropic  Gaussian section $\sigma$ of $\spi{s}$ is determined by the function $k(\h)=\Gamma(R_2(\h))$, indeed
\be\label{eq:Gappa}
\Gamma(R(\f,\h,\psi))=k(\h)e^{is(\f+\psi)}.
\ee
\begin{prop}
The function $k(\h)= \E\left\{X(\mathbb{1})\overline{X(R_2(\h))}^T\right\}$ has the following properties
\begin{enumerate}
\item $k\colon \R\to \R$ is $2\pi$ periodic and even.
\item $k$ is semipositive definite:
\be 
\sum_{i,j}k(\h_i-\h_j)z_i\overline{z_j}\ge 0 \quad \text{for any $z_1,\dots,z_n\in\C$ and $\h_1,\dots,\h_n\in \R$}.
\ee
\end{enumerate}
\end{prop}
\begin{proof}
Of course $k\colon \R\to\C$ is $2\pi-$periodic. Let $\h\in [0,\pi]$, then the Euler coordinates of $R_2(-\h)$ are given by:
\bega
R_2(-\h)&=\begin{pmatrix}
\cos\h &0&-\sin\h \\
0& 1& 0\\
\sin\h &0&\cos\h
\end{pmatrix}
\\
&=\begin{pmatrix}
-1 &0&0 \\
0& -1& 0\\
0 &0&1
\end{pmatrix}\begin{pmatrix}
\cos\h &0&\sin\h \\
0& 1& 0\\
-\sin\h &0&\cos\h
\end{pmatrix}\begin{pmatrix}
-1 &0&0 \\
0& -1& 0\\
0 &0&1
\end{pmatrix}
\\
&=R_3(\pi)R_2(\h)R_3(\pi)=R(\pi,\h,\pi).
\eega
thus $k$ is even
\be 
k(-\h)=\Gamma(R_2(-\h))=\Gamma(R(\pi,\h,\pi))=e^{is\pi}k(\h)e^{is\pi}=k(\h),
\ee
and real
\be 
\overline{k(\h)}=\overline{k(-\h)}=\overline{\Gamma(R_2(\h)^{-1})}=\Gamma(R_2(\h))=k(\h).
\ee
Positive definiteness follows from the fact that $k$ is the covariance function of the stationary Gaussian random field $\tau\colon \R\to\C$ defined by
$
\tau(\h)=X(R_2(\h)).
$
\end{proof}
\begin{remark}
We leave as an open issue whether (1) and (2) are enough to classify all functions $k$ coming from an isotropic Gaussian spin $s$ section.
\end{remark}
\begin{remark}\label{rem:dakedepende}
As we can see from equation \eqref{eq:Gappa}, given $p,q\in S^2$, the covariance of $\sigma(p)$ and $\sigma(q)$ does not depend only on the angular distance between $p,q$, i.e. on $\h=\arccos(\langle p, q\rangle)$. Indeed, if $p=ge_3$ and $q=gR(\f,\h,\psi)e_3$, then
\be 
\sigma(p)=X(g)\left(ge_2\right)^{\otimes s}, \quad \sigma(q)=X(gR(\f,\h,\psi))\left(gR(\f,\h,\psi)e_2\right)^{\otimes s}
\ee 
and
\be 
\E\{X(g)\overline{X(gR(\f,\h,\psi))}^T\}=\Gamma(R(\f,\h,\psi))=k(\h)e^{i(\f+\psi)}.
\ee
\end{remark}
\begin{example}\label{ex:spharmoel}
We will work with sections of the form 
\be\label{eq:spharmoel}
\sigma=\frac{1}{\sqrt{\frac{2\ell+1}{4\pi}}}\sigma_{\ell;s}:=\sum_{m=-\ell}^\ell a_{\ell m;s}\sigma_{\ell m;s},
\ee
where $\sigma_{\ell m;s}$ are the spin spherical harmonics defined in Definition \ref{def:spharmo} and $a_{\ell m;s}\sim N_\C(0,1)$ are iid. The pullback field is
\be 
X=X_{\ell ;s}=\sum_{m=-\ell}^\ell a_{\ell m;s} \overline{D}_{m,-s}^\ell.
\ee
Therefore 
\be\label{eq:covspharmoel}
k(\h)=k_{\ell;s}(\h)=d_{-s,-s}^\ell(\h).
\ee 
$X$ is isotropic, because $D^\ell\colon SO(3)\to U(2\ell+1)$ and it is a group homomorphism, thus
\bega
X(gR)&=\sum_{m=-\ell}^\ell a_{\ell m;s}\overline{D}_{m,-s}^\ell(gR)=
\\
&=\sum_{m=-\ell}^\ell a_{\ell m;s} \sum_{i=-1}^\ell\overline{D_{m,i}^\ell(g)D_{i,-s}^\ell(R)}=
\\
&=\sum_{m=-\ell}^\ell\left(  \sum_{i=-1}^\ell a_{\ell i;s}
\overline{D_{i,m}^l(g)}\right)\overline{D}_{m,-s}^l(R)=
\\
&=\sum_{m=-\ell }^\ell b_{\ell m;s}\overline{D}_{m,-s}^\ell (R);
\eega
and since $D^\ell (g)$ is unitary, it follows that the random variables $b_{\ell m;s}=\left(  \sum_{i=-1}^\ell a_{\ell i;s}
\overline{D_{i,m}^\ell (g)}\right)$ are again iid $\sim N_\C(0,1)$.
\end{example}

%% file: c_Jet_and_LipKill.tex
\section{Jet bundles and Type-$W$ singularities}\label{sec:jetW}
In this Section we introduce the geometric tools that we are going to exploit to establish our main results. Let us recall that in \cite{mttrps} the authors study the singularities of polynomial maps $\psi\colon S^m\to \R$ that arise as preimages via the jet prolongation map $j^r\psi$ of subsets $W\subset J^r(S^m,\R^k)$ of the jet space. In other words a singularity $Z=(j^r\psi)^{-1}(W)$ is the set of points $p\in S^m$ where the Taylor polynomial of $\psi$ at $p$ satisfies a given set of conditions, encoded in $W$, an obvious example being the set of critical points, or extrema. In this paper we will study the same objects, but replacing $\psi$ with a $N$-tuple of spin functions.
\begin{defi}
For any  $\underline{s}:=s_1,\dots,s_N\in \Z^N$ we define
\be 
 \mathcal{E}^{\us}:=\spi{s_1}\oplus \dots \oplus \spi{s_N}
\ee
that is the complex hermitian vector bundle whose total space and projections are denoted by:
\be 
\pi^{\us} \colon E^{\us}:=(TS^2)^{\otimes s_1}\oplus \dots \oplus (TS^2)^{\otimes {s_N}}\to S^2
\ee
\end{defi}
A section $\Es$ is a $N$-tuple of sections of $\spi{s_i}$, for $i=1,\dots ,N$. For this reason we will call them \emph{multispin functions} and we
will denote them as
\be\label{eq:multispin}
\usig=(\sigma^1,\dots, \sigma^N)\colon S^2\to E^{\us}.
\ee 
We denote by $\mC^r(S^2|\Es)$
the space of all $\mC^r$ sections of $\Es$ and by $J^r(S^2|\Es)$ the space of $r$-jets of sections (we reserve the notation $\mC^r(S^2|\Es)$ for the larger space of all $\mC^r$ functions). 
The $r$-jet at $p$ of a $\mC^r$ section $\usig$ of $\Es$, denoted $j^r_p\usig$, is the equivalence class of all sections that in one (and hence every) trivialization of $\Es$ over a neighborhood of $p$ have the same derivatives at $p$, up to the order $r$. 

The jet is an intrinsic version of the notion of Taylor polynomial, in that it encodes all the properties of the latter which do not depend on the chosen trivialization.

Moreover, the space of all jets
\be
J^r(S^2|\Es):=\bigsqcup_{p\in S^2}J^r_p(S^2|\Es):=\left\{j^r_p\usig | p\in S^2 \text{ and }\sigma \in \mC^r(S^2|\Es)\right\}
\ee
is a smooth vector bundle over $S^2$, with the obvious projection map $j^r_p\usig\mapsto p$, called the \emph{source}.

The point of view of jets allows us to put under the same umbrella any set defined by some conditions on the derivatives of a collections of spin functions. Indeed we will view those as the preimage of a given subset $W\subset J^r(S^2|\Es)$ via the jet prolongation map, that is the map associated to a $\mC^r$ section $\usig\in \mC^r(S^2,\Es)$ that evaluates the jet at each point:
\be 
j^r\usig\colon S^2\to J^r(S^2|\Es), \qquad 
j^r\usig (p):=j^r_p\usig.
\ee

We refer to the books \cite{Hirsch, eliash} for the theory of jet bundles.

\begin{remark}
If $\usig$ is of class $\mC^{k}$, then $j^r\usig$ is of class $\mC^{k-r}$.
\end{remark}
\begin{defi}
Let $W\subset J^r(S^2|\Es)$ be a subset and let $\usig\in\mC^r(S^2|\Es)$ be a multispin function. The \emph{type-$W$ singularity} of $\usig$ is the set
\be 
Z_W(\usig):=\left\{p\in S^2\colon j^r\usig\in W\right\}=(j^r\usig)^{-1}(W).
\ee
We say that $W\subset J^r(S^2|\Es)$ is the \emph{singularity type}.
\end{defi}
\subsubsection{Examples} Obvious examples of singularities are the excursions sets, the critical points and the extrema for the modulus of a given section $\usig\in \mC^r(S^2,\Es)$, see below for more explicit computations.   

\subsubsection{Intrinsic singularity type $W$}\label{subsec:intrising}
We need to restrict the class of subsets $W\subset J^r(S^2|\Es)$ under consideration, in order to say something meaningful. First of all since we are interested in isotropic spin and multispin random functions, it makes sense to restrict ourselves to the class of $W$ that are isotropic in some sense. In the paper \cite{mttrps} the notion of \emph{intrinsic subset} of a jet space was introduced for the same reason. Let us repeat it here, in a version adapted to our case. 
\begin{defi}\label{def:intrisub} [Intrinsic subset]
Let $\underline{s}=(s_1,\dots,s_N)\in\Z^N$. A subset $W\subset J^r(S^2|\Es)$ is said to be \emph{intrinsic} if there is a subset $W_0\subset J^r(\D, \C^N)$, called \emph{model}, such that for any embedding $\f\colon \D\hookrightarrow S^2$ and any metric-preserving trivialization of $\Es$ over $\f(\D)$, namely an isomorphism of vector bundles
\be 
\tau=(\tau_0,\tau_1)\colon \Es|_{\f(\D)} \to \D\times \C^N,
\ee
such that $\pi^{\us}\circ \tau^{-1}(u,z)=\f(u)$ and $|\tau^{-1}(u,z)|=|z|$,
one has that $(j^r\tau)^*(W)=W_0$, where 
\be\label{eq:jettrivialize}
(j^r\tau)^*\colon J^r\left(\f(\D)|\Es \right)\to J^r\left(\D,\C^N\right), \quad j^r_{\f(p)}\usig \mapsto j^r_p(\tau_1\circ\usig\circ \f).
\ee
\end{defi}
The jet space $J^r(\D,\C^N)$ is canonically isomorphic to the product space $\D\times (P_r)^{2N}$, where $P_r\subset \R[x,y]$ denotes the space of real polynomials of degree at most $r$ in two variables (the coordinates on $\D$). Therefore we can make the identification $J^r(\D,\C^N)=\D\times \R^{k}$, where $k=N(r+1)(r+2)$.

The above definition implies that the model $W_0\subset J^r(\D,\C^N)$ is itself an intrinsic singularity type of the form 
\be
W_0=\D\times \Sigma,
\ee
for some $\Sigma \subset \R^k$. 
We will say that a subset $\Sigma\subset \R^k$ is \emph{intrinsic} if the subset $\D\times \Sigma=W_0$ is an intrinsic singularity type.
\begin{remark}
Note that not all subsets $\Sigma \subset \R^k$ are intrinsic. An obvoius counterexample is $W_0=\{j^r_pf\subset J^r(S^2,\C):f(p)=c\}$ for $c\in \R\-\{0\}$, because, of course, the value of $f$ depends from the choice of local coordinates.
\end{remark}

\begin{defi}(Intrinsic function)\label{def:intrifun}
Let $\underline{s}=(s_1,\dots,s_N)\in\Z^N$ and let $W\subset J^r(S^2|\Es)$ be an intrinsic subset with model $W_0=\D\times \Sigma\subset J^r(\D,\C^N)=\D\times \R^k$. Let $\a\colon W\to \R$ be a function. We say that $\a$ is \emph{intrinsic} if there is a function $\a_0\colon W_0\to \R$ such that, under any trivialization of the type described in Equation \eqref{eq:jettrivialize}, we have that $\a$ corresponds to $\a_0$.
\end{defi}
\subsubsection{Semialgebraic singularity type $W$}
We will consider only singularity types $W\subset J^r(S^2|\Es)$ that are intrinsic and for which $\Sigma\subset\R^k$ is \emph{semialgebraic} (see \cite{BCR:98} or \cite{GoreskyMacPherson}). In particular, this implies that $W$ admits a Whitney stratification (see \cite[Sec. 8.1]{AdlerTaylor}, or \cite{Mather70noteson}, or \cite[Sec.1.2]{GoreskyMacPherson}), that is a partition $\mathscr{S}$ of $W$ into a locally finite family of disjoint smooth embedded submanifolds $S\subset J^r(S^2|\Es)$ called the \emph{strata} of the stratification, such that for each $S\in \mathscr{S}$, the set $(\overline{S}-S)\cap W$ is a union of
strata (this is known as the \emph{frontier condition}, see \cite{Mather70noteson}) and such that each pair $(X,Y)$ of distinct strata satisfies \emph{Whitney condition B} (cf.\cite{Mather70noteson} or \cite[Sec. 1.1]{GoreskyMacPherson}) whenever $\overline{X}\supset Y$  (Whitney condition B implies that in this case $\dim X>\dim Y$, see \cite{Mather70noteson}).
A \emph{Whitney stratified subset} of a smooth manifold $M$ is a pair $(W,\mathscr{S})$, such that $\mathscr{S}$ is a Whitney stratification of $W\subset M$. However, we will most frequently just say that $W$ is a Whitney stratified subset, without mentioning the stratification. 

It follows that a Whitney stratified subset $W$ admits a partition (depending on the stratification) $W=\de_0W\sqcup \de_1W\sqcup \de_2W\sqcup\dots$, where 
\be 
\de_iW:=\bigsqcup_{S\in\mathscr{S},\ \dim S=i}S
\ee
is a smooth embedded submanifold (because it is a \emph{locally finite} union of smooth embedded submanifolds) of dimension $i$. Here, we are using the notation of \cite{AdlerTaylor}.
\subsubsection{Transversality}
Let $W\subset J^r(S^2|\Es)$ be a Whitney stratified subset. Let $\usig\in \mC^\infty(S^2|\Es)$ be a smooth multispin function (see Equation \eqref{eq:multispin} above). Then the jet map $j^r\usig\colon S^2\to J^r(S^2|\Es)$ is \emph{transverse} to $W$ if and only if it is transverse to $S$ for each stratum $S\in\mathscr{S}$ of $W$, see \cite{Hirsch}. This is denoted 
\be \label{eq:transversality}
j^r\usig\transv W.
\ee
Moreover, in this case we say that the singularity $Z_W(\usig)$ is \emph{nondegenerate}.
For instance, consider the singularity type 
\be W:=\{j^1_p\sigma\in J^1(S^2|\spi{s})| \langle\sigma(p),(\nabla \sigma)_p\rangle=0\},
\ee
then $Z_W(\sigma)= \text{Crit}(|\sigma|^2)$. In this case $j^1\sigma \transv W$ if and only if the function $|\sigma|^2\colon S^2\to \R$ is Morse.

\begin{remark}
There is a little abuse of notation in Equation \eqref{eq:transversality} in that the transversality condition depends also on the chosen stratification.
\end{remark}
By classical arguments of differential topology (see \cite{Mather70noteson,GoreskyMacPherson}), if the singularity $Z_W(\usig)$ is nondegenerate, it follows that $Z_W(\usig)\subset S^2$ admits a Whitney stratification $(j^r\usig)^{-1}\mathscr{S}$ obtained by taking the preimages of the strata of $W$ having codimension smaller or equal to two.
\begin{prop}
Let $W\subset J^r(S^2|\Es)$ be a semialgebraic subset with a given Whitney stratification $\mathscr{S}$. Let $\usig\in \mC^\infty(S^2|\Es)$ be a smooth multispin function such that $Z_W(\usig)$ is nondegenerate. Then the set $(j^r\usig)^{-1}\mathscr{S}$ of all subsets $(j^r\usig)^{-1}(S)$ with $S\in\mathscr{S}$ is a Whitney stratification of $Z_{W}(\usig)\subset S^2$. Moreover, if $\de_kW$ has codimension $2$, then
\be 
\de_{i}Z_W(\usig)=Z_{\de_{k+i}W}(\usig)
\ee 
is a smooth embedded submanifold of dimension $i\in \{0,1,2\}$ in $S^2$ and it is a nondegenerate singularity of $\usig$ of type $W_{k+i}$, so that 
\be\label{eq:stratopartition} 
Z_W(\usig)=Z_{W_{k}}(\usig)\sqcup Z_{W_{k+1}}(\usig)\sqcup Z_{W_{k+2}}(\usig).
\ee 
\end{prop}
\begin{remark}
The decomposition \eqref{eq:stratopartition} does not need to be a Whitney stratification. For instance if $Z$ has an isolated point, then it belongs to $\de_0Z$, but not to $\overline{\de_1Z}$; this would violate the frontier condition.
\end{remark}
\subsubsection{Whitney stratified subsets of the sphere}
It is worth to spell out explicitely the definition of a closed Whitney stratified subset of $S^2$, since it is actually quite simple.
\begin{prop}\label{prop:stratosphere}
Let $Z\subset S^2$ be a closed semialgebraic nondegenerate singularity with a partition $Z=Z_0\sqcup Z_1\sqcup Z_2$  as in \eqref{eq:stratopartition}. Then the following conditions hold.
\begin{enumerate}[i.]
\item $Z_0=\{p_1,\dots,p_{n_0}\}$ is a finite set.
\item $Z_1=\xi_1(\R)\sqcup\dots \sqcup \xi_{n_1}(\R)\sqcup \gamma_1(S^1)\sqcup\dots \sqcup \gamma_{n_1'}(S^1)$, where $\xi_i\colon \R\to S^2$ and $\gamma_i\colon S^1\to S^2$ are smooth embeddings with pairwise disjoint image.
\item $Z_2=U_1\sqcup \dots \sqcup U_{n_2}$ is a union of open connected components of $S^2\-(Z_0\sqcup Z_1)$.
\item $\exists \lim_{t\to \pm \infty}\xi_i(t)\in Z_0.$
\item Let $\lim_{t\to \pm\infty }\xi_i(t)=y$. Then the following limits exist and they are equal: \be 
\exists \lim_{t\to\pm\infty}\frac{y-\xi(t)}{|y-\xi(t)|}=\lim_{t\to\pm\infty}\frac{\xi'(t)}{|\xi'(t)|}.
\ee
\end{enumerate}
\end{prop}
\begin{proof}
The proof follows from standard techniques in semialgebraic geometry and is omitted for brevity's sake.
\end{proof}
\begin{remark}
If we remove the closedness assumption, none of the above property has to hold. Moreover, property $v.$ is due to the semialgebraicity of $Z$.
\end{remark}
\subsubsection{Euler-Poincaré characteristic}
\begin{defi}
The $i^{th}$ Betti number of a topological space $Z$ is the dimension of the $i^{th}$ real homology group (see \cite{Hatcher}):
\be 
b_i(Z):=\dim H_i(Z,\R).
\ee
We denote by $b(Z)\in \N\cup \{+\infty\}$ the sum of all Betti numbers.
The Euler-Poincaré characteristic of a topological space $X$ is defined whenever $b(Z)<+\infty$ and it is the alternating sum of all Betti numbers:
\be 
\chi(Z)=b_0(Z)-b_1(Z)+b_2(Z)-\dots
\ee
\end{defi}
By standard arguments, we see that if $Z\subset S^2$ is a Whitney stratified subset of $S^2$, the only non-zero Betti numbers are $b_0$ (which is the number of connected components), $b_1$ and $b_2$. Moreover, $b_2(Z)$ is non-zero only if $Z=S^2$, case in which $b_0=1$, $b_1=0$ and $b_2=1$.

When $Z$ is closed, one can give it the structure of a CW-complex, using the description given in Proposition \ref{prop:stratosphere} (after passing to a finer stratification, possibly), so that the Euler-Poincaré characteristic can also be expressed as follows.
\begin{prop}
Let $Z\subset S^2$ be a closed Whitney stratified subset and let $n_0,n_1,n_2$ be as in Proposition \ref{prop:stratosphere}. Then
\be 
\chi(Z)=n_0-n_1+n_2-b_1(Z_2),
\ee
unless $Z=S^2$ with the trivial stratification: $Z_0=\emptyset, Z_1=\emptyset$ and $Z_2=S^2$.
More generally, let $A\cup B=Z$ such that $A,B$ and $ A\cap B$ are closed and are union of connected components of strata of $Z$. Then
\be 
\chi(Z)=\chi(A)+\chi(B)-\chi(A\cap B).
\ee
\end{prop}
\section{Lipschitz-Killing curvatures }\label{sec:LKcurv}
This Section collects some basic definitions and properties of intrinsic volumes/Lipschitz-Killing Curvatures; the presentation is tailored for the main results and proofs to follow in the remaining part of the paper.

\subsection{Normal Morse index}
Let $Z\subset S^2$ be a Whitney stratified subset. For any $p\in \de_{0} Z$, we define the set of degenerate covectors (cf. \cite{GoreskyMacPherson}) at $p$ as the set
\be 
D_pZ:=\left\{\nu\in T^*_pS^2\colon 
\begin{aligned}
&\exists p_n\in \de_i Z\smallsetminus\de_0 Z\text{ s.t. } p_n\to p \text{ and }
\\&\exists \lim_n T_{p_n}Z_i=Q\subset T_{p}M  \text{ s.t. } \nu(Q)=0
\end{aligned}\right\}.
\ee
In particular if $p\in \de_0 Z\smallsetminus \overline{Z_1\cup Z_2}$, then $D_pZ=\emptyset$. Moreover, we define $D_pZ=\emptyset$ for every $p\in Z_2$; $D_pZ=\{0\}$ for every $p\in Z_1\cap \overline{Z_2}$ and $D_pZ=\emptyset$ for $p\in Z_1\smallsetminus \overline{Z_2}$. We leave to the reader to check that this definition of \emph{degenerate covectors} corresponds to the general one from \cite[Section 1.8]{GoreskyMacPherson} in the special case of stratified subsets of the sphere.
\begin{defi}
Let us consider $p\in Z_i$, and
$v\in T_p^{\perp}Z_i\setminus D_pZ$, we define the (normal) Morse index  
\begin{equation}
    \alpha(p,v) := 1 - \chi(O_{p}\cap Z\cap\f^{-1}\lbrace x: \langle x,v\rangle \le -\epsilon \rbrace),
\end{equation}
where $\f:O_p\to \mathbb \D\subset \R^2$ is a coordinate chart centered at $p$ (i.e. a diffeomorphism such that $\f(p)=0$.). 
\end{defi}
Due to the cone structure of Whitney stratified subsets, $\f^{-1}\lbrace x : \langle x,v\rangle < 0\rbrace$ can be retracted homotopically to a subset of the boundary of $O_p$ which in our case is a finite union of intervals, hence the Euler-Poincar\'e characteristic is the number of connected components hence
\begin{equation}
    \alpha(p,v) = 1 - b_0(O_{p}\cap Z\cap\f^{-1}\lbrace x: \langle x,v\rangle \le -\epsilon \rbrace).
\end{equation}
\begin{remark}
In \cite[Equation 8.1.1]{AdlerTaylor} the definition of $\chi(T)$ is given with a sign that depends on the dimension of $T$:
\be 
\chi_{AT}(T)=(-1)^{\dim T}\chi(T).
\ee This is not in agreement with the most standard conventions (in topology), indeed with this defnition $\chi$ would not be invariant under homotopy equivalences, because the dimension is not.
\end{remark}
\subsection{Esplicit formula for stratified subsets of the sphere}\label{sec:EpliLK}
Let $Z\subset S^2$ be a closed semialgebraic nondegenerate singularity with a partition $Z=Z_0\sqcup Z_1\sqcup Z_2$  as in \eqref{eq:stratopartition}. Now we define Lipschitz-Killing curvature measures as in \cite[(10.7.1)]{AdlerTaylor}: in our setting the formula becomes, for any $A\subset S^2$ Borel subset:
\begin{eqnarray}\label{eq:LK}
\mathcal L_2(Z,A) &=& \mathscr{H}^2(A\cap Z_2)\\
\mathcal L_1(Z,A) &=& \frac{1}{2}\int_{A\cap Z_1}\beta^1(p)\,dZ_1\\
\mathcal L_0(Z,A) &=& \frac{1}{2\pi} \sum_{p\in  A\cap Z_0} \beta^0(p) +\frac{1}{2\pi} \int_{A\cap  Z_1\cap\{\beta^1=1\}}S(p)dZ_1(p)+
\\
 & &+\frac{1}{2\pi}\mathscr{H}^2(A\cap Z_2),
\end{eqnarray}
where for any point $p\in Z_j$, with $j\in\{0,1\}$, we define $S(T_pZ_j^\perp):=\{v\in T_pS^2: |v|=1, T_pZ_j\subset v^\perp  \}$,
\be 
\beta^j(p):=\int_{S(T_pZ_j^\perp)}\a(p,v)\mathscr{H}^{1-j}(dv),
\ee
and $S(p)$ is the geodesic curvature of $Z_1$ at $p$, see Equation \eqref{eq:geodesicurv} below.
\begin{remark}
{The only term that is specific to the round sphere is the last summand in the formula for $\mathcal{L}_0(Z)$. To have a formula that is valid on every Riemannian surface, one should replace it with
\be 
\mathcal L_0(Z) = \frac{1}{2\pi} \sum_{p\in Z_0} \beta^0(p) +\frac{1}{2\pi} \int_{Z_1\cap\{\beta^1=1\}}S(p)dZ_1(p)
+\frac{1}{2\pi}\int_{Z_2}\kappa(p)dZ_2(p),
\ee
where $\kappa(p)$ is the Gaussian curvature.}
\end{remark}
The intrinsic volumes or Lipschitz-Killing curvatures of $Z$ are then defined as $\mathcal{L}_i(Z):=\mathcal{L}_i(Z,Z)$.
\begin{thm}[Chern-Gauss-Bonnet, {\cite[Theorem 12.6.1]{AdlerTaylor}}]
$\mathcal{L}_0(Z)=\chi(Z)$.
\end{thm}
\subsubsection{Description of $\beta^1$}
Concerning the strata of dimension $1$, there are three possibilities: let us define for $p\in Z_1$
\begin{equation}\label{eq:beta1}
    \beta^1(p) = 2-\# \lbrace \text{strata of dimension 2 adjacent to $Z_1$ at $p$} \rbrace;
\end{equation}
since we are on the sphere that is two-dimensional, $\beta^1(p)\in \lbrace 0,1,2\rbrace$.

Let $T_pZ:=T_pS$, where $S$ is the stratum of $Z$, containing $p\in Z$ and let us define the tangent cone of $Z$ at $p$ as the set 
\be 
\hat T_pZ:=\{v\in T_pS^2\colon \exists \text{ $\mC^1$ curve $\gamma\colon [0,\e)\to Z$ s.t. $\gamma(0)=p,\dot\gamma(0)=v$}\}.
\ee
Let $p\in Z_1$ be a boundary point: $\beta^1(p)=1$. Then we define $S(p)$ as the geodesic curvature of $Z_1$ at the point $p$ in the inward direction $v\in S(T_pZ_1^\perp)\cap \hat T_pZ$:
\be 
S(p):=\langle\nabla_{\dot{\gamma(0)}}\dot{\gamma},v \rangle,
\ee
where $\gamma$ is any $\mC^1$ curve parametrizing $Z_1$ such that $\gamma(0)=p$ and $|\dot\gamma|=1$.
\subsubsection{Description of $\beta^0$}
For $p\in Z_0$, the quantity $\beta^0(p)$ can be understood as follows.
We call \emph{tangent link} of $Z$ at $p$ the subset $S(\hat T_pZ)$ of unit vectors in the tangent cone. Moreover, define the \emph{link} of $Z$ at $p$, denoted $\mathrm{link}_p(Z)$ as the topological space $\de O_p\cap Z$, for a small enough spherical ball $O_p$ around $p$. 
\begin{remark}
The link and the tangent link are both homeomorphic to a finite union of intervals, but they are not necessarily homeomorphic nor homotopic to each other. The only characterization of them as a pair of spaces is that there is a surjective continuous map $\mathrm{link}_p(Z)\to S(\hat T_pZ)$. Indeed the intervals of the link could become points in the tangent link. Moreover, the tangent link may have less connected components than the link. This is due to the existence of semialgebraic cusps.
\end{remark}
\begin{prop}
\be 
\beta^0(p)=2\pi-\mathscr{H}^1\left(S(\hat T_pZ)\right)-\chi\left(\mathrm{link}_p(Z)\right)\pi.
\ee
\end{prop}
\begin{proof}
Let $N:=b_0(\mathrm{link}_p(Z)$ be the number of connected components of the link, then $\de O_p\cap Z$ contracts homotopically to a set of $N$ points, or it is homeomorphic to $S^1$. Only in this latter case, calling $\e:=b_1(\mathrm{link}_p(Z))$, we have that $\e=1$, while $\e=0$ otherwise. 
The formula to prove is:
\begin{equation}\label{eq:geodesicurv}
    \beta^0(p) 
    =2\pi - (N-\e)\pi -\mathscr{H}^1(S(\hat T_pZ)).
\end{equation}
Assume $\e=0$. Let $C_1,\dots, C_N$ be the connected components of $ (O_p\smallsetminus\{p\})\cap Z$, for small enough $\e>0$ and small enough $O_p$.
For each $i$, consider the subset $I_i\subset \hat T_pZ$ that comes from $C_i$, that is the subset consisting of those $\dot\gamma(0)\in S(\hat T_pZ)$, such that $\gamma((0,\e])\subset C_i$.
Now, let $\h_i:=\vol_1(I_i)$ be the total angle spanned by $I_i$. Then
\bega
    \beta^0(p) &=
    \int_{S(T_pS^2)}1-\sum_{i=1}^N\chi(C_i\cap \f^{-1}\{x\colon\langle x,v\rangle<-\e \})\mathscr{H}^1(dv)
    \\
    &=
    2\pi-2N\pi+\sum_{i=1}^N\int_{S(T_pS^2)}1-\chi(C_i\cap \f^{-1}\{x\colon\langle x,v\rangle<-\e \})\mathscr{H}^1(dv)
    \\
    &=2\pi -2N\pi + \sum_{i=1}^N(\pi-\h_i)\\
    &=2\pi - N\pi -\mathscr{H}^1(S(T_pZ))
\eega
Notice that there might be distinct indices $i,j$ for which $I_i\subset I_j$ or even $I_i=I_j$. However, in any case we have that $\h_1+\dots + \h_N=\vol^1(S(T_pZ))$.
If $\e=1$, hence $N=1$, then both links are homeomorphic to $S^1$, thus $\mathscr{H}^1(S(T_pZ))=2\pi$. In this case $\a(p,v)=0$ for all $v$, therefore $\beta^0(p)=0=2\pi-(1-1)\pi-2\pi$.
\end{proof}
\subsection{Stratified Morse theory}
We will make extensive use of the stratified version of Morse theory, for which we refer to the standard textbook by Goresky and Macpherson \cite{GoreskyMacPherson}. The most important result for our purposes are the stratified and probabilistic versions of Morse Theorem and of the Gauss-Bonnet Theorem from the book \cite{AdlerTaylor}, in which a large portion of stratified Morse theory is reported, including most of the results that we will need here. 

The following is the definition of a stratified Morse function specialized to our case. 
 \begin{defi}\label{def:mors}
Given a closed Whitney stratified subset $Z=\de_0 Z\sqcup \de_1 Z\sqcup \de_2 Z$ of $S^2$, we say that a function $f\colon Z\to \R$ is a \emph{Morse function}
if $f$ is the restriction of a smooth function $\tilde{f}\colon S^2\to \R$ such that
\begin{enumerate}[(a)]
\item $f|_{\de_i Z}$ is a Morse function on $\de_iZ$, for all $i=0,1,2$. A point $p\in \de_iZ$ is \emph{critical point} of $f|_Z$ if and only if $p$ is a critical point of $f|_{\de_i Z}$. All points of $\de_0 Z$ are critical points, by convention. The set of critical points is denoted by
\be 
\crit\left(f|_Z\right)= \crit\left(f|_{\de_2 Z}\right)\sqcup \crit\left(f|_{\de_1 Z}\right)\sqcup \de_0Z
\ee
\item For every critical point $p\in \crit$ we have $d_pf\notin D_pZ$, i.e. the covector $d_pf$ is nondegenerate.
\end{enumerate}
If $f|_Z$ is a Morse function and $p\in\crit(f)\cap \de_i Z$, we define the \emph{index of $f$ at $p$}, denoted as $\iota_pf\in\N$, as the index of $f|_{\de Z_i}$, that is the dimension of the negative eigenspace of the second derivative $d^2_p(f|_{\de_i Z})$.
\end{defi}
\begin{thm}[Morse Theorem, see
{\cite[Theorem 9.3.2]{AdlerTaylor} or \cite{GoreskyMacPherson}}]
Let $f|_{Z}$ be a Morse function, then
\begin{equation}
    \chi(Z) = \sum_{p\in \text{Crit}(f)} \alpha(p, d_pf) (-1)^{\iota_pf}.
\end{equation}

\end{thm}
\subsubsection{Semialgebraic Morse inequalities}
We will need the following specialization of \cite[Theorem 8]{mttrps}, incorporating also \cite[Remark 12]{mttrps}, to our setting. The following theorem will be central in the proof Theorem \ref{thm:mainE} in that it allows to reduce it to the case of zero dimensional singularities, hence to apply effectively a generalized Kac-Rice formula (developed by one of the authors in \cite{stecconi2021kacrice}). See \ref{step2} and \ref{step3}.
\begin{thm}[See {\cite[Theorem 8]{mttrps}}]\label{thm:strat}
Let $W\subset J^r(S^2|\Es)$ be an intrinsic semialgebraic subset, with a given semialgebraic Whitney stratification $W=\sqcup_{S\in \mathscr{S}}S$.
There exists an intrinsic semialgebraic subset $W'\subset J^{r+1}(S^2|\C\oplus \Es)$ having codimension $\ge 2$, equipped with a semialgebraic Whitney stratification $\mathscr{S}'$ that satisfies the following properties with respect to any couple formed by a smooth section $\usig\colon S^2\to \Es$ and a smooth function $\sigma_0\colon S^2\to \C$. Let ${\us}':=(0,\us)$ and ${\usig}':=(\sigma_0,\usig)\colon S^2\to {\Es}'=\C\oplus\Es$. Let $g:=\Re (\sigma_0)\colon S^2\to\R$ be the real part of $\sigma_0$. Let $Z_W(\usig)=J^r\usig^{-1}(W)$.
\begin{enumerate}
\item 
If $j^r\usig\transv W$ and $j^{r+1}\usig'\transv W'$, then $g|_{Z_W(\usig)}$ is a Morse function on $Z_W(\usig)$ with respect to the stratification $(j^r\usig)^{-1}\mathscr{S}$ and
\be \label{eq:stratcrit}
\text{Crit}(g|_{Z_W(\usig)})=Z_{W'}(\usig')=(j^{r+1}\usig')^{-1}(W').
\ee
More precisely: if $d_p(j^r\usig)\transv TW$, then $p\in \text{Crit}(g|_{Z_W(\usig)})$ if and only if $j^{r+1}_p(\sigma_0,\usig)\in W''$; moreover, if $j^{r+1}_p(\sigma_0,\usig)\in W'$ and $d_p(j^{r+1}(\sigma_0,\usig))\transv T_{j^{r+1}_p(\sigma_0,\usig)}W'$ then  $d_p(j^r\usig)\transv TW$, thus $Z_W(\usig)$ is a Whitney stratified subset in a neighborhood $O_p$ of $p$, and $p$ is a Morse critical point of $g|_{Z_W(\usig)\cap O_p}$.
\item If $W$ is closed, then $W'$ is closed.
\item
There is a constant $N_W>0$ depending only on $W$ and $\mathscr{S}$, such that if $j^{r}\usig\transv W$ and $j^{r+1}\usig'\transv W'$, then
\be 
b_i\left(Z_W(\usig)\right)\le N_W\#Z_{W'}(\usig')
\ee
for all $i=0,1,2.$
\item There exists a bounded and locally constant and intrinsic function $\a'\colon W''\to \Z$, where
\be W'':=\left\{
\begin{aligned} 
&j^{r+2}_p(\sigma_0,\usig)\in J^{r+2}(S^2|\Es')\colon 
\\
&j^{r+1}_p(\sigma_0,\usig)\in W', d_p(j^{r+1}(\sigma_0,\usig))\transv T_{j^{r+1}_p(\sigma_0,\usig)}W'
\end{aligned}
\right\}
\ee
depending only on $W$ and $\mathscr{S}$, such that if $j^{r}\usig\transv W$ and $j^{r+1}\usig'\transv W'$, then for every $p\in Z_{W'}(\usig')$ we have 
\be 
\a(p,d_pg)(-1)^{\iota_pg}= \a'(j_p^{r+2}\usig'),
\ee
and
\be 
\chi\left(Z_W(\usig)\right)=\sum_{p\in Z_{W'}(\usig')}\a'(j_p^{r+2}\usig').
\ee
\item The stratification $ \mathscr{S}'$ of $\tilde{W}'$ can be taken in such a way that each stratum $S'\in \mathscr{S}'$ is of the form
\be 
S'=\left\{j^{r+1}(\sigma_0,\usig)\in J^{r+1}(S^2|\Es')\colon j^r\usig \in S, \ j^1\Re(\sigma_0)\in \mathcal{U}(j^{r+1}_p\usig)\right\},
\ee
for some stratum $S\in \mathscr{S}$ of $W$ and a family $\{\mathcal{U}(\theta)\}_{\theta\in J^{r+1}(S^2|\Es')}$ of subsets of $J^1(S^2,\R)$.
\end{enumerate}
\end{thm}
\begin{proof}
This result is a natural generalization (from scalar valued functions to sections of vector bundles) of Theorem 8 in \cite{mttrps}. For this reason the proof is omitted; note that the analogous results of points (4) and (5) were not discussed in \cite{mttrps}, however a careful inspection of the proofs reveals easily that these results hold.
\end{proof}
\begin{remark}
Heuristically, the importance of the previous result can be explained as follows: it shows that it is always possible to define on the singularity set a smooth function such that its critical points form a new singularity involving one more derivative and the auxiliary function, but of dimension zero. The power of this construction is that statistics such as Betti numbers or Euler-Poincaré characteristics can now be equivalently reduced to the study of these random  sets with finite cardinality. This trick will be heavily exploited in the sections to follow.
\end{remark}

%% file: d_MainResultsAndProofs.tex
\section{Scaling Assumption}\label{sec:scaass}
\subsection{Scaling assumption for the covariance}
Let $X_n \colon SO(3)\to \C$ be a sequence of smooth isotropic GRFs with right spin$=-s_n$ (possibly dependent on $n$), i.e. pullbacks of isotropic Gaussian sections $\sigma_n $ of $\spi{s_n}$. Let $\Gamma_n\colon SO(3)\to \C$ and $k_n\colon \R\to \R$ be the corresponding \emph{circular covariance functions} (see Section \ref{sec:SpinRandomFields}).

In the following we are going to clarify the assumption of a ``scaling limit'' for $k_n $, see \ref{ass}:
Roughly speaking, this happens if the restrictions of $\sigma_n$ to arbitrary spherical disks of a certain radius $\rho_n>0$ has a limiting behavior. Indeed the sequence:
\be 
\sigma_n\Big|_{B_{\rho_n}}\colon B_{\rho_n}\to \spi{s_n}\Big|_{B_{\rho_n}}
\ee
 can be interpreted as a sequence of GRFs $ 
 \xi_n\colon \D\to \C
 $
  on a fixed disk and the assumption \ref{ass} implies (see Theorem \ref{thm:covaconv} below) that this sequence converges in law to a limit stationary GRF $\xi_\infty\colon \D\to \C$ with covariance function $k_\infty$. For instance, $k_\infty=J_0$ in the Berry case \cite{Berry_1977,Berry_2002,NourdinPeccatiRossi2019}, see Section \ref{sec:nonuniversality}.
\begin{ass}\label{ass}
Let $X_n \colon SO(3)\to \C$ be a sequence of smooth isotropic GRFs with right spin$=-s_n$, i.e. pullbacks of isotropic Gaussian sections $\sigma_n $ of $\spi{s_n}$. Let $\Gamma_n\colon SO(3)\to \C$ and $k_n\colon \R\to \R$ be the corresponding \emph{circular covariance functions} (see Section \ref{sec:SpinRandomFields}). Assume that there exists a sequence of positive real numbers $\rho_n \to 0$ such that
\be 
k_n \left(\rho_n x\right)=k_\infty(x)+\e_n (x),
\ee
with
\be 
\nu_n ^r:=\|\e_n \|_{\mC^r\left(\left[0,\pi\right],\R\right)}\xrightarrow[n \to+\infty]{}0, \quad \forall r\in\N,
\ee
and
\be 
\lim_{n\to+\infty}s_n\rho_n^2=\beta \in\R.
\ee
\end{ass}
\begin{remark}
We may additionally require that: \emph{(Short memory assumption)}
\be 
k_\infty(t)\xrightarrow[t \to+\infty]{} 0.
\ee
\end{remark}
\subsection{The rescaled field}
Let $X_n\colon SO(3)\to \C$ be a sequence of Gaussian random fields with right spin $=-s_n$ that satisfy the Assumption \ref{ass} with respect to the sequence $\rho_n\to 0$ and $\beta\in\R$. Then, given any sequence of spherical balls $B_n$ of radius $\rho_n$, there are trivializations (see \ref{def:rescafield}) of the vector bundle $\spi{s}|_{B_n}\cong \D\times \C$ for which the local representation of $\sigma_n|_{B_n}$ is given by the following Gaussian smooth function on the standard disk.

\begin{defi}\label{def:rescafield}
Let us make the identification  $\R^3=\C\times \R$ and let $\D\subset \C$ be the standard disk. For any $g\in SO(3)$, define
\be 
\phi_\rho ^g\colon \D\to  B_{\rho}(ge_3)\subset S^2
\ee
\be 
z=t e^{i\f}\mapsto g\cdot \begin{pmatrix}
\sin\left(\rho t\right)e^{i\f} \\ \cos\left(\rho t\right)
\end{pmatrix}=\exp_{ge_3}\left(\rho z \cdot ge_1\right)\in S^2
\ee
\end{defi}
Here, $\phi^g_\rho$ is constructed via the Riemannian exponential map $\exp_{ge_3}\colon T_{ge_3}S^2=\{z\cdot ge_1\colon z\in\C\}\to S^2$ precomposed with a rescaling of $\C$. In particular, $\phi^g_\rho(u)=g\cdot \phi^{\mathbb{1}}_1(\rho u)$ and the map $\phi^{\mathbb{1}}_1$ corresponds to the one that we called $p$ in \eqref{eq:triviatrivializacion}: $ 
\phi^\mathbb{1}_\rho\left(t e^{i\f}\right)=p(\f,\rho t,\psi)=R(\f,\rho t,\psi)e_3
$, independently from the value of $\psi\in\R$.

Over the ball $B_\rho(ge_3)=\phi^g_\rho(\D)$, the line bundle $\spi{s}$ have a nonvanishing smooth section $(g\hat\f e^{-i\f})^{\otimes s}$ (see proposition \ref{prop:localspindef}), where $\hat\f(p(\f,\rho t,0))e^{-i\f}=R(\f,\rho t,-\f)e_2$
over the ball $B_\rho(ge_3)$ (see equation \eqref{eq:triviatrivializacion}). This defines a trivialization of the vector bundle $\spi{s}$ over the ball $B=B_\rho(ge_3)\subset S^2$:
\be 
\tau^g_\rho\colon \D
\times \C \xrightarrow{\cong} \spi{s}|_{B}=\{(p;v)|\  p\in B,\ v\in (T_pS^2)^{\otimes s} \}
\ee 
\be 
\tau^g_\rho(z,\xi)=\left(\phi^g_\rho(z);\xi\cdot \left(g\hat\f(\phi^g_\rho(z))e^{-i\f}\right)^{\otimes s}\right)
\ee
It follows that a section $\sigma\in \mC^\infty(S^2|\spi{s})$ has a local representation over the ball $B\subset S^2$ as the function $\xi\colon \D\to \C$ such that 
\bega 
\tau^g_\rho(z,\xi(z))
&=
\sigma(\phi^g_\rho(z))
\\
&=
\left(R(\f,\rho t,-\f)e_3;X(gR(\f,\rho t,-\f)) \cdot \left(gR(\f,\rho t,-\f)e_2\right)^{\otimes s}\right)
\\
&=
\left(\phi^g_\rho(z);X(gR(\f,\rho t,-\f))\cdot  \left(g\hat\f(\phi^g_\rho(z))e^{-i\f}\right)^{\otimes s}\right).
\eega
where the second equality is the very definition of the pull-back correspondence between $\sigma$ and $X$, see Equation \eqref{eq:idspirule}. The above construction justifies the initial discussion and allows us to reduce the local study of $\sigma_n$ to the study of the sequence of Gaussian functions so defined:
\begin{defi}(The rescaled field)\label{defi:rescaled field}
Let $\xi_n \colon \D\to \C$ be the GRF defined, for any $z=te^{i\f}\in\D$, as
\be 
\xi_n (z)
=
X_n \left(R\left(\f,\rho_n t,-\f\right)\right).
\ee
\end{defi}
As it is well known, Gaussian random functions $\xi\colon \D\to \R^2$ are characterized by their covariance function: $K^\R_\xi(z_1,z_2)=\E\{\xi(z_1)\xi(z_2)^T\}$, with values in $\R^{2\times 2}$. In complex notation ($\R^2=\C$) it is useful to observe that $K^\R_\xi(z_1,z_1)$ is determined by the pair of complex numbers $\E\{\xi(z_1)\overline{\xi(z_2)}\}$ and $\E\{\xi(z_1){\xi(z_2)}\}$, and viceversa. In our case the second is always zero, because we are only considering circularly symmetric complex Gaussian fields. Therefore, we will call \emph{covariance function} of a random field $\xi\colon \D\to \C$, the function
\be 
K_\xi\colon \D\times \D \to \C,
\quad 
K_\xi(z_1,z_2)=\E\{\xi(z_1)\overline{\xi(z_2)}\}.
\ee
If $\xi$ is a smooth Gaussian field, then $K_\xi\in\mC^\infty(\D\times \D,\C)$ and the application $\xi\mapsto K_\xi$ is injective. 
\begin{defi}\label{def:scalimit}
Let $\xi_\infty\colon \D\to \C$ be \emph{the} smooth GRF with covariance function
\be 
K_{\xi_\infty}(z_1,z_2)=k_\infty(|z_1-z_2|)\exp\left(\beta i \Im(z_1\overline{z_2})\right).
\ee
\end{defi}
The well-posedness of the above definition is actually a consequence of Theorem \ref{thm:covaconv} below (see also Remark \ref{rem:covaconv}).
\subsection{Smooth convergence of the covariance functions}
\begin{lemma}\label{lemma:trasfrule}
Let $\h_1,\h_2\in [0,\pi),\f\in \R$, $\tilde{\f}\in \R$, $\tilde{\h}\in[0,\pi)$ and $\tilde{\psi}\in [0,2\pi)$ such that
\bega
\cos\left(\frac{\tilde\h}{2}\right)e^{i\frac{\tilde{\f}+\tilde{\psi}}{2}}
&=
\cos\left(\frac{\h_1}{2}\right)\cos\left(\frac{\h_2}{2}\right)e^{i\frac\f2}+\sin\left(\frac{\h_1}{2}\right)\sin\left(\frac{\h_2}{2}\right)e^{-i\frac\f2};
\\
\sin\left(\frac{\tilde\h}{2}\right)e^{i\frac{-\tilde{\f}+\tilde{\psi}}{2}}
&=
\sin\left(\frac{-\h_1}{2}\right)\cos\left(\frac{\h_2}{2}\right)e^{i\frac\f2}+\cos\left(\frac{\h_1}{2}\right)\sin\left(\frac{\h_2}{2}\right)e^{-i\frac\f2}
\eega
Then
\be\label{eq:RR} 
R_2(-\h_1)R_3(\f)R_2(\h_2)=R_3(\tilde{\f})R_2(\tilde{\h})R_3(\tilde{\psi}).
\ee
\end{lemma}
\begin{proof}
Let us lift the equation \eqref{eq:RR} to $SU(2)$, using the convention in \cite[Prop. 17]{stecconi2021isotropic} for the precise definition of the covering $\pi:SU(2)\to SO(3)$. We obtain the equation
\be 
\begin{pmatrix}
\cos(\frac{\h_1}{2}) & \sin(\frac{\h_1}{2}) \\
-\sin(\frac{\h_1}{2}) & \cos(\frac{\h_1}{2})
\end{pmatrix}
\begin{pmatrix}
e^{i\frac{\f}{2}} & 0 \\
0 & e^{-i\frac{\f}{2}}
\end{pmatrix}
\begin{pmatrix}
\cos(\frac{\h_2}{2}) & -\sin(\frac{\h_2}{2}) \\
\sin(\frac{\h_2}{2}) & \cos(\frac{\h_2}{2})
\end{pmatrix}
=
\e\begin{pmatrix}
\a & -\overline{\beta} \\
\beta & \overline{\a}
\end{pmatrix}
\ee
where $\a=\cos\left(\frac{\tilde\h}{2}\right)e^{i\frac{\tilde{\f}+\tilde{\psi}}{2}}$ and $\beta=\sin\left(\frac{\tilde\h}{2}\right)e^{i\frac{-\tilde{\f}+\tilde{\psi}}{2}}$ and $\e\in\{-1,+1\}$. The sign $\e$ is due to the two possible choices of preimages via $\pi$. The Euler coordinates of these two preimages differ by a translation $\tilde{\psi}\mapsto \tilde{\psi}+2\pi$, therefore we can always restrict to $\tilde{\psi}\in [0,2\pi)$.
\end{proof}

\begin{thm}\label{thm:covaconv}
Let $X_n\colon SO(3)\to \C$ satisfy the Assumption \ref{ass}. Let $\xi_n\colon \D\to \C$ be the smooth GRF defined in Definition \ref{defi:rescaled field}.
Then $K_{\xi_n}\to K_{\xi_\infty}$ in $\mC^\infty(\D\times\D,\C)$.
\end{thm}
\begin{proof}
Let $z_1=x_1e^{i\f_1}$ and $z_2=x_2e^{i\f_2}$, then
\bega\label{eq:firstline}
K_{\xi_n}(z_1,z_2)&=\Gamma_n \left(R_2\Big(-\rho_n x_1\Big)R_3\Big(\f_2-\f_1\Big)R_2\Big(\rho_nx_2\Big)\right)e^{is_n(\f_1-\f_2)}
\\
&=k_n(\tilde{\h}_n)e^{is_n(\tilde{\f}_n+\tilde{\psi}_n)}e^{is_n(\f_1-\f_2)}.
\eega
Where $\tilde{\f}_n,\tilde{\h}_n,\tilde{\psi_n}$ are the Euler angles defined as in Lemma \ref{lemma:trasfrule}, with $\h_i=\rho_nx_i$ and $\f=\f_2-\f_1$ so that
\be 
R\left(\tilde{\f}_n,\tilde{\h}_n,\tilde{\psi_n}\right)=R_2\Big(-\rho_n x_1\Big)R_3\Big(\f_2-\f_1\Big)R_2\Big(\rho_nx_2\Big).
\ee
Then $\tilde{\h}_n$ is the spherical distance between $\phi^\mathbb{1}_n(z_1)$ and $\phi^\mathbb{1}_n(z_2)$, hence it is given by the formula
\bega\label{eq:costita} 
\cos\tilde{\h}_n&=
\langle \phi^\mathbb{1}_n(z_1),\phi^\mathbb{1}_n(z_2)\rangle
\\
&=
\sin\left(\rho_nx_1\right)\sin \left(\rho_nx_2\right)\cos(\f_1-\f_2) +\cos\left(\rho_nx_1\right)\cos\left(\rho_nx_2\right)
\\
&=
\left(\rho_n^2x_1x_2+O\left(\rho_n^{4}\right)\right)\cos(\f_1-\f_2)+
\\
&\quad +\left(1-\frac{1}{2}\rho_n^2x_1^2+O\left(\rho_n^{4}\right)\right)\left(1-\frac{1}{2}\rho_n^2x_2^2+O\left(\rho_n^{4}\right)\right)
\\
&= 1+\rho_n^2\langle z_1,z_2\rangle-\rho_n^2\frac{|z_1|^2}{2}-\rho_n^2\frac{|z_2|^2}{2} + O\left(\rho_n^4\right)
\\
&=1-\rho_n^2\frac{|z_1-z_2|^2}{2}+ O\left(\rho_n^4\right)
.
\eega
This implies that we have the following limit for the radial part of $K_{\xi_n}(z_1,z_2)$:
\bega
\lim_{n\to +\infty}k_{n}(\tilde{\h}_n)
&=
\lim_{n\to+\infty} k_n\left(\arccos\left(1-\frac{\rho_n^2|z_1-z_2|^2+O(\rho_n^2)}{2} \right)\right)
\\
&=\lim_{n\to+\infty} k_n\left(\rho_n|z_1-z_2|+O(\rho_n^2)\right)
\\
&=k_\infty(|z_1-z_2|).
\eega
By definition (see Lemma \ref{lemma:trasfrule}) we have 
\bega 
\a_n:=\cos\left(\frac{\tilde\h_n}{2}\right)e^{i\frac{\tilde{\f}+\tilde{\psi}}{2}}&=\cos\left(\frac{\rho_nx_1}{2}\right)\cos\left(\frac{\rho_nx_2}{2}\right)e^{i\frac\f2}+
\\
&\quad +\sin\left(\frac{\h_1}{2}\right)\sin\left(\frac{\h_2}{2}\right)e^{-i\frac\f2}
\\
&=\cos\left(\frac{\rho_nx_1}{2}\right)\cos\left(\frac{\rho_nx_2}{2}\right)e^{i\frac\f2}(1+t_1t_2e^{-i\f}),
\eega
where $t_i=\tan (\frac{\rho_nx_i}{2})=\frac{1}{2}\rho_nx_i+O(\rho_n^3)$ and $\f=\f_2-\f_1$. Therefore
\bega\label{eq:phase}
\lim_{n\to\infty}e^{is_n(\tilde{\f}_n+\tilde{\psi}_n)}e^{is_n(\f_1-\f_2)}&=
\lim_{n\to\infty}\left(\frac{\a_n}{|\a_n|}\right)^{2s_n}e^{-is_n\f}
\\
&=\lim_{n\to\infty}\frac{\left(1+t_1t_2e^{-i\f}\right)^{2s_n}}{\left|1+2t_1t_2\cos(\f)+t_1^2t_2^2\right|^{s_n}}
\\
&=\lim_{n\to\infty}
\frac{\left(1+\frac{\rho_n^2}{2}x_1x_2e^{-i\f}+O(\rho_n^4)\right)^{\frac{2}{\rho_n^2}(s_n\rho_n^2)}}{\left|1+\rho_n^2\frac{x_1x_2}{2}\cos(\f)+O(\rho_n^4)\right|^{\frac{1}{\rho_n^2}(s_n\rho_n^2)}}
\\
&=\exp\left(\beta x_1x_2e^{-i\f}\right)\exp\left(-\beta\frac{x_1x_2}{2}\cos(\f)\right)
\\
&=\exp\left(-\beta i x_1x_2\sin(\f_2-\f_1)\right)
\\
&=\exp\left(\beta i \Im (z_1\overline{z_2})\right).
\eega
Combining the latter computation, with the first line \eqref{eq:firstline} and with the estimate of $\tilde\h$, we get that
\be 
\lim_{n\to+\infty}K_{\xi_n}(z_1,z_2)=k_\infty(|z_1-z_2|)\exp\left(\beta i \Im (z_1\overline{z_2})\right).
\ee
By an analogous argument, the above limit can be shown to hold in the $\mC^\infty$ sense, thus we conclude.
\end{proof}
\begin{remark}\label{rem:covaconv}
It was proved in \cite{dtgrf}, that the convergence of the covariance functions in the $C^\infty$ topology is equivalent to the convergence in law $\xi_n\nrw \xi_\infty$ as random elements of $C^\infty(\D,\C)$, i.e. to the weak-$*$ convergence of the corresponding sequence of probability measures. In particular, it also implies that the limit of the covariance functions, if exists, is the covariance function of a smooth Gaussian field, hence that Definition \ref{def:scalimit} is well posed. The following are equivalent (in virtue of Portmanteau's theorem) characterizations of such convergence: 
\begin{enumerate}
\item For any continuous function $\mathcal{F}\colon \mC^\infty(\D,\C)\to [0,1]$, we have that
\be 
\lim_{\ell\to+\infty}\mathcal{F}(\xi_\ell)=\mathcal{F}(\xi_\infty).
\ee
\item For any Borel subset $B\subset \mC^\infty(\D,\C)$, we have that
\be 
\P\{\xi_\infty\in \text{int}(B)\}\le \liminf_{\ell\to+\infty}\P\{\xi_\ell\in B\}\le \limsup_{\ell\to+\infty}\P\{\xi_\ell\in B\}\le \P\{\xi_\infty\in \overline{B}\}.
\ee
\end{enumerate}
\end{remark}
\section{Main Results}\label{sec:MainResults}
\subsection{Setting}\label{sec:setting}
In this section we will consider the following setting. Let $N\in\N$ and let $\us_n=(s^1_n,\dots,s^N_n)\in \Z^N$ be a sequence of $N$-tuples of spin weights. Let $\usig_n$ be a sequence of isotropic Gaussian random sections of the complex vector bundle
\be 
\Esn=\spi{s^1_n}\oplus \dots\oplus \spi{s^N_n},
\ee
i.e. $\usign=(\sigma_n^1,\dots \sigma_n^N)$ is a $N$-tuple of isotropic spin Gaussian fields and we assume that the whole collection $\{\sigma^i_n\}_{i,n}$ is an independent family. Let $\uXn=(X_n^1,\dots,X_n^N)\colon SO(3)\to \C^N$ be the corresponding sequence of isotropic Gaussian functions on $SO(3)$ and let $\ukn=(k_n^1,\dots,k_n^N)\colon \R\to \R^N$ be their circular covariance functions (see Section \ref{sec:SpinRandomFields}). Let $\ubeta=(\beta^1,\dots,\beta^N)\in \R^N$.
\begin{ass}\label{multiass}
Assume that, for every $i=1,\dots,N$, the sequence $\{\sigma_n^i\}_n$ satisfies the scaling assumption \ref{ass} with $\beta=\beta^i$ with respect to the same infinitesimal sequence of positive real numbers $\rho_n\to 0^+$.
\end{ass}
Let $\uxin=(\xi_n^1,\dots,\xi_n^N)\colon \D\to \C^N$ be the sequence of rescaled fields (see Definition \ref{defi:rescaled field}) and let $\uxinf=(\xi_\infty^1,\dots,\xi_\infty^N)$ be the $N$-tuple of limit fields (see Definition \ref{def:scalimit}).

\subsubsection{Jets and type-W singularities}
As we did in Subsection \ref{subsec:intrising}, we identify
$ 
J^r(\D,\C^N)=\D\times J^r_0(\D,\C^N)=\D\times \R^k,
$
so that taking the jet at a point $p\in \D$ yields a map $j^r_p\colon \mC^\infty(\D,\C^N)\to \R^k$.
\begin{defi}\label{def:Y}
For $n\in\N\cup\{\infty\}$, let $Y_n\colon \D\to \R^k$ be the Gaussian field such that $Y(p)=j^r_p\uxin$.
\end{defi}

We will consider the random subset of the disk $\D$ given by the type-W singularity 
\be 
Z_W(\usign)=j^r\usign^{-1}(W)\subset S^2,
\ee 
defined by a closed intrinsic semialgebraic subset $W\subset J^r(S^2|\Esn)$ modeled on $W_0=\D\times\Sigma \subset J^r(\C,\C^N)= \D\times \R^k$ (see Subsection \ref{subsec:intrising}). Asking for the semialgebraicity of $W$ is equivalent to assume that $\Sigma\subset \R^k$ is semialgebraic. 
For all $n\in\N\cup\{\infty\}$, let 
\be 
Z_n:=Z_{W_0}(\uxin)=Y_n^{-1}(\Sigma)\subset \D.
\ee

By construction (see the discussion before Definition \ref{defi:rescaled field}), if $B_n\subset S^2$ is a sequence of shrinking spherical balls of radius $\rho_n$, then there is a sequence of diffeomorphisms $\phi^g_{\rho_n}\colon \D\to B_n$ such that 
\be\label{eq:diffeoZ}
\phi^g_{\rho_n}(Z_n)= Z_W(\usign)\cap \overline{B_n}.
\ee 
Moroever, let $\mathring{Z}_n:=Z_n\cap B_n=Z_n\- \de B_n$.
\subsubsection{Supports}
\begin{defi}
For $n\in\N\cup \{\infty\}$, define
$F_n:=\spt(\xi_n)\subset \mC^\infty(\D,\C^N)$ to be the \emph{topological support} of the law of $\uxin$, i.e. 
\be 
F_n=\left\{f\in \mC^\infty(\D,\C^N)\colon \P\left\{\uxin\in O_f\right\}>0 \text{ for any $O_f$ open neighborhood of $f$}\right\}
\ee
\end{defi}
By standard arguments (see \cite{dtgrf,NazarovSodin2012}, for instance), the above definition is well posed and the support $F_n$ is always a closed subspace of $\mC^\infty(\D,\C^N)$, indeed it is the smallest closed subset with $[\uxin]$-probability one.
By construction we have that 
\be 
\spt(Y_n(p))=j_p^rF_n=\left\{j_p^rf\colon f\in F_n\right\}\subset \R^k
\ee

The next assumption ensures that the type-W singularities $Z_W(\usign)\subset S^2$ and $Z_n\subset \D$ are nondegenerate, thus they are Whitney stratified subsets in the sense of  subsection \ref{subsec:intrising}. This will be proved in Theorem \ref{thm:mainlaw} below.
\begin{ass}\label{transvass}
For every $n\in\N\cup\{\infty\}$, $j^r_0F_n\transv \Sigma$.
\end{ass}
\subsection{Convergence in distribution}
In this subsection we provide a rigorous statement and a proof for Theorem \ref{thm:mainlaw} in Section \ref{sec:overview}.
\begin{remark}
We recall that, by construction, we have that the random subsets
$Z_W(\usign)\cap B_{\rho_n}\cong Z_n$ are diffeomorphic in the sense of \eqref{eq:diffeoZ}.
\end{remark}
\begin{thm}\label{thm:tecnomainlaw}
\begin{enumerate}
Assume that $\Sigma$ is closed and that $j^r_0F_n\transv\Sigma$ for all $n\in\N\cup\{\infty\}$. Then the following properties hold.
\item Almost surely, $Z_W(\usign)\subset S^2$ is nondegenerate  for all $n\in\N$. The same holds for $Z_n\subset \D$ for all $n\in\N\cup\{+\infty\}$.
\item There exists a discrete limiting probability law $p_W(S)$ on the set of diffeomorphisms classes of Whitney stratified subsets $S\subset \D$:
\be 
\exists \lim_{n\to +\infty}\P\{Z_n\text{ is diffeomorphic to }S\}=p_W(S).
\ee
\item Whenever $S$ is diffeomorphic to a nondegenerate type-W singularity of some smooth function $f\in F_\infty$, we have that $p_W(Z)>0$.
\item There is convergence in law: $\mathcal{L}_i(Z_n)\nrw \mathcal{L}_i(Z_\infty)$ and $b_i(Z_n)\nrw b_i(Z_\infty)$.
\end{enumerate}
\end{thm}
\begin{proof}
(1). We want to apply \cite[Theorem 7]{dtgrf} to the random section $\usign$ and to the finite union of smooth submanifolds $W\subset J^r(S^2|\Esn)$. To see that the hypotheses of the theorem are satisfied, just observe that, if $W$ is intrinsic with model $W_0=\D\times \Sigma$ and $\usign$ is isotropic, then $\spt[j^r_p\usign]\transv W$ if and only if $\spt[j^r\uxin]\transv W_0$, if and only if $j^r_0F\transv\Sigma$. Therefore, by \cite[Theorem 7]{dtgrf}, we have that
\be 
\P\{j^r\usign \transv W\}=1 \quad \text{and} \quad \P\{Y_n\transv \Sigma \text{, } Y_n|_{\de \D}\transv \Sigma\}=1,
\ee
for all $n\in \N$. The second identity holds for $n=\infty$ as well, for the same reason.

(2). Consider the set: 
\be 
U_S:=\{f\in \mC^\infty(\D,\R^k)\colon f^{-1}(W) \text{ is diffeomorphic to } S\}.
\ee
As it is explained in  \cite{dtgrf}, by Thom isotopy theorem, if $f_t$ is a homotopy of maps such that $f_t\transv \Sigma$ and $f_t|_{\de \D}\transv \Sigma$, then the diffeotopy type of the pair $(\D,f_t^{-1}(\Sigma))$ is constant. Moreover, if $\Sigma$ is closed, then the transversality condition is open in the space of smooth functions, therefore we have that
\be 
\mathrm{int}(U_S)=U_S\- \Delta_\Sigma \quad \text{and}\quad \de U_S \subset \Delta_\Sigma,
\ee
where $\Delta_s=\{f\in \mC^\infty(\D,\R^k)\colon f\not \transv \Sigma\text{ or }f|_{\de \D}\not \transv \Sigma\}$. Notice that by Theorem \ref{thm:covaconv} and Remark \ref{rem:covaconv} we have 
\be 
\P\{Y_\infty\in \text{int}(U_S)\}\le \liminf_{n\to+\infty}\P\{Y_n\in U_S\}\le \limsup_{n\to+\infty}\P\{Y_n\in U_S\}\le \P\{Y_\infty\in \overline{U_S}\}.
\ee
Therefore, since by point (1) we have that $\P\{Y_n\in\Delta_W\}=0$, it follows that 
\be 
\exists \lim_{n\to \infty}\P\{Y_n\in U_S\}=\P\{Y_\infty \in U_S\}.
\ee

(3). Let $f\in U_S \cap F_\infty\-\Delta_\Sigma$ and assume that  $S$ is diffeomorphic to $f^{-1}(\Sigma)$. Then, by Thom isotopy theorem again, the same holds for all $g$ on a neighborhood $O_f\subset \mC^\infty(\D,\R^k)$ of $f$. In other words the set $U_S\cap F_\infty\-\Delta_\Sigma$ is open. Since $\Delta_\Sigma$ has zero probability, we have that $\P\{Y_\infty\in U_S\}>0$ if and only if $U_S \cap F_\infty\-\Delta_\Sigma\neq \emptyset$. This proves (3).

(4). The convergence in law of Betti numbers follows directly from (2). For the Lipschitz-Killing curvatures $\mathcal{Z_i}$ we could essentially repeat the argument used to prove (2). A more direct way is to observe that the functional $L:f\mapsto \mathcal{L}(f^{-1}(\Sigma)$ is continuous on $\mC^\infty(\D,\R^k)\-\Delta_\Sigma$.
Since $\P\{Y_\infty\in\Delta_\Sigma\}=0$, this implies that the composition $L\circ Y_n$ converges in law in $\R$.
\end{proof}
\subsection{Convergence of Expectations}
In this subsection we provide a rigorous statement and a proof of Theorem \ref{thm:mainE} in section \ref{sec:overview}.
\begin{thm}\label{thm:tecnomainE}
Let $\Sigma\subset \R^k$ be closed. Assume that $j^r_0F_n\transv\Sigma$ for all $n\in\N\cup\{\infty\}$. For all $i=0,1,2$ we have:
\begin{enumerate}
\item 
 \be \E\mathcal{L}_i(Z_W(\sigma_n))=
    \rho_n^i\frac{\vol(S^2)}{\vol(B_{\rho_n})}\left(\E\mathcal{L}_i(\mathring Z_\infty)+O(1)\right).
    \ee
    \item 
    There are constants $c_i^W\ge 0,C_i^W>0$ such that
    \begin{equation}\label{betti}
    \frac{\vol(S^2)}{\vol(B_{\rho_\ell})}c_i^W\le \E b_i(Z_W(\sigma_n))\le 
    \frac{\vol(S^2)}{\vol(B_{\rho_n})}C_i^W;
    \end{equation}
    \item If there exists a smooth function $f\in F_\infty$ such that $Z_{W_0}(f)$ is regular and it has a connected component $C\subset \mathrm{int}(\D)$, with $b_i(C)>0$, then $c^W_i>0$.
\end{enumerate}
\end{thm}
\subsubsection{Outline of the proof of Theorem \ref{thm:mainE}}\label{subsub:outlineE}
\begin{enumerate}
\item As it will be clear from points (2) and (3), it is enough to prove the theorem in the case when $W\subset J^r(S^2|\Es\oplus\C)$ has codimension $2$, i.e. when the only nontrivial quantity is the cardinality $\#=b_0$ of the random finite set $Z^W(\usig)$. However, we need to prove this case in a slightly more general form to include weighted count of such set of points: Let $\a\colon \mC^\infty(S^2|\Es)\times S^2\to \R$ be a measurable function and define, for any $A\subset S^2$
    \be 
    \#^\a_{j^r\usig\in W}(A):=\sum_{p\in Z^W(\usig)\cap A}\a(\usig,p).
    \ee
    In \cite[Theorem 4.1]{stecconi2021kacrice} it is shown that these kinds of counting measures admit an integral formula (Kac-Rice-style).
    \be 
    \E\left\{\#^\a_{j^r\usig\in W}(A)\right\}=\int_A \delta^\a_{j^r\usig\in W}
    \ee
    Arguing as in the proof of \cite[Theorem 27]{mttrps} and  \cite[Corollary 3.9]{stecconi2021kacrice} we will be able to understand their asymptotic behavior. This, together with items 2 and 3 below, will also prove the general case automatically.
\item We then exploit Theorem \ref{thm:strat} to show that for any $W\subset J^r(S^2|\Es)$, and every $i=0,1,2$, there exists another singularity type $W_i\subset {J^{r+1}(S^2|\Es\oplus \C)}$, having codimension $2$ and $\a=\a_{W_i}\colon \mC^\infty(S^2|\Es\otimes \C)\times S^2\to \R$ measurable and bounded such that 
    \be 
    \mathcal{L}_i(Z^W(\usig))=\sum_{p\in Z^{W_i}(\usig,\sigma_0)}\a((\usig,\sigma_0),p):=\#^\a_{j^{r+1}(\usig,\sigma_0)\in W_i},
    \ee
    where $\sigma_0\colon S^2\to \C$ is a random function with spin $0$.
    \item Finally, we establish a similar statement for the Betti number, although here we are only able produce an inequality (derived from Morse inequalities):
    \be 
    b_i(Z^W(\usig))\le C_W\sum_{p\in Z^{\hat{W}}(\usig,\sigma_0)}\a((\usig,\sigma_0),p):=\#^\a_{j^{r+1}(\usig,\sigma_0)\in W'},
    \ee
    for some constant $C_W>0$ depending only on $W$ and some higher singularity type $\hat{W}\subset{J^{r+1}(S^2|\Es\oplus \C)}$ of codimension $2$. This follows again from Theorem \ref{thm:strat}.
\end{enumerate}

\section{Proof of Theorem \ref{thm:mainE}}\label{sec:proofmainE}
In this section we give a full proof of the convergence of the expectation; the proof is split into three steps, as described above; the proof is based upon two auxiliary results, whose proofs are given elsewhere, see \cite{Stecconi_2022}.
\subsection{Step 1}\label{step1}
The following theorem is the main technical result of this section.
\begin{thm}\label{thm:prettydamnbadassgun}
Let $\usign$ be the sequence of isotropic Gaussian multi-spin functions that falls in the setting described in Section \ref{sec:setting}, in addition to Assumptions \ref{multiass} and \ref{transvass}, assume that $W\subset J^r(S^2|\Es)$ has codimension $2$. Let $\a\colon W'\to \R$ be a continuous and intrinsic function (see Definitions \ref{def:intrisub} and \ref{def:intrifun}), where
\be 
W':=\left\{j^{r+1}_p\usig\in (S^2|\Es)\colon j^r_p\usig \in W, d_p\left(j^r\usig\right)\transv T_{j^r_p\usig}W\right\}
\ee
and define, for $A\subset S^2$,
\be 
\#^\a_{j^r\usig\in W}(A):=\sum_{p\in Z_W(\usig)\cap A}\a(j^{r+1}_p\usig).
\ee
Then
\be 
\E\#^\a_{j^r\usig\in W}(S^2)=\frac{\vol(S^2)}{\vol(B_{\rho_n})}\left(\E\#^{\a_0}_{j^r\xi_\infty\in W_0}(\D)+o(1)\right).
\ee
\end{thm}
\begin{proof}
Since $\a$ and $W$ are assumed to be intrinsic, it follows that the signed measure $A\mapsto \E\#^\a_{j^r\usig\in W}(A)$ is a multiple of the volume measure on $S^2$. This reduces the problem to its local counterpart, which is Lemma \ref{lem:mainlemma}, applied to the sequence of fields $Y_n$ defined as in Definition \ref{def:Y}.
\end{proof}
Let us define the subset $\Sigma'\subset J^1(\D,\R^k)$ such that
\be 
\Sigma':=\left\{j^1_pf\in J^1(\D,\R^k): f(p)\in\Sigma \text{ and } d_pf\transv T_{f(p)}\Sigma \right\}.
\ee
Here, recall that $T_y\Sigma$ is the tangent space to the stratum of $\Sigma$ containing $y$.
\begin{defi}
Let $\a\colon \Sigma'\to \R$, let $f\colon \D\to \R^k$ and $A\subset \D$. We define
\be 
\#^\a_{f\in \Sigma}(A):=\sum_{p\in f^{-1}(\Sigma)\cap A} \a(p,f).
\ee
Let also $T_\Sigma\subset \mC^\infty(\D,\R^k)$ be the set
\be 
T_\Sigma:=\{f\in \mC^\infty(\D,\R^k)\colon f\transv \Sigma\}.
\ee
Notice that $f\in  T_\Sigma$ if and only if $(p,f)\in \Sigma'$ for every $p\in f^{-1}(\Sigma)$. 
\end{defi}
\begin{lemma}\label{lem:mainlemma}
Let $Y_n$ be as in Definition \ref{def:Y} under the assumptions \ref{multiass} and \ref{transvass}, let $\Sigma$ be closed and have codimension $2$. Let $\a\colon \hat T_\Sigma\to \R$ be continuous and bounded. 
\be 
\lim_{n\to\infty}\E\#^\a_{Y_n\in\Sigma}(\D)=\E\#^\a_{Y_\infty\in\Sigma}(\D).
\ee
\end{lemma}
In the following, we will take up the notations of \cite{stecconi2021kacrice}, in which a Kac-Rice formula for the expectation $\E\#^\a_{Y_n\in\Sigma}$ is proved: the formula \eqref{eq:delta} below is given by \cite[Theorem 4.1 ]{stecconi2021kacrice}. In particular, given two subspaces $V,W\subset \R^k$, the quantity $\sigma_{\R^k}(V,W)$ is the product of the sines of the principal angles in $\R^k$ between the vector subspaces $V$ and $W$, See \cite[Appendix B]{stecconi2021kacrice}. We will omit the pedice and write just $\sigma(V,W)$, when the ambient space is clear. Moreover, if $S$ is a Riemannian manifold we denote its Riemannian volume density at $y\in S$ as $dS(y)$, so that the integral of a function $f\colon S\to \R$ with respect to the Riemannian volume density will be written as $\int_SfdS=\int_Sf(y)dS(y)$, see \cite[Appendix A]{stecconi2021kacrice}.
\begin{proof}
A consequence of Assumption \ref{transvass} is that $Y_n\transv \Sigma$ with probability one (see \ref{thm:tecnomainlaw}), therefore $Z_n=Y_n^{-1}(\Sigma)$ is a random discrete subset. Using \cite{stecconi2021kacrice}, we can write 
\bega\label{eq:delta}
&\E\#^\a_{Y_n\in \Sigma}(A)
=
\int_A\delta_n^\a(z)dz
\\
&=
\int_A\int_{\Sigma\cap j_p^rF_n}\E\left\{
\a(z,Y_n)J_zY_n\frac{\sigma_y(d_zY_n,\Sigma)}{\sigma_y(j_p^rF_n,\Sigma)}\Big| Y_n(z)=y
\right\}\times
\\
&
\times\rho_{Y_n(z)}(y) d\left(\Sigma\cap j^r_pF_n\right)(y)dz,
\eega
where, for $y\in \spt[Y_n(z)]$, we define $\rho_{[Y_n(z)]}(y)$ to be the density of the Gaussian random vector $Y_n(z)$ evaluated at $y\in \spt[Y_n(z)]=j^r_zF_n$.
The quantity $\sigma_y(d_zY_n,\Sigma)$ is the product of the sines of the principal  in $\R^k$ between the vector subspaces $d_zY_n(\R^2)$ and $T_y\Sigma$ and it is defined for all $y\in\Sigma$. See \cite[Appendix B]{stecconi2021kacrice} for the precise definition and more details.
\be 
\sigma_y(d_zY_n,\Sigma)=\sigma_{\R^k}(d_zY_n(\R^2),T_y\Sigma).
\ee
\end{proof}
The next two results are addressed in \cite{Stecconi_2022}.
\begin{lemma}[Claim]
Let $z\in\D$, let $V_n:=(Y_n(z),d_zY_n)\in\R^{k+2k}$ be convergent in law  $V_n\nrw V_\infty$, let $S_n:=\spt[V_n]$ and let $\rho_{V_n}\colon S_n\to \R$ be the density of the Gaussian random vector $V_n$, for all $n\in\N\cup\{\infty\}$.
\be 
\delta^\a_{n}(z)=\int_{(\Sigma\times \R^{2k})\cap S_n}f(z,v,K_{V_n})\frac{\rho_{V_n}(v)}{\sigma(T_v\Sigma\times \R^{2k},S_n)}d\left((\Sigma\times\R^{2k})\cap S_n\right)(v),
\ee
where $f(z,v,K)$ is a continuous function such that
\be 
f(z,v,K)\le C(1+|v|^C)
\ee
for some $C>0$.
\end{lemma}

\begin{thm}[Claim]
Let $V_n\nrw V_\infty$ be a weakly convergent sequence of centered Gaussian random vectors in $\R^N$. Let $K_n\to K_\infty$ be the sequence of their covariance matrices. Let $S_n=\text{im}(K_n), S_\infty=\text{im}(K_\infty)\subset \R^N$ be the sequence of their supports. Let $W\subset \R^N$ be a closed semialgebraic subset such that $S_n\transv W$ for all $n\in\N\cup\{\infty\}$. Let $f_n\colon W\to \R$ be convergent sequence of continuous functions $f_n\to f_\infty$ (uniformly on compact sets) and assume  that there is a uniform constant $C>0$ such that $f_n(y)\le C(1+|y|^C)$. Define
\be 
I_n:=\int_{W\cap S_n}f_n(y)\frac{\rho_{V_n}(y)}{\sigma(T_yW,S_n)}d(W\cap S_n)(y).
\ee
Then $I_n$ is finite and $I_n\to I_\infty$.
\end{thm}
\subsection{Step 2: Lipschitz-Killing curvatures}\label{step2}
Let $\sigma_{0}^n:S^2\to \C$ be a reindexing of the sequence of isotropic smooth Gaussian random function defined in Example \ref{ex:spharmoel} (see also Equation \eqref{eq:monocromatic}), with spin equal to zero. 
\be 
\sigma^n_0=\sum_{m=-\ell(n)}^{\ell(n)}a^{\ell(n)}_{m,0}D^{\ell(n)}_{m,0}.
\ee 
The circular covariance function of $\sigma^n_0$ is $k_{\sigma_0^n}(\h)=d^{\ell(n)}_{0,0}$, which satisfies the scaling assumption \ref{ass} with rate $\ell(n)^{-1}$ and $k_\infty=J_0$. Clearly we can choose $\ell(n)$ so that $\ell(n)^{-1}\sim \rho_n$, by repeating or skipping some $\ell$s. Define $h_n:=\Re( \sigma_0^n)$.

Let ${\us}_n':=(0,\us_n)$, ${\Esn}'=\C\oplus \Esn$ and define $\usignp:=(\sigma_{0}^n,\usign)\in\mC^\infty(S^2|\Esn')$. 
Now, observe that this new sequence of Gaussian isotropic multi-spin sections $\usignp$ satisfies Assumption \ref{multiass}, with the same shrinking rate $\rho_n\to 0$.

We can now exploit Theorem \ref{thm:prettydamnbadassgun} to prove each case (Lipschitz-Killing curvatures and Betti numbers) of Theorem \ref{thm:tecnomainE}. 
\subsubsection{The Euler-Poincaré Characteristic}\label{sssec:EPproof}
By Assumption \ref{transvass} and Theorem \ref{thm:mainlaw}.(1), we know that $j^r\usign\transv W$. Consider the semialgebraic intrinsic subset $W'\subset J^{r+1}(S^2|\Esn')$ defined in Theorem \ref{thm:strat}. We claim that $\usignp$ satisfies Assumption \ref{multiass} with respect to $W'$, as well. The reasons why this is true are two: first, the support of $j^1_ph_n$ is the whole fiber of the jet space: $J^1_0(S^2,\R)$ and second, the structure of $W'$, established by Theorem \ref{thm:strat}.(4), implies that 
the normal bundle $N_wS'$ of any stratum $S'$ of $W'$ at a point $w\in S'$ projects onto the space $J^1(S^2,\R)\times N_{\pi_1(w)}S$ via the natural map 
\bega
J^{r+1}(S^2|\Esn')/T_wS' = N_wS'&\xrightarrow{\pi=(\pi_1,\pi_2)} N_{\pi_1(w)}S\times J^1(S^2,\R),
\\
j^{r+1}_p(\sigma_0,\usig)+T_wS'&\mapsto (j^r\usig + T_{\pi_1(w)}S, j^1\Re(\sigma_0))
\eega
where $N_{\pi_1(w)}S$ is the normal bundle of the stratum $S$ of $W$ (meant as a quotient of the ambient space modulo $TS$), containing $\pi_1(w)$ (by definition $w\in W'$ only if $\pi_1(w)\in W$).
Therefore, if $\spt(j^r_p\usign)\transv W$ and $\spt(j^1\Re(\sigma^n_0))=J^1_p(S^2,\R)$, then $\spt(j^{r+1}_p\usignp)\transv W'$. The latter condition is equivalent to Assumption \ref{transvass}.
By Theorem \ref{thm:strat}, we have that, almost surely,
\be 
\mathcal{L}_0\left(Z_W(\usign)\right)=\chi\left(Z_W(\usign)\right)=\sum_{p\in Z_{W'}(\usignp)}\a'(j_p^{r+2}\usig')=\#^{\a'}_{j^{r+1}\usignp\in W''}(A).
\ee
By the previous discussion, we see that we are now in position to apply Theorem \ref{thm:prettydamnbadassgun} to the sequence $\usignp$, the semialgebraic intrinsic submanifold $W''$ and the intrinsic function $\a'$, therefore
\be\label{eq:sostufo}
\E \mathcal{L}_0\left(Z_W(\usign)\right)=\E\#^{\a'}_{j^{r+1}\usignp\in W''}(S^2)=\frac{\vol(S^2)}{\vol(B_{\rho_n})}\left(\E\#^{\a_0'}_{j^{r+1}\xi_\infty\in W_0''}(\D)+o(1)\right).
\ee
Finally, we conclude by observing that 
\be 
\E\#^{\a_0'}_{j^{r+1}\xi_\infty\in W_0''}(\D)= \mathcal{L}_0\left(Z_\infty,\mathring{Z}_\infty\right)=\mathcal{L}_0\left(\mathring{Z}_\infty\right).
\ee
This follows from the fact that both $\E\#^{\a'}_{j^{r+1}\usignp\in W''}(\cdot)$ and $\E \mathcal{L}_0\left(Z_W(\usign),\cdot\right)$ are invariant measures on $S^2$, thus equation \eqref{eq:sostufo} implies that they are equal. Evaluating on the (open) ball $B_n$ gives
\be 
\E\#^{\a_0'}_{j^{r+1}\xi_n\in W_0''}(\D)=\E\#^{\a'}_{j^{r+1}\usignp\in W''}(B_n)=\E \mathcal{L}_0\left(Z_W(\usign),B_n\right)=\E\mathcal{L}_0\left(Z_n,\mathring{Z}_n\right).
\ee
\subsubsection{The first intrinsic volume}
Notice that $h_n\in \ker(\Delta-\lambda(n))$, with $\lambda(n)=\ell(n)(\ell(n)+1)\sim \rho_n^{-2}$. Indeed, $k_{\sigma_0^n}''(0)=(d^\ell_{0,0})''(0)=\frac{\ell(\ell+1)}{2}$, thus, by Proposition \ref{prop:spinzero}, we see that the conformal factor of the Adler-Taylor metric $g^h$ of $h$ is $\frac{\lambda(n)}{2}$. Therefore, expressing $\mathcal{L}_1$ as in Subsection \ref{sec:EpliLK} and using the formula \ref{prop:vol}, we have the identity:
\bega 
\mathcal{L}_1(Z_W(\usign))&=\frac12 \int_{\de_1 Z_W(\usign)}\beta^1(j^r_p\usign)d\mathcal{H}^{1}(p)
\\
&= \left(\frac{\lambda(n)}{2}\right)^{-\frac 12}\frac{ \pi}{2}\E\left(\sum_{p\in \de_1 Z_W(\usign)\cap \{h_n=0\}}\beta^1(j^r_p\usign)\right),
\eega
where $\beta_1(j^r_p\usign)$ is defined as $2$ minus the number of $2-$dimensional strata of $Z_W(\usign)=j^r\usign^{-1}(W)$ that are adjacent to $p$ (see section \ref{sec:EpliLK}). Therefore, reasoning as for the Euler-Poincaré characteristic, we can easily define $W'\subset J^r(S^2|\Es')$ and $\a$ intrinsic such that $\de_1Z_W(\usign)\cap \{h_n=0\}=Z_{W'}(\usignp)$ and $\beta^1(j^r_p\usign)=\a(j^{r}_p\usignp)$, so that Theorem \ref{thm:prettydamnbadassgun} yields:
\bega 
\E\mathcal{L}_1(Z_W(\usign))&=
\sqrt{\lambda(n)}\frac{\pi}{2\sqrt{2}} \E\#^\a_{\usignp\in W'}(S^2)
\\
&=
\rho_n\frac{\vol(S^2)}{\vol(B_{\rho_n})}\left(\frac{\pi}{2\sqrt{2}}\E\#^{\a_0}_{\xi_\infty'\in W_0'}(\D)+o(1)\right)
\\
&=\rho_n\frac{\vol(S^2)}{\vol(B_{\rho_n})}\left(\E\mathcal{L}_1(Z_\infty)+o(1)\right).
\eega
The last identity is due to another application of Proposition \ref{prop:vol}. Indeed, by construction, we have that $(\xi_\infty')^{-1}(W_0')=\xi_\infty^{-1}(W_0)\cap \{h_\infty=0\}$ and $\xi_\infty^{-1}(W)=Z_\infty$, where $h_\infty$ is the scaling limit of $h_n$, that is the real Berry field $h_\infty$, with covariance $J_0$. Here, $\lambda=1$ (indeed $h_\infty$ is a solution of the Helmoltz equation, see Appendix \ref{app:berry}) thus the conformal factor is $\frac12$.
\subsubsection{The area}
The case of $\mathcal{L}_2$ is the easiest and it can be proven directly by changing the order of integration in 
\be 
\mathcal{L}_2(Z_W(\usign))=\int_{S^2}1_{W}(j^r\usign),
\ee
or by reasoning analogously to the previous case, taking an additional random function $h_n'$ as an independent copy of $h_n$ and using Proposition \ref{prop:vol} to obtain the identity: 
\be 
\mathcal{L}_2(Z_W(\usign))=\lambda(n)\pi \E\left(\#\de_2 Z_W(\usign)\cap \{h_n=0\}\cap \{h_n'=0\}\right).
\ee
\subsection{Step 3: Betti numbers}\label{step3}
Let us define $\usignp$ as above and let us consider the same $W'\subset J^{r+1}(S^2|\Esn')$ as in the case of the Euler-Poincaré characteristic (see \ref{sssec:EPproof}), i.e. the one coming from Theorem \ref{thm:strat}. Let $i\in\{0,1,2\}$. By point $(2)$ of Theorem \ref{thm:strat}, we have, for some $N_W>0$, the inequality 
\be 
b_i\left(Z_W(\usig)\right)\le N_W\#Z_{W'}(\usig').
\ee
Taking the expectation on both sides and using Theorem \ref{thm:prettydamnbadassgun}, we deduce the upper bound:
\be 
\E b_i\left(Z_W(\usig)\right)\le \frac{\vol(S^2)}{\vol(B_{\rho_n})}\left(\E\#Z_{W'_0}(\xi_\infty)+o(1)\right).
\ee
It remains to show the lower bound, with the additional hypothesis that there exists $f\in F_\infty$ such that $Z_{W_0}(f)$ is regular and contains a closed connected component $C\subset Z_{W_0}(f)\cap \mathrm{int}(\D)$ such that $b_i(C)>0$. Recall that, by definition, the regularity of $Z_{W_0}(f)$ is equivalent to the transversality of the map $j^rf\colon \D\to J^r(\D,\C^N)$ to each stratum of $W$. Since $W$ is assumed to be closed, such condition is open, i.e. there is a whole open subset $\mathcal{U}\subset \mC^\infty(\D,\C^N)$ such that for all $g\in U$, we have that $j^rg\transv W$. It follows that $ \mathcal{U}\cap F_\infty$ is a non-empy open subset of $F_\infty$, because it contains $f$. Define $U\subset F_\infty$ to be the path connected component of $\mathcal{U}\cap F_\infty$, so that for any $g\in U$ there is a homotopy $g_t\in F_\infty$ of smooth maps, such that $g_0=f$ and $g_1=g$ and such that  $j^rg_t\transv W$ for every $t$. By Thom Isotopy Theorem, it follows that the isotopy type of $j^rg_t^{-1}(W)$ is constant, hence that there is a connected component $C_t$ with $b_i(C_t)\ge 1$. As a consequence we get that
\be 
\E b_i(Z_\infty^{\mathrm{(int)}})\ge \P\left(\xi_\infty \in U\right)>0,
\ee
where $Z_\infty^{\mathrm{(int)}}$ is the union of all the connected components of $Z_\infty$ that are contained in the interior of $\D$.
because $U$ is a non-empty open subset of $F_\infty$, the topological support of $\xi_\infty$.

After this consideration we can repeat the argument used in \cite{mttrps} to prove the lower bound on the expectation of Betti numbers of Kostlan singularities. Consider a sequence of subsets $\mathcal{B}_n\subset S^2$ such that for every $n$, $\mathcal{B}_n$ is a disjoint union of $L_n$ balls of radius $\rho_n$. Here, we can assume that \be 
L_n\ge c\frac{\vol(S^2)}{\vol(B_{\rho_n})}.
\ee
for some fixed constant $c>0$.
Then we have that 
\be 
\E b_i(Z_W(\usign))\ge \E\left(  b_i\left((Z_W(\usign)\cap \mathcal{B}_n)^{\mathrm{(int)}}\right)\right)= L_n \E\left(  b_i\left(Z_n^{\mathrm{(int)}}\right)\right),
\ee
where $X^{\mathrm{(int)}}$ is the union of the closed connected components of $X$. The last identity is due to the isotropy of $\usign$ and to the fact that $W$ is intrinsic. To conclude, it is sufficient to show that \be \label{eq:fatou}
\liminf_{d\to +\infty} \E\left(  b_i\left(Z_n^{\mathrm{(int)}}\right)\right)\ge \E b_i(Z_\infty^{\mathrm{(int)}}).
\ee
The latter inequality follows from Fatou's Lemma applied to the sequence of random variables $b_i(Z_n^{\mathrm{(int)}})$, which converges in law because of Theorem \ref{thm:mainlaw}.
\begin{proof}[This concludes the proof of Theorem \ref{thm:mainE}]
\end{proof}

%% file: e_Monocromatic.tex
\section{Monochromatic waves}\label{sec:monocrom}
The previous sections established a general framework to investigate the geometry of spin fields. In the present section, we specify those results to the more definite cases where the singular set $Z_W(\sigma_{\ell})$ are the excursion sets of a sequence monochromatic fields $\sigma_\ell$ with spin $s(\ell)$. To this aim, our first tool is to establish the local scaling behavior of the circular covariance function in this particular case.

We recall from the introduction that a spin monochromatic Gaussian random wave takes the form:
\be\label{eq:monocromatic2}
\sigma_{\ell}=\sum_{m=-\ell}^\ell a^\ell_{m,s(\ell)}Y^\ell_{m,s(\ell)}\in \mC^\infty(S^2|\spi{s(\ell)}),
\ee
where $Y^\ell_{m,s(\ell)}$ denote spin spherical harmonics and $a^\ell_{m,s(\ell)}$ are i.i.d. complex Gaussian variables. The field is normalized to have unit variance $\E\{\|\sigma_\ell(p)\|^2\}=1$ for every $p\in S^2$.
As mentioned in the introduction, we will allow the spin value to depend on $\ell$:
\be 
|s_\ell|=\ell-r_\ell
\ee
and we will focus on three different cases:
\begin{enumerate}[a.]
    \item (The Berry regime) $\liminf_{\ell\to \infty}\gap_\ell=+\infty$, thus the assumption \eqref{eq:beta} is satisfied with $\beta=0$; the shrinking rate is $\rho_\ell=\rho_\ell(s_\ell)$ defined in Equation \eqref{eq:berryrate}, hence in particular, $\rho_\ell=\frac{1}{\ell}$ when $s$ is fixed. For the covariance function $k_\ell =d_{-s,-s}^\ell $, it can be checked by Hilb's asymptotics (see Appendix \ref{app:hilb}) 
\be 
d^\ell_{-s-s}\left(\frac{t}{\ell+\frac12}\right) =J_0(t)\left(
\sqrt{
	\frac{
		\frac{t}{(\ell +\frac12)}
	}
	{
	\sin\left(\frac{t}{(\ell +\frac12)}\right)
	}
}
\right)+\delta_\ell \left(\frac{t}{(\ell +\frac12)}\right),
\ee
and hence we have
\be 
k_\ell\left(\rho_\ell\ \cdot\right)=d^{\ell}_{-s_\ell,-s_\ell}\left(\frac{\cdot}{2\sqrt{(\gap_\ell+1)(2\ell-\gap_\ell)}}\right)\xrightarrow[\ell\to +\infty]{\mC^\infty} J_0(\cdot).
\ee
\item (Middle regime) In this case $s_\ell=\ell-\gap$ for some fixed $\gap\in \N$; it is possible to establish the asymptotic convergence of the covariance function to an explicit analytic function $M_{r}$, see Equation \eqref{eq:Polynomiallimit}. The shrinking rate is:
    \be 
    \rho_\ell=\frac{1}{\sqrt{2(\gap_\ell+1) \ell}}.
    \ee
\item (Complex Bargmann-Fock/Gaussian entire process)
In the particular case $\ell=s_\ell$, that is $r_\ell=0$, we see that scaling hypothesis \ref{ass} is again satisfied, with the shrinking rate of $\rho_\ell=\sqrt{\frac{1}{2\ell}}$ and $k_\infty(x)=e^{-\frac{x^2}{4}}$, so that $\beta=\frac12$ and
\be 
k_\ell\left(\rho_\ell\  \cdot\right)=d^{\ell}_{-\ell,-\ell}\left(\frac{\cdot }{\sqrt{2\ell}}\right)\xrightarrow[\ell\to +\infty]{\mC^\infty} k_\infty(\cdot).
\ee
This confirms the fact that, in the case $\ell=s$, the spin field $\sigma_\ell$ is an holomorphic section of $\spi{s}=O(2s)$, in that the limit field $\xi_\infty$ is a deterministic multiple of the complex Bargmann-Fock GRF, which is almost surely holomorphic:
\be 
\xi_{\infty}(z)=\left(\sum_{n=0}^{\infty}\gamma_k\left(\frac{1}{n!}\right)^{\frac12}\left(\frac{z}{\sqrt 2}\right)^n\right)e^{-\frac{|z|^2}{4}},
\ee
where $\gamma_k\sim N_\C(0,1)$ are i.i.d. 

This can be seen by computing the covariance function:
\bega 
K_{\xi_\infty}(z_1,z_2)&=\E\{\xi_\infty(z_1)\overline{\xi_\infty(z_2)}\}=
\exp\left(\frac{z_1\overline{z_2}}{2}\right)e^{-\frac{|z_1|^2}4}e^{-\frac{|z_2|^2}4}
\\
&=
\exp\left(-\frac{|z_1-z_2|^2}{4}\right)\exp\left(\frac{i}{2}\Im (z_1\overline{z_2})\right)=
\\
&=
k_\infty\left(|z_1-z_2|\right)\exp\left(\frac{i}{2}\Im (z_1\overline{z_2})\right).
\eega
\end{enumerate}
\subsection{Betti numbers of the excursion set}\label{sec:excu}
A particular case of Theorems \ref{thm:mainlaw} and \ref{thm:mainE} is when the singular set is the excursion set of the norm, that is, when $W\subset J^0(S^2|\spi{s})=\spi{s}$ is the complement of the radius $u$ ball bundle:
\bega 
Z^W(\sigma_n)&=\{p\in S^2\colon |\sigma(p)|\ge u\};
\\
W=B^{c}_u(\spi{s})&=\{(v^{\otimes s},p)\colon p\in S^2, v\in T_pS^2, |v|\ge1\}.
\eega
Let $\xi\colon \D\to \C$ be the Gaussian random field arising as the local scaling limit  of a sequence of isotropic spin Gaussian fields $\sigma_n$. Thus, its covariance function is of the form:
\be 
K_{\xi}(z,w)=\E\{\xi(z)\overline{\xi(w)}\}=k_\infty(|z-w|)\exp\left(i\beta \Im (z\overline{w})\right).
\ee
In this section we give two simple sufficient conditions to apply point (3) of Theorem \ref{thm:mainE}. They are both based on the observation that the support $F:=\mathrm{supp}(\xi)\subset \mC^\infty(\D,\C)$ of the limit field must contain the function $f:x\mapsto k_\infty(|x|)$ and all of its real multiples. Indeed, by \cite[Theorem 6]{dtgrf} the support $F:=\mathrm{supp}(\xi_r)\subset \mC^\infty(\D,\C)$ is the closed vector subspace generated by functions of the form $K_{\xi_r}(z,\cdot)$, for all points $z\in \D$. Therefore, $f\in F$ because
\bega 
f(z)&:=K_{\xi_r}(0,z)=k_\infty(|z|).
\eega
Notice that $k_\infty$ has always a local maximum at $0$, due to Cauchy-Schwartz inequality: $k_\infty(|z|)\le k_\infty(0)$.
\begin{lemma}\label{lem:b0}
If $k_\infty$ is not constant, then for all $u>0$ there exists $f_u\in \mathrm{supp}(\xi)$ such that $\{|f|\ge u\}$ is non-degenerate and  has a connected component entirely contained in $\mathrm{int}(\D)$.
\end{lemma}
\begin{proof}
By the Cauchy-Schwartz inequality, if $k_\infty$ is not constant then there exists radiuses $t_1,t_2\in(-1,1)$ such that $k_\infty(0)>k_\infty(t)>k_\infty(t_2)>0$.
We see that choosing the function $f_u:=\frac{u}{k_\infty(t)+\e}f\in F$, we have that the excursion set
\be 
\{z\in\D\colon |f_u(z)|\ge u\}=\{z\in\D\colon k_\infty(|z|)\ge k_\infty(t)+\e\}
\ee
must have a (non-empty) connected component $C\subset \{|z|\le t\}$ thus contained in the interior of $\D$. Observe that under these hypotheses we also have that $k_\infty(0)>0$, which implies that the excursion set of $\xi$ is non-degenerate with probability one, by Theorem \ref{thm:mainlaw}. Therefore, the non-degeneracy of the equation $f=0$ can be achieved by a small perturbation of $f$ in the $\mC^0$ topology and within the support, since the property established above is stable under $\mC^0$ perturbations.
\end{proof}
For what concerns the first Betti number $b_1$, an analogous lemma could be stated with the hypotheses that $k_\infty$ has a strict local maximum in $(-1,1)$. However, we can do something a little bit better by exploiting a topological property of the sphere: namely, \emph{Alexander duality}, which tells us that, almosts surely, \be b_1(\{p\in S^2\colon |\sigma(p)|\ge u\})=b_0(\{p\in S^2\colon |\sigma(p)|\le u\})-1.
\ee 
Therefore, to prove that the lower bound in point (2) of Theorem \ref{thm:mainE} is non-trivial (i.e., $c_i^W>0$), it is enough to show the validity of point (3), for the complement of the excursion set, which requires only that $|k_{\infty}|$ is not monotone on $[0,1]$. This strategy is strictly better because, due to the shape of the covariance function $k_\infty$ it is easier to have minima than maxima. Indeed, there may be cases in which point (3) of Theorem \ref{thm:mainE} does not hold, but the Lemma below does.
\begin{lemma}\label{lem:b1}
If there are $0<t_1<t_2<1$ such that $|k_\infty(t_1)|<|k_\infty(t_2)|$, then for all $u>0$ there exists $f_u\in \mathrm{supp}(\xi)$ such that $\{|f|\le u\}$ is non-degenerate and  has a connected component entirely contained in $\mathrm{int}(\D)$.
\end{lemma}
\begin{proof}
The proof follows the same lines as that of the previous lemma.  We choose again the same function $f_u:=\frac{u}{|k_\infty(t_1)|+\e}f\in F$, but this time, the hypothesis implies that the set $\{|f|\le u\}$ has a connected component $C$ contained in $\{|z|\le t_2\}$.
\end{proof}
\subsection{Excursion sets -- proof of Theorem \ref{thm:berry}}
To establish Theorem \ref{thm:berry}, it is sufficient to notice that the conditions for the validity of Theorem \ref{thm:mainE} are met, as shown in a. above. As a consequence, to derive the expected values of Lipschitz-Killing curvatures it is sufficient to investigate the case with spin zero; this is done in appendix \ref{prop:spinzero}, exploiting the general form of the Gaussian kinematic formula, see \cite{AdlerTaylor}.
For the number of connected components, see Lemma \ref{lem:b0}. 
For $b_1$ one can modify the scaling sequence $\rho_n$ by a constant factor $c$ strictly bigger than the second zero of $J_0$ and then run the argument discussed in section \ref{sec:excu} to prove that $c^W_1>0$ using Lemma \ref{lem:b1} and Alexander's duality.
\begin{remark}[Excursion sets in the middle and Bargmann-Fock regimes]\label{rem:excpol}
Of course, Theorem \ref{thm:mainE} can be applied to the case $|s_\ell|=\ell-\gap$, with $\gap\in \N$. Moreover, by Lemma \ref{lem:b0}, the lower bound for the number of connected components is non-trivial.
However, it should be noted that in this framework we are not able to give a lower bound for the first Betti number $b_1$. In the case of Bargmann-Fock Limiting Behaviour, the reason for this failure is easy to get: as well known, because of the maximum principle the excursion set of the norm of a  holomorphic function must be convex. This property continues to hold when the function is multiplied by the concave function $\exp({-\frac{|x|^2}{4}})$.
\end{remark}

%% file: f_Appendices.tex
\section{Expected number of zeroes: proof of Theorem \ref{thm:Enumber}}\label{sec:zeroproof}
\subsection{The covariance function of the rescaled field}
Let $X\colon SO(3)\to\C$ be the pull-back of a Gaussian isotropic spin-$s$ function $\sigma$, with circular covariance function $k\colon \R\to \R$, defined by:
\be\label{eq:Gappax}
\Gamma\left\{R(\f,\h,\psi)\right\}=k(\h)e^{is(\f+\psi)},
\ee
where $\Gamma(g):=\E\{X(\mathbb{1})\overline{X(g)}\}$, see Section \ref{sec:SpinRandomFields}.
Define $\xi\colon \C\to \C$ as the field:
\be 
\xi(\h e^{i\f}):=X(R(\f,\h,-\f)).
\ee
\begin{lemma}\label{lemma:recaledcov} Let $K_\xi(z,w)=\E\{\xi(z)\overline{\xi(w)}\}$ be the covariance function of $\xi$. Then
\begin{enumerate}[(i)]
    \item $K_\xi(0,0)=k(0)$;
    \item $\frac{\de}{\de x}K_\xi(x,0)|_{x=0}=\frac{\de}{\de y}K_\xi(iy,0)|_{y=0}=0$;
    \item $\frac{\de^2}{\de x^2}K_\xi(x,0)|_{x=0}=\frac{\de^2}{\de y^2}K_\xi(iy,0)|_{y=0}=k''(0)$;
    \item $\frac{\de^2}{\de x\de y}K_\xi(x,iy)|_{x=y=0}=-i\frac{s}{2}k(0)$.
\end{enumerate}
\end{lemma}
\begin{proof}
We start by observing that:
\bega
K_\xi(xe^{i\f},ye^{i\psi})&=\E\left\{X(R(\f,x,-\f))\overline{X(R(\psi,y,-\psi))}\right\}=
\\
&=
\Gamma\left\{R(\f,x,-\f)^{-1}R(\psi,y,-\psi)\right\}
\\
&=\Gamma\left\{R_2(-x)R_3(\psi-\f)R_2(y)\right\}e^{-is(\psi-\f)}
\eega
Thus $(i)$ is obvious, while $(ii)$ and $(iii)$ follow from 
\be 
k(t)=K_\xi(0,te^{i\psi})=K_\xi(-te^{i\f},0)
\ee 
and the fact that $k$ is even. 
The difficult case is:
\be 
\frac{\de^2}{\de x\de y}K_\xi(x,iy)|_{x=y=0}=\frac{\de^2}{\de x\de y}|_{x=y=0}\Gamma\left\{R_2(-x)R_3\left(\frac\pi2\right)R_2(y)\right\}(-i)^s=\dots
\ee
Let us write $R_2(-x)R_3\left(\frac\pi2\right)R_2(y)=R({\f}, {\h},{\psi})$, with ${\f}, {\h},{\psi}$ determined from Lemma \ref{lemma:trasfrule}:
\bega
\a:=\cos\left(\frac{\h}{2}\right)e^{i\frac{{\f}+{\psi}}{2}}
&=
\cos\left(\frac{x}{2}\right)\cos\left(\frac{y}{2}\right)e^{i\frac\pi4}+\sin\left(\frac{x}{2}\right)\sin\left(\frac{y}{2}\right)e^{-i\frac\pi4};
\\
\beta:=\sin\left(\frac{\h}{2}\right)e^{i\frac{-{\f}+{\psi}}{2}}
&=
\sin\left(\frac{-x}{2}\right)\cos\left(\frac{y}{2}\right)e^{i\frac\pi4}+\cos\left(\frac{x}{2}\right)\sin\left(\frac{y}{2}\right)e^{-i\frac\pi4}.
\eega
Then we have:
\bega 
\dots&=(-i)^s\frac{\de^2}{\de x\de y}|_{x=y=0}k({\h})e^{is(\f+\psi)}
\\
&=
(-i)^s\frac{\de}{\de x}\Big|_{x=0}\left(
k(x)\frac{\de}{\de y}\Big|_{y=0}e^{is(\f+\psi)}+
e^{is\frac\pi2}\frac{\de}{\de y}\Big|_{y=0}k(\h)
\right)
\\
&=
(-i)^s\frac{\de}{\de x}\Big|_{x=0}\left(
k(x)\frac{\de}{\de y}\Big|_{y=0}\left(\frac{\a^{2}}{|\a|^{2}}\right)^s+
e^{is\frac\pi2}k'(x)\frac{\de \h}{\de y}\Big|_{y=0}
\right)
\\
&=
(-i)^s\frac{\de}{\de x}\Big|_{x=0}\left(
k(x)s\left(\frac{\a}{|\a|}\right)^{2s-2}
\frac{\de}{\de y}\left(\frac{\a^{2}}{|\a|^{2}}\right)\Big|_{y=0}+
e^{is\frac\pi2}k'(x)\frac{\de \h}{\de y}\Big|_{y=0}
\right)
=\dots
\eega
Let us examine each derivative separately. Using that $\frac{\de \a}{\de y}|_{y=0}=\frac12\sin(\frac x2)e^{-i\frac\pi4}$ and that ${\a|_{y=0}=\cos(\frac x2)e^{i\frac\pi4}}$ we get:
\bega
\frac{\de}{\de y}\Big|_{y=0}\left(\frac{\a^2}{|\a|^2}\right)&
=2\frac{\a}{|\a|^2}\frac{\de \a}{\de y}\Big|_{y=0}-\frac{\a^2}{|\a|^4}\langle2\a,\frac{\de \a}{\de y}\Big|_{y=0}\rangle
\\
&=
\tan\left(\frac x2\right)-(\dots)\langle e^{i\frac\pi4},e^{-i\frac\pi4}\rangle=\tan\left(\frac x2\right).
\eega
Therefore
\be 
\frac{\de}{\de y}\Big|_{y=0}\left(\frac{\a}{|\a|}\right)=\left(\frac{\a}{|\a|}\right)^{-1}\frac{\de}{\de y}\Big|_{y=0}\left(\frac{\a^2}{|\a|^2}\right)=\frac12\tan\left(\frac x2\right)e^{-i\frac\pi4},
\ee
from which we deduce that
\bega 
\frac12\sin\left(\frac x2\right)e^{-i\frac\pi4}&=\frac{\de \a}{\de y}\Big|_{y=0}
\\
&=-\frac{1}2\sin\left(\frac x2\right)\left(\frac{\a}{|\a|}\right)\frac{\de \h}{\de y}\Big|_{y=0}+\cos\left(\frac x2\right)\frac{\de}{\de y}\Big|_{y=0}\left(\frac{\a}{|\a|}\right)
\\
&=-\frac{e^{i\frac\pi4}}2\sin\left(\frac x2\right)\frac{\de \h}{\de y}\Big|_{y=0}+\cos\left(\frac x2\right)\frac12\tan\left(\frac x2\right)e^{-i\frac\pi4}.
\eega
Thus $\frac{\de \h}{\de y}\Big|_{y=0}=0$. Now, we can take up the main line of computations to conclude the proof:
\bega 
\dots
&=(-i)^s\frac{\de}{\de x}\Big|_{x=0}\left(
k(x)se^{i(2s-2)\frac\pi4}\tan\left(\frac x2\right)
\right)
\\
&=-ik(0)s\frac12.
\eega 
\end{proof}
\begin{cor}\label{cor:cov1stjet} The first jet of the rescaled field $(\xi,\frac{\de \xi}{\de x},\frac{\de \xi}{\de y})$ has the following covariance matrix:
\be 
\E\left\{\begin{pmatrix}
\xi \\ \frac{\de \xi}{\de x} \\ \frac{\de \xi}{\de y}
\end{pmatrix}\begin{pmatrix}
\overline{\xi} & \overline{\frac{\de \xi}{\de x}} & \overline{\frac{\de \xi}{\de y}}
\end{pmatrix}\right\}=\begin{pmatrix}
1 & 0 &0 
\\ 0 & -k''(0) & -i\frac{s}{2}k(0) 
\\ 0 & i\frac{s}{2}k(0) & -k''(0)
\end{pmatrix}.
\ee
\end{cor}
\begin{remark}
Note that for $s\neq 0$, the first order derivatives and hence the ``real and complex'' components of the spin bundle are not independent for any choice of local coordinates.
\end{remark}

\subsection{Proof of Theorem \ref{thm:Enumber}}
Since $\sigma$ is isotropic, there exists a constant $c\ge 0$ such that 
\be 
\E\#\{p\in A\colon \sigma(p)=0\}=c\cdot\vol(A)
\ee
for every (Borel) subset $A\subset S^2$.
Let us consider the field $\xi\colon\D\to \C$ given by
\be 
\xi(\h e^it)=X(R(\f,\h,-\f)),
\ee
where $X\colon SO(3)\to\C$ is the pull-back random field of $\sigma$. Notice that $\xi$ corresponds to the rescaled field of Definition \ref{defi:rescaled field} in the case $\rho_\ell=1$, thus it represents the section $\sigma$ with respect to a trivialization of the bundle $\spi{s}$ over the local chart given by the exponential map at the north pole. In particular, the number of zeroes of $\sigma$ on a spherical disk of radius $\e$ around the north pole equals the number of zeroes of $\xi$ in $\e\D$, for all $\e>0$.
Moreover, by Kac-Rice formula, applied to $\xi$ we have
\bega
\E\#\{z=x+iy\in A\colon \xi(z)=0\}
\\=\int_{A}\E\left\{\left|\det\begin{pmatrix}\frac{\de \xi}{\de x}  \frac{\de \xi}{\de y}\end{pmatrix}\right|\bigg | \xi(z)=0\right\}\rho_{\xi(z)}(0)dxdy.
\eega
Where $\rho_{\xi(z)}\colon \C\to \R_+$ is the density of the random variable $\xi(z)$ and $A\subset \D$ is a Borel subset. Combining these two formulas we deduce that
\be\label{eq:c}
c=\E\left\{\left|\det\begin{pmatrix}\frac{\de \xi}{\de x} & \frac{\de \xi}{\de y}\end{pmatrix}\right|\bigg | \xi(z)=0\right\}\rho_{\xi(z)}(0)\frac{1}{\sqrt{g(z)}}
\ee
for any $z\in\D$, where $\sqrt{g(z)}dxdy$ is the area form of $S^2$ written in the coordinates $x,y$. Evaluating \eqref{eq:c} at the point $z=0$ yields
\bega
\E\#\{\xi=0\}&= \E\left\{\left|\det\begin{pmatrix}\frac{\de \xi}{\de x} & \frac{\de \xi}{\de y}\end{pmatrix}\right|\bigg | \xi(0)=0\right\}\rho_{\xi(z)}(0) \vol(S^2)
\\
&=
\E\left\{\left|\det\begin{pmatrix}\frac{\de \xi}{\de x} & \frac{\de \xi}{\de y}\end{pmatrix}\right|\right\}\frac{4}{k(0)}.
\eega
Where in the last equality we used the fact that $\xi(0)$ is independent from $d_0\xi$, as it can be seen from Corollary \ref{cor:cov1stjet}.

To end the computation it is convenient to express the differential of $\xi$ in terms of the Wirtinger derivatives:
\bega
\frac{\de \xi}{\de z}&:=\frac12\left(\frac{\de \xi}{\de x}-i\frac{\de \xi}{\de y}\right),
\qquad
\frac{\de \xi}{\de \bar{z}}&:=\frac12\left(\frac{\de \xi}{\de x}+i\frac{\de \xi}{\de y}\right).
\eega
As a consequence of the shape of the covariance matrix established by Corollary \ref{cor:cov1stjet}, we have that the two complex random variables  $\frac{\de \xi}{\de z}$ and $\frac{\de \xi}{\de \bar{z}}$ are independent with variances:
\bega
\E\left\{\left|\frac{\de \xi}{\de {z}}\right|^2\right\}&=-\frac12 k''(0)+\frac{s}{4}k(0),
\qquad
\E\left\{\left|\frac{\de \xi}{\de \bar{z}}\right|^2\right\}&=-\frac12 k''(0)-\frac{s}{4}k(0).
\eega
Moreover, 
\be 
\det\begin{pmatrix}
\frac{\de \xi}{\de x} 
& 
\frac{\de \xi}{\de y}
\end{pmatrix} 
= 
\left|\frac{\de \xi}{\de {z}}\right|^2-\left|\frac{\de \xi}{\de \bar{z}}\right|^2.
\ee
\begin{lemma}\label{lem:integral}
Let $\gamma_1,\gamma_2$ be two independent complex normal variables with variances $a,b>0$, we have
\be 
\E\left\{\left||\gamma_1|^2-|\gamma_2|^2\right|\right\}=\frac{a^2+b^2}{a+b}
\ee
\end{lemma}
\begin{proof}
The proof is a straightforward computation of an integral and is omitted.
\end{proof}
A direct computation concludes the proof of Theorem \ref{thm:Enumber}.
\bega 
\E\#\{\xi=0\}&=\frac{4}{k(0)}\E\left\{\left|\det\begin{pmatrix}\frac{\de \xi}{\de x} & \frac{\de \xi}{\de y}\end{pmatrix}\right|\right\}
\\
&=
\frac{4}{k(0)}\E\left\{\left|\left|\frac{\de \xi}{\de {z}}\right|^2-\left|\frac{\de \xi}{\de \bar{z}}\right|^2\right|\right\}
\\
&=
\frac{4}{k(0)}\frac{(-\frac12 k''(0)+\frac{s}{4}k(0))^2+(-\frac12k''(0)-\frac{s}{4}k(0))^2}{-k''(0)}
\\
&=
\frac{2 k''(0)^2+\frac12 s^2k(0)^2}{-k(0)k''(0)}.
\eega
\section{Proof of technical lemmas}
\subsection{Expected nodal volume of Gaussian fields}
A smooth Gaussian random field $X\colon M\to \R$, defines a semipositive definite scalar product on $T_pM$ via the following formula (see \cite{AdlerTaylor}) 
\be 
g^X_p(v,w):=\E\left\{d_pX(v)d_pX(w)\right\}.
\ee
Such tensor is a Riemannian metric (i.e. it is positive definite) if and only if $d_pX$ is a non-degenerate Gaussian vector for every $p\in M$. In this case, we call it the \emph{Adler-Taylor metric} of $X$. Many probabilistic features of $X$ are related to the Riemannian geometry of $(M,g^X)$, starting from the expected nodal volume, i.e. the Hausdorff measure of $\mathcal{H}^{m-1}(X^{-1}(0))$. The formula that use, in particular in the proof of Theorem \ref{thm:mainE} is the following.
Let $s_k:=\mathcal{H}^k(S^k)$ be the volume of the $k$-dimensional sphere.
\begin{prop}\label{prop:vol}
Let $(M,g)$ be a compact Riemannian manifold and let
Let $X_{i}\sim X\colon M\to \R$ be i.i.d. copies of a smooth Gaussian random field $X$ such that $g^X=\frac{\lambda}2 g$ and with constant variance $\E\{|X(p)|^2\}=\sigma^2$. Let $C\subset M$ be a smooth immersed submanifold of dimension $d$ and let $f\in\mC^\infty(M)$. Then the integral of $f|_C$ with respect to the $d$-dimensional Hausdorff measure $\mathcal{H}_g^d$ (i.e. the Riemannian volume measure of $C$) is
\bega
\int_Cf(p)d\mathcal{H}^{d}_g(p)=\left(\frac{\lambda}{2\sigma^2}\right)^{-\frac d2}\frac{ s_{d}}{2}\E\left\{\sum_{p\in C\cap\{ X_{1}=\dots=X_{d}=0\}}f(p)\right\}.
\eega
\end{prop}
\begin{proof}
It is sufficient to show the formula for $f=1$ and then extend it by dominated convergence. Moreover, observe that $Y=X|_C$ is a smooth Gaussian field on $C$ and $g^Y=\frac{\lambda}2 g|_C$, therefore we can assume that $C=M$. Thus, we only have to prove that
\bega\label{eq:uador}
\mathcal{H}_g^{m}(M)=\left(\frac{\lambda}{2\sigma^2}\right)^{-\frac m2}\frac{ s_{m}}{2}\E(\#\left(\{ X_{1}=\dots=X_{m}=0\}\right)).
\eega
We can further reduce to the case $\sigma=\frac{\lambda}2=1$, by replacing $X$ with $Y=\frac{1}{\sigma}X$ and $g$ with $g^Y$. Indeed observe that the right hand side doesn't change, while the left hand side changes as:
\be 
g^Y=\frac{\lambda}{2\sigma^2}g\quad  \text{and} \quad \mathcal{H}_g^{m}(M)\e^{\frac m2}=\mathcal{H}_{\e g}^{m}(M).
\ee
Kac-Rice formula tells us in particular that the two quantities in \eqref{eq:uador} are proportional. The correct constant can be thus deduced from a simple case without making computations. The simplest case is that of the standard sphere: $M=S^{m}$ with $X(p):=\gamma^Tp$ for $\gamma\sim N(0,\mathbb{1}_k)$, so that $g^X$ is the standard round metric. Since the random set $C\cap\{ X_{1}=\dots=X_{m}=0\}$ consists almost surely of $2$ points, we conclude.
\end{proof}
\begin{remark}
The intuition behind the identity $g^X=\frac{\lambda}2 g$ is that the conformal factor is $\frac\lambda 2$, when $X$ is a Gaussian eigenfunction in $\ker (\Delta-\lambda)$ on $S^2$ (see \ref{prop:spinzero}).
\end{remark}
\subsection{Explicit formulas for L-K curvatures for spin equal to zero}
The next result follows quite directly from the general Gaussian kinematic formula of \cite{AdlerTaylor}.
\begin{prop}[L-K curvatures for spin$=0$]\label{prop:spinzero}
Let $\sigma_0\colon S^2\to \C$ be a complex isotropic smooth Gaussian random field having independent real and imaginary parts and with $\sigma(p)\sim N_\C(0,1)$. Let $k(\h)$ be its circular covariance function. Then for any $u>0$, we have the following identities.
\begin{enumerate}[i.]
\item $k(\h)=\E\{\sigma_0(p)\sigma_0(q)\}$ for every $p,q\in S^2$ such that $\mathrm{dist}_{S^2}(p,q)=\h$.
\item $\frac{\lambda}2:=|k''(0)|=\E\{|\de_v\sigma_0|^2\}$ for any unit tangent vector $v\in TS^2$. This is the conformal factor of the Adler-Taylor metric $g^{\sigma_0}$ induced by $\sigma_0$, see \cite{AdlerTaylor}, meaning that
$ 
g^{\sigma_0}=\frac\lambda 2 g_{S^2}.
$
\item $\E\#\{\sigma_0=0\}=\lambda$.
    \item $
    \E\vol_2(\{|\sigma_0|\ge u\})=4\pi e^{-\frac{u^2}{2}}$.
    \item 
    $
    \E\vol_1(\{|\sigma_0|= u\})=\lambda^\frac12 \cdot 2\pi^{\frac32}ue^{-\frac{u^2}{2}}$.
    \item $
    \E\chi(\{|\sigma_0|\ge u\})=
    (\lambda\cdot (u^2-1)+2)e^{-\frac{u^2}{2}}$.
\end{enumerate}
\end{prop}
\begin{remark}
Notice that $\E\chi(\{|\sigma_0|\ge u\})$ is not continuous at $u= 0$, in that
\be\label{eq:echicont} 
\lim_{u\to 0^+}\E\chi(\{|\sigma_0|\ge u\})=2-\E\#\{\sigma_0=0\},
\ee
while $\chi(S^2)=2$. This should not surprise, in that for small values of $u$, the escursion set $\{|\sigma_0|\ge u\}$ is just the complement of a small neighborhood of the zero set, thus \eqref{eq:echicont} holds almost surely, without the expectations. On the other hand, it is clear that the two identities $iv.$ and $v.$ are still true for $u=0$.
\end{remark}
\begin{proof}
$i.$ and the first part of $ii.$ are a straightforward consequence of isotropy. Moreover, notice that $iii.$ follows from Theorem \ref{thm:Enumber}. The fact that $\frac{\lambda}2$ is indeed the conformal factor can be deduced combining $iii.$ with Proposition \ref{prop:vol}.
\be 
\E\#\{\sigma_0=0\}=\frac{2}{4\pi}\vol_{g^{\sigma_0}}(S^2).
\ee
This ends the proof of $ii$.
To show $iv,v$ and $vi$, we apply \cite[Theorem 15.9.4]{AdlerTaylor}, which states that for every $j=0,1,2$, we have the formula:
\be\label{eq:lipescAT}
\E\mathcal{L}_i\left(\{|\sigma_0|\ge u\}
\right)
=\sum_{j=0}^{2-i}{{i+j}\choose{j}}\frac{\w_{i+j}}{\w_i\w_j}\mathcal{L}_{i+j}(S^2)\rho_j(u^2),
\ee
where $\w_n=\pi^{\frac n2}\Gamma(\frac n2 +1)^{-1}$ is the volume of the standard unit ball of dimension $n$; $\mathcal{L}_{j}$ are the Lipschitz-Killing curvatures computed with respect to the metric $g^{\sigma_0}$; the coefficients $\rho_j(u)$ are universal functions of $u$, that can be computed via the formula in \cite[Theorem 15.10.1]{AdlerTaylor} (the formula in the book is for the Lipschitz-Killing curvatures of the set $\{|\sigma_0|^2\ge u\}$, thus we have to evaluate the formula in $u^2$ instead than $u$):
\be 
\rho_2(u^2)=\frac{e^{-\frac{u^2}{2}}}{2\pi}(u^2-1), \qquad 
\rho_1(u^2)=\frac{e^{-\frac{u^2}{2}}}{\sqrt{2\pi}}u,
\qquad 
\rho_0(u^2)=e^{-\frac{u^2}{2}}.
\ee
For $A\subset S^2$ smooth submanifold with boundary, we have
\be
\mathcal{L}_2(A)=\frac{\lambda}{2}\vol_2(A),
    \qquad \mathcal{L}_1(A)=\frac12\left(\frac{\lambda}{2}\right)^\frac12\vol_1(\de A),
    \qquad \mathcal{L}_0(A)=\chi(A).
\ee
Here, $\vol_1$ and $\vol_2$ are now meant with respect to the standard metric of $S^2$.
Therefore formula \eqref{eq:lipescAT} gives:
\bega 
\E\chi\{|\sigma|\ge u\}
=
e^{-\frac{u^2}{2}}
\left(\chi(S^2)+\frac{(u^2-1)}{2\pi}\frac\lambda 2\vol_2(S^2)\right)
=
e^{-\frac{u^2}{2}}
\left(2+(u^2-1)\lambda \right);
\eega
\bega 
\frac12\left(\frac{\lambda}{2}\right)^\frac12\E\vol_1\{|\sigma|= u\}
=
e^{-\frac{u^2}{2}}
{2\choose1} \frac{\w_2}{\w_1^2}\frac{u}{\sqrt{2\pi}}\frac{\lambda}{2}\vol_2(S^2)
=
e^{-\frac{u^2}{2}}
\frac{\pi^{\frac32}}{\sqrt2}u\lambda;
\eega
\bega 
\frac{\lambda}{2}\E\vol_2\{|\sigma|\ge u\}
=
e^{-\frac{u^2}{2}}
\frac{\lambda}{2}\vol_2(S^2)
=
e^{-\frac{u^2}{2}}\frac{\lambda}{2}4\pi.
\eega
\end{proof}
\section{Alternative to Hilb's asymptotic}\label{app:hilb}
This appendix collects some explicit computations which are instrumental for the derivation of the limiting behavior of the covariances of monochromatic spin fields.
\begin{defi}(Bessel functions of the first kind)
Let $n\in \Z$. The \emph{Bessel function of the first kind} of order $n$, denoted by $J_n(x)$, is a (regular at $0$) solution of the \emph{Bessel equation} $x^2J_n''(x)+xJ_n'(x)+(x^2-n^2)J_n=0$. It is an anaytic function $J_n\colon \R\to \R$ described by the following power series: if $n\ge 0$,
\bega
J_n(x)&=\sum_{j=0}^{\infty}\frac{(-1)^j}{j!(j+n)!}\left(\frac{x}{2}\right)^{2j+n}; 
\\
J_{-n}(x)&=\sum_{j\ge n}^{\infty}\frac{(-1)^j}{j!(j-n)!}\left(\frac{x}{2}\right)^{2j-n}=(-1)^nJ_n(x).
\eega
\end{defi}
\begin{lemma}
Let $\a,\beta\colon SO(3)\to \C$ be the two functions defined (see in \cite[Sec $3.2.1$]{libro}) by
\be 
\a(R(\f,\h,\psi))=\cos\left(\frac{\h}{2}\right)e^{i(\frac{\f+\psi}{2})}; \quad \beta(R(\f,\h,\psi))=\sin\left(\frac{\h}{2}\right)e^{i(\frac{\f-\psi}{2})}
\ee
Then (with a little abuse of notation), for all $\ell\ge |s|$, we have
\bega
&D_{m,-s}^\ell(\a,\beta):=D_{m, -s}^\ell(g)=
\\
&\sum_{j\ge \max\{0,-(m+s)\}}^{\min\{\ell-s,\ell-m\}}\frac{(-1)^{s+m}\sqrt{(\ell+s)!(\ell-s)!(\ell+m)!(\ell-m)!}}{(\ell-s-j)!(\ell-m-j)!}\frac{(-1)^{j}}{j!(s+m+j)!}\times
\\
&\times \a^{\ell-m-j}\overline{\a}^{\ell-s-j}\beta^{j}\overline{\beta}^{j+m+s}.
\eega
\end{lemma}
\begin{proof}
(It is Proposition $3.7$ in the book \cite{libro}).
By Definition, we have
\be 
{\binom{2\ell}{\ell+s}}^\frac12(\a z_1-\overline{\beta}z_2)^{\ell+s}(\beta z_1+\overline{\a}z_2)^{\ell-s}=\sum_{m=-\ell}^\ell \binom{2\ell}{\ell-m} D_{m,-s}^\ell(\a,\beta)z_1^{\ell-m}z_2^{\ell+m}.
\ee 
Expanding the first term we get
\begin{multline}
{\binom{2\ell}{\ell+s}}^\frac12(\a z_1-\overline{\beta}z_2)^{\ell+s}(\beta z_1+\overline{\a}z_2)^{\ell-s}
\\
=
{\binom{2\ell}{\ell+s}}^\frac12\sum_{i=0}^{\ell+s}\sum_{j=0}^{\ell-s}\binom{\ell+s}{I}\binom{\ell-s}{j}\a^i\overline{\a}^{\ell-s-j}\beta^j\overline{\beta}^{\ell+s-i}(-1)^{\ell+s-i}z_1^{i+j}z_2^{2\ell-(i+j)}
\\
={\binom{2\ell}{\ell+s}}^\frac12\sum_{m=-\ell}^{\ell}\sum_{j=0}^{\ell-s}\binom{\ell-s}{j}\binom{\ell+s}{\ell-m-j}\a^{\ell-m-j}\overline{\a}^{\ell-s-j}\beta^{j}\overline{\beta}^{j+m+s}\times 
\\
\times (-1)^{j+m+s}z_1^{\ell-m}z_2^{\ell+m},
\end{multline}
where $m$ and $j$ must satisfy the additional constraints $\ell-m-j=i\ge 0$ and
$m+s+j=(\ell+s)-i\ge 0$. Therefore we can conclude by observing that
\be 
\frac{ 
	{\binom{2\ell}{\ell+s}}^\frac12\binom{\ell-s}{j}\binom{\ell+s}{\ell-m-j}\a^{\ell-m-j}
}{
	{\binom{2\ell}{\ell-m}}^\frac12
}
=
\frac{\sqrt{(\ell+s)!(\ell-s)!(\ell+m)!(\ell-m)!}}{j!(\ell-s-j)!(\ell-m-j)!(s+m+j)!}.
\ee
\end{proof}
\begin{cor}For $\ell\ge |s|$, we have
\begin{multline}
d_{m,-s}^\ell(\h)=
\\
\sum_{j\ge \max\{0,-(m+s)\}}^{\min\{\ell-s,\ell-m\}}\frac{(-1)^{s+m}\sqrt{(\ell+s)!(\ell-s)!(\ell+m)!(\ell-m)!}}{(\ell-s-j)!(\ell-m-j)!}\times
\\
\times\left(\cos \frac\h 2\right)^{2(\ell-j)-(m+s)}
\frac{(-1)^{j}}{j!(s+m+j)!}\left(\sin \frac\h 2\right)^{2j+(m+s)}.
\end{multline}
\end{cor}
\begin{thm}\label{thm:smoothconv}
For $\ell\to+\infty$, we have that
\be 
d^{\ell}_{m,-s}\left(\frac{x}{\ell}\right)\xrightarrow[\ell\to+\infty]{} (-1)^{m+s}J_{m+s}(x).
\ee
The convergence holds in the $\mathcal{C}^\infty$ topology.
\end{thm}
\begin{proof}
Since for $\ell\to +\infty$ we have that 
\be\label{eq:coseno} 
\a=\left(\cos{\frac{x}{2\ell}}\right)^\ell\sim_{\mC^\infty} 1; \quad \text{and} \quad \beta=\sin{\frac{x}{2\ell}}\sim_{\mC^\infty} \frac{x}{2\ell},
\ee
we can restrict our study to the function 
\begin{multline}
D_{m,-s}^\ell\left(1,\frac{x}{\ell}\right)=
\\
= \sum_{j\ge \max\{0,-(m+s)\}}^{\min\{\ell-s,\ell-m\}}
\frac{(-1)^{s+m}\sqrt{(\ell+s)!(\ell-s)!(\ell+m)!(\ell-m)!}}{(\ell-s-j)!(\ell-m-j)!}\times 
\\
\times\frac{(-1)^{j}}{j!(s+m+j)!}
\left(\frac{x}{2\ell}\right)^{2j+(m+s)}.
\end{multline}
Since the above function is a power series with convergence radius $=+\infty$, its convergence in $\mC^\infty$ can be checked one coefficient at a time. Now, observe that (assume, for simplicity, that $m\ge 0$ and $s\ge 0$. In the other cases, the argument is essentially the same)
\be
\frac{\sqrt{(\ell+s)!(\ell-s)!(\ell+m)!(\ell-m)!}}{(\ell-s-j)!(\ell-m-j)!}\frac{1}{\ell^{2j+m+s}}=1+o(1).
\ee
It follows that 
\bega 
D_{m,-s}^\ell\left(1,\frac{x}{\ell}\right)
\xrightarrow[\ell\to+\infty]{} 
&\sum_{j= \max\{0,-(m+s)\}}^{\infty}(-1)^{s+m}\frac{(-1)^{j}}{j!(s+m+j)!}\left(\frac{x}{2}\right)^{2j+(m+s)}
\\
&=(-1)^{s+m}J_{s+m}(x).
\eega
(both if $s+m\ge 0$ and if $s+m<0$).
\end{proof}
\begin{remark}\label{rem:sss} 
As $s\to+\infty$,
\be 
d^s_{-s,-s}\left(\frac{x}{\sqrt{s}}\right)=\left(\cos\left(\frac{x}{2\sqrt{s}}\right)\right)^{2s}\sim e^{-\frac{x^2}{4}}.
\ee
\end{remark}

\subsection{Limit of the covariance of monocromatic fields} Let us take, as before, $|s_\ell|=\ell-\gap_\ell$.
Now note that
\begin{multline}
d_{-s_\ell,-s_\ell}^{\ell}\left(\frac{x}{\sqrt{(\gap+1)(2\ell-\gap)}}\right)
=\sum_{j=0}^{\gap} 
\frac{(2\ell-\gap)!\gap!}{(\gap-j)!(2\ell-\gap-j)!}\frac{(-1)^{j}}{j!j!}\times
\\
\times\sin\left(\frac{x}{2\sqrt{(\gap+1)(2\ell-\gap)}}\right)^{2j}\cos\left(\frac{x}{2\sqrt{(\gap+1)(2\ell-\gap)}}\right)^{2(\ell-j)}
=
\\
\sum_{j=0}^{\gap} 
\frac{(2\ell-\gap)!}{(2\ell-\gap-j)!(2\ell-\gap)^j}
\frac{\gap!}{(\gap-j)!(\gap+1)^j}
\frac{(-1)^{j}}{j!j!}
\\
\left(\frac{x}{2}\right)^{2j}
\left(1-\frac{x^2}{4((\gap+1)(2\ell-\gap))}\right)^{2\ell}
+o_{\ell\to+\infty}(1)
=\dots
\end{multline}
Note also that we have always $2\ell-\gap_\ell\to +\infty$. 

We must consider two cases. In the first $\gap_\ell\to \gap\in \N$, which is equivalent to $\gap_\ell$ being fixed; in this case the shrinking rate is $\rho_\ell=O(\frac{1}{\sqrt{2(\gap+1)\ell}})$ and thus $\beta=\frac{1}{2(\gap+1)}$, and, moreover, we obtain easily
\bega\label{eq:Polynomiallimit}
\dots
&=
\xrightarrow[s\to+\infty]{}
\sum_{j=0}^\gap\frac{\gap!}{(\gap+1)^j(\gap-j)!}\frac{(-1)^j}{j!j!}\left(\frac{x}{2}\right)^{2j}e^{-\frac{x^2}{4(\gap+1)}}=:M_\gap(x).
\eega
On the other hand, in the second case $\gap_\ell\to +\infty$ and we obtain 
\bega
\dots
&=
\xrightarrow[s\to+\infty]{}
\sum_{j=0}^\infty\frac{(-1)^j}{j!j!}\left(\frac{x}{2}\right)^{2j}=J_0(x).
\eega
In this second case, we have 
\be 
\beta=\lim_{\ell\to+\infty}\frac{\ell-\gap}{(\gap+1)(2\ell-\gap)}=\lim_{\ell\to+\infty}\frac{\ell-\gap}{(\ell+(\ell-\gap))(\gap+1)}
=0.
\ee 
Note that this second scenario covers the Berry regime.
Notice also that $M_r(x)=1-\frac{x^2}{4}+O(x^4)=J_0(x)+O(x^4)$ and $M_0(x)=e^{-\frac{x^2}{4}}$ is the real part of the covariance function of the complex Bargmann-Fock field.

\section{Berry's Complex Random Wave Model}\label{app:berry}

Berry's Complex Random Wave Model is a complex Gaussian random field $\xi := \lbrace \xi(x), x\in \mathbb R^2\rbrace$ on $\mathbb R^2$ represented as 
\begin{equation*}
    \xi(x) = \sum_{n\in \mathbb Z} a_n J_{|n|}(r) e^{in\theta}, \qquad x=(r,\theta),
\end{equation*}
where $J_\alpha$ denotes the Bessel function of the first kind of order $\alpha$ and $(a_n)_n$ is a sequence of i.i.d. standard complex Gaussian random variables. The sample paths are a.s. $\mathcal C^\infty$ functions. 

It is straightforward to check that $\xi$ a.s. solves the Helmholtz equation on the Euclidean plane, i.e. 
\begin{equation*}
    \Delta_{\mathbb R^2} \xi = - \xi \qquad a.s. 
\end{equation*}
Indeed, writing $\Delta_{\mathbb R^2}$ in polar coordinates 
\begin{multline}
   \left ( \frac{\partial^2}{\partial r^2} + \frac{1}{r}\frac{\partial}{\partial r} + \frac{1}{r^2} \frac{\partial^2}{\partial \theta^2}\right ) \xi (x) = \sum_{n\in \mathbb Z} \frac{a_n}{r^2} \left (r^2\frac{\partial^2}{\partial r^2} + r\frac{\partial}{\partial r}- n^2  \right ) J_{|n|}(r) e^{in\theta} \cr 
  = \sum_{n\in \mathbb Z} \frac{a_n}{r^2} \underbrace{\left (r^2\frac{\partial^2}{\partial r^2} + r\frac{\partial}{\partial r}- n^2 + r^2 \right ) J_{|n|}(r)}_{=0} e^{in\theta} - \sum_{n\in \mathbb Z} a_n  J_{|n|}(r) e^{in\theta}
   = -\xi(x). 
\end{multline}

\begin{lemma}
Every (smooth) solution $f$ of the Helmholtz equation is of the form 
\begin{equation}\label{f_sol_plane}
    f(x) = \sum_{n\in \mathbb Z} \gamma_n J_{|n|}(r) e^{in\theta}.
\end{equation}
\end{lemma}
\begin{proof}
Let $\mathbb D$ denote the unit disc, then $f$ restricted to $\mathbb D$ can be written as
\begin{equation}
    f(x) = \sum_{n\in \mathbb Z} R_n(r) e^{in\theta},
\end{equation}
for some $C^\infty$ functions $R_n$.  Since $f$ solves the Helmholtz equation, the following holds true for every $r$, $\theta$
\begin{equation}
   \sum_{n\in \mathbb Z} \frac{1}{r^2} \left (r^2\frac{\partial^2}{\partial r^2} + r\frac{\partial}{\partial r}- n^2 + r^2  \right ) R_n(r) e^{in\theta} = 0 
\end{equation}
which implies for every $n$ 
\begin{equation}
    \left (r^2\frac{\partial^2}{\partial r^2} + r\frac{\partial}{\partial r}- n^2 + r^2  \right ) R_n = 0.
\end{equation}
The two fundamental solutions of this PDE are $J_n$ and $Y_n$, the Bessel function of order $n$ of the first and second type respectively. Hence there are coefficients $\gamma_n$, $\eta_n$ such that 
\begin{equation}
    R_n = \gamma_n J_n + \eta_n Y_n
\end{equation}
and $f$ is of the form 
\begin{equation}
    f(x) = \sum_{n\in \mathbb Z} \left ( \gamma_n J_n(r) + \eta_n Y_n(r) \right ) e^{in\theta}.
\end{equation}
The function $f$ is $\mathcal C^\infty$, and so is 
\begin{equation}
    \int_{[0,2\pi]} f(r, \theta) e^{im\theta}\,d\theta = R_m(r)  
\end{equation}
for every $m\in \mathbb Z$, hence $\eta_m=0$ for every $m$ and $f$ is of the form 
\begin{equation}
    f(x) = \sum_{n\in \mathbb Z} \gamma_n J_{|n|}(r) e^{in\theta}
\end{equation}
at least on the disc (note that $J_n$ and $J_{-n}$ solve the same PDE, and $J_{-n} = (-1)^n J_n$), thus on the whole plane. 
\end{proof}
Let $f$ be of the form (\ref{f_sol_plane}), then for every $r\ge 0$
\begin{equation}\label{meanf}
    \int_{[0,2\pi]} f(r,\theta)\,d\theta = \gamma_0 J_0(r),
\end{equation}
in particular if $\bar r$ is any positive zero of $J_0$, then the mean of $f$ over the circle of radius $\bar r$ is zero. Moreover, $f(0)=\gamma_0$ hence (\ref{meanf}) can be rewritten as 
\begin{equation}\label{meanf2}
    \int_{[0,2\pi]} f(r,\theta)\,d\theta = f(0) J_0(r).
\end{equation}
This identity can be thought as a modified mean value theorem valid for solutions of the Helmoltz equation. It entails immediately that on any disk of radius $r$ with $J_0(r)\le 2\pi$, there exists solutions $\Delta f+f=0$ such that their maximum value on the boundary is strictly smaller than the value at the center.